\documentclass[reqno,english]{amsart}

\usepackage{amsmath}    
\usepackage{amssymb}    
\usepackage[bbgreekl]{mathbbol}   

\usepackage[all]{xy}
\usepackage{graphics}

\usepackage{floatrow}   
\usepackage{afterpage}  

\usepackage{stmaryrd}   

\title{On a notion of homotopy Segal $E_\infty$-Hopf cooperad}

\author{Benoit Fresse}
\author{Lorenzo Guerra}

\address{Univ. Lille, CNRS, UMR 8524 -- Laboratoire Paul Painlev\'e, F-59000 Lille, France}
\email{Benoit.Fresse@univ-lille.fr}
\email{Lorenzo.Guerra@univ-lille.fr}


\date{November 22, 2020 (minor writing corrections on February 8, 2021)}

\DeclareFloatVCode{linebelow}{\vspace{.5em}\begin{center}\makebox[8cm]{\dotfill}\end{center}\vspace{.5em}}
\theoremstyle{plain}

\newtheorem{thm}[subsubsection]{Theorem}
\newtheorem{thm-defn}[subsubsection]{Theorem-Definition}
\newtheorem{prop}[subsubsection]{Proposition}
\newtheorem{prop-defn}[subsubsection]{Proposition-Definition}
\newtheorem{lemm}[subsubsection]{Lemma}
\newtheorem{cor}[subsubsection]{Corollary}

\theoremstyle{definition}

\newtheorem{defn}[subsubsection]{Definition}
\newtheorem{constr}[subsubsection]{Construction}
\newtheorem{constr-defn}[subsubsection]{Construction-Definition}
\newtheorem{remark}[subsubsection]{Remark}
\newtheorem{recoll}[subsubsection]{Recollections}

\DeclareMathOperator{\sh}{sh}
\DeclareMathOperator{\EM}{EM}
\DeclareMathOperator{\AW}{AW}
\DeclareMathOperator{\TR}{TR}

\DeclareMathOperator{\kk}{\mathbb{k}}   
\DeclareMathOperator{\FF}{\mathbb{F}}   
\DeclareMathOperator{\NN}{\mathbb{N}}   
\DeclareMathOperator{\ZZ}{\mathbb{Z}}   
\DeclareMathOperator{\mymod}{mod}       
\DeclareMathOperator{\eq}{eq}     

\DeclareMathOperator{\Tree}{\mathit{Tree}}
\DeclareMathOperator{\CubeCat}{\square} 
\DeclareMathOperator{\NCat}{\mathcal{N}}
\DeclareMathOperator{\Mod}{\mathit{Mod}}  
\DeclareMathOperator{\Set}{\mathit{Set}}  
\DeclareMathOperator{\EAlg}{\EOp\mathit{Alg}}

\DeclareMathOperator{\Hopf}{\mathit{Hopf}}    
\DeclareMathOperator{\SegOp}{\mathit{SegOp}}    
\DeclareMathOperator{\hSegOp}{\mathit{SegOp}}    

\DeclareMathOperator{\dg}{\mathit{dg}}      
\DeclareMathOperator{\cosimp}{\mathit{c}}   
\DeclareMathOperator{\simp}{\mathit{s}}     

\DeclareMathOperator{\Mor}{\mathtt{Mor}}    

\DeclareMathOperator{\id}{\mathit{id}}      
\DeclareMathOperator{\pt}{\mathit{pt}}      

\DeclareMathOperator{\DGSigma}{\mathtt{\Sigma}} 

\DeclareMathOperator{\DGB}{\mathtt{B}}          
\DeclareMathOperator{\DGK}{\mathtt{K}}
\DeclareMathOperator{\DGH}{\mathtt{H}}          
\DeclareMathOperator{\DGN}{\mathtt{N}}          
\DeclareMathOperator{\DGW}{\mathtt{W}}          
\DeclareMathOperator{\DGG}{\mathtt{G}}          

\DeclareMathAlphabet{\mathsfit}{OT1}{cmss}{m}{sl}   
\DeclareMathOperator{\AOp}{\mathsfit{A}}
\DeclareMathOperator{\BOp}{\mathsfit{B}}
\DeclareMathOperator{\COp}{\mathsfit{C}}
\DeclareMathOperator{\EOp}{\mathsfit{E}}        
\DeclareMathOperator{\FOp}{\mathsfit{F}}        

\DeclareMathOperator{\POp}{\mathsfit{P}}        
\DeclareMathOperator{\QOp}{\mathsfit{Q}}        

\DeclareMathOperator{\ROp}{\mathsfit{R}}

\DeclareMathOperator{\KOp}{\mathsfit{K}}        

\DeclareMathOperator{\ComOp}{\mathsfit{Com}}    
\DeclareMathOperator{\AsOp}{\mathsfit{As}}      

\DeclareMathOperator{\rset}{\underline{\mathsf{r}}}     

\DeclareMathOperator{\stree}{\underline{\mathsf{S}}}    
\DeclareMathOperator{\ttree}{\underline{\mathsf{T}}}    
\DeclareMathOperator{\utree}{\underline{\mathsf{U}}}    
\DeclareMathOperator{\ytree}{\underline{\mathsf{Y}}}    

\DeclareMathOperator{\sigmatree}{\underline{\mathsf{\Sigma}}}    
\DeclareMathOperator{\thetatree}{\underline{\mathsf{\Theta}}}    
\DeclareMathOperator{\gammatree}{\underline{\mathsf{\Gamma}}}    

\begin{document}

\begin{abstract}
We define a notion of homotopy Segal cooperad in the category of $E_\infty$-algebras. This model of Segal cooperad that we define in the paper,
which we call `homotopy Segal $E_\infty$-Hopf cooperad', covers examples given by the cochain complex of topological operads
and provides a framework for the study of the homotopy of such objects.

In a first step, we consider a category of Segal $E_\infty$-Hopf cooperads, which consist of collections of $E_\infty$-algebras indexed by trees
and equipped with coproduct operators, corresponding to tree morphisms, together with facet operators, corresponding to subtree inclusions.
The coproduct operators model coproducts of operations inside a tree. The facet operators are assumed to satisfy a Segal condition
when we take a partition of a tree into subtrees.
The homotopy Segal cooperads that we aim to define are formed by integrating homotopies in the composition schemes of the coproduct operators.
For this purpose, we replace the functorial structure that governs the composition of the coproduct operators
by the structure of a homotopy functor which we shape on a cubical enrichment of the category of $E_\infty$-algebras.
(We then use a particular model of the category of $E_\infty$-algebras which we associate to the chain Barratt--Eccles operad.)

We prove that every homotopy Segal $E_\infty$-Hopf cooperad in our sense is weakly-equivalent to a strict Segal $E_\infty$-Hopf cooperad.
We also define a notion of homotopy morphism of homotopy Segal $E_\infty$-Hopf cooperads.
We prove that every homotopy Segal $E_\infty$-Hopf cooperad admits a cobar construction
and that every homotopy morphism of homotopy Segal $E_\infty$-Hopf cooperads
induces a morphism on this cobar construction, so that our approach provides a lifting to the context of $E_\infty$-algebras
of classical homotopy cooperad structures that are modeled on the bar duality of operads
when we work in a category of differential graded modules.
\end{abstract}

\maketitle


\section*{Introduction}

\setcounter{section}{-1}

The goal of this paper is to provide an effective framework for the study of homotopy models of operads. Various models of $\infty$-operads in simplicial sets and in topological spaces
have been introduced in the literature.
The model that we propose in this paper relies on the Mandell model for the homotopy of spaces, which takes place in the category of $E_{\infty}$-algebras (see~\cite{Mandell,MandellIntegral}).
We construct our model for the homotopy of operads within the category of $E_{\infty}$-algebras used by Mandell.

By passing from spaces to $E_{\infty}$-algebras, we have to replace operads by cooperad structures, which are dual to operads in the categorical sense.
In a first step, we define a notion of \emph{strict} Segal $E_{\infty}$-Hopf cooperad, which is close to an $E_{\infty}$-algebra counterpart
of the Cisinski--Moerdijk notion of dendroidal space~\cite{CisinskiMoerdijkI,CisinskiMoerdijkII}.
In a second step, we define a notion of \emph{homotopy} Segal $E_{\infty}$-Hopf cooperad. The idea is to integrate homotopies in the composition schemes
that govern the structure of our objects. This notion of homotopy Segal $E_{\infty}$-Hopf cooperad
is the model that we aim to define and study in the paper.

If we forget about $E_{\infty}$-algebra structures and focus on operads and cooperads defined in a category of differential graded modules,
then we can use the bar duality theory to define notions of homotopy operads and of homotopy cooperads.
The bar duality approach enables authors to apply effective methods of perturbation theory (like the basic perturbation lemma)
for the study of homotopy operads and of homotopy cooperads.
We prove that every homotopy Segal $E_{\infty}$-Hopf cooperad admits a cobar construction, and hence, defines a homotopy cooperad in the classical sense.
We also define a notion of homotopy morphism of homotopy Segal $E_{\infty}$-Hopf cooperads and we prove that every homotopy morphism
induces a morphism on the cobar construction.
Hence, our notion of homotopy Segal $E_{\infty}$-Hopf cooperad provides a lift of the homotopy cooperads
that are defined in terms of the cobar construction when we forget about $E_{\infty}$-algebra structures
and work in a category of differential graded modules.

We implement these ideas as follows.
We work with the (chain) Barratt--Eccles operad, denoted by $\EOp$ hereafter, which defines an $E_\infty$-operad in the category of differential graded modules.
We take the category of algebras over the Barratt--Eccles operad as a model for the category of $E_\infty$-algebras in differential graded modules.
By the main result of~\cite{BergerFresse}, the normalized cochain complex of a simplicial set $\DGN^*(X)$ is endowed with an action of this operad.
The object $\DGN^*(X)$, equipped with this particular $E_\infty$-algebra structure, defines a representative of the Mandell model of the space $X$.

The Barratt--Eccles operad is endowed with a diagonal. This observation implies that the Barratt--Eccles operad
acts on tensor products, and therefore, that the category of algebras over the Barratt--Eccles operad
inherits a monoidal structure from the category of differential graded modules.
But this monoidal structure is only symmetric up to homotopy, because the diagonal of the Barratt--Eccles operad is only homotopy cocommutative.
For this reason, we can hardly define cooperads in the ordinary sense in the category of algebras over the Barratt--Eccles operad.
To work out this problem, a first idea is to define homotopy cooperads in terms of a functor on the category of trees,
which represent the composition schemes of operations in an operad.
We follow this idea to define the notion of a strict Segal $E_\infty$-Hopf cooperad.
We explicitly define a strict Segal $E_\infty$-Hopf cooperad as a functor from the category of trees to the category of $E_\infty$-algebras
equipped with facet operators that model subtree inclusions. The morphisms of $E_\infty$-algebras that we associate to the tree morphisms
model the composition structure of our objects. We will therefore refer to these morphisms as the coproduct operators.

We prove that every strict Segal $E_\infty$-Hopf cooperad is weakly-equivalent (quasi-isomorphic) to a strict cooperad in the ordinary sense
when we forget about $E_\infty$-algebra structures and transport our objects
to the category of differential graded modules.
We use a particular feature of the category of algebras over the Barratt--Eccles operad to simplify the definition of this forgetful functor: the coproduct of any collection of objects
in this category is weakly-equivalent to the tensor product (in general, in a category of algebras over an $E_\infty$-operad, such results are only valid for cofibrant objects).
We use a version of the Boardman--Vogt $W$-construction to establish this result.

We already mentioned that the notion of a strict Segal $E_\infty$-Hopf cooperad is close to Cisinski--Moerdijk's notion of a dendroidal space~\cite{CisinskiMoerdijkI,CisinskiMoerdijkII}.
(We also refer to~\cite{LeGrignou} for the definition of an analogous notion of homotopy operad in the differential graded module context).
The facet operators of our definition actually correspond to the outer facets of dendroidal spaces, while the coproduct operators correspond to the inner facets.
The main difference lies in the fact that we do not take a counterpart of operadic units and arity zero terms in our setting.
The paper~\cite{Ching}, about the bar duality of operads in spectra, also involves a notion of quasi-cooperad, which forms an analogue, in the category of spectra, of our strict cooperads.

To define a notion of a homotopy cooperad, an idea is to replace the category of trees by a resolution of this category (actually, a form of the Boardman--Vogt construction).
We then replace our functors on trees by homotopy functors in order to change the functoriality relation,
which models the associativity of the coproduct operators,
into a homotopy relation. The resolution of the category of trees has a cubical structure. We actually define our homotopy functor structure
by taking a cubical functor on this category, by using a cubical enrichment
of the category of $E_\infty$-algebras.
In the context of algebras over the Barratt--Eccles operad, this cubical enrichment can be defined by using tensor products $A\otimes I^k$, $k\geq 0$,
where $I^k$ represents the cellular cochain algebra of the $k$-dimensional cube $[0,1]^k$.
We equivalently have $I^k = \DGN^*(\Delta^1)^{\otimes k}$, where we consider the $k$-fold tensor product of the normalized cochain complex of the one-simplex $\Delta^1$.
The cubical functor structure that models the composition structure of our homotopy Segal cooperads
can then be defined explicitly, in terms of homotopy coproduct operators
associated to composable sequences of tree morphisms
and with values in tensor products with these cubical cochain algebras $I^k$, $k\geq 0$.
This is exactly what we do in the paper to get our model of homotopy Segal $E_\infty$-Hopf cooperads.
To carry out this construction, we crucially use that the objects $I^k$ are endowed with an action of the Barratt--Eccles operad and are equipped with compatible connection operators,
which we associate to certain degeneracies in the category of trees.
Note that in comparison to other constructions based on dendroidal objects (see for instance the study of homotopy operads in differential graded modules
of~\cite{LeGrignou}),
we keep strict associativity relations for the facet operators corresponding to subtree inclusions.
To conclude the paper, we also prove that every homotopy Segal $E_\infty$-Hopf cooperad is weakly-equivalent to
a strict Segal $E_\infty$-Hopf cooperad

Previously, we mentioned that we use a comparison between coproducts and tensor products
to define a forgetful functor from Segal $E_\infty$-Hopf cooperads
to Segal cooperads in differential graded modules.
However, the map that provides this comparison, which is an instance of an Alexander--Whitney diagonal, is not symmetric.
For this reason, we consider shuffle cooperads (in the sense of~\cite{DotsenkoKhoroshkin})
rather than symmetric cooperads
when we pass to Segal cooperads in differential graded modules.
Briefly recall that a shuffle (co)operad is a structure that retains the symmetries of the composition schemes of operations in a (co)operad,
but forgets about the internal symmetric structure of (co)operads.

The category of shuffle operads and the dual category of shuffle cooperads were introduced by Vladimir Dotsenko and Anton Khoroshkin in~\cite{DotsenkoKhoroshkin},
with motivations coming from the work of Eric Hoffbeck~\cite{Hoffbeck},
in order to define an operadic counterpart of the classical notion of a Gr\"obner basis.
This theory provides an effective approach for the study of the homotopy of the bar construction of operads
(in connection with the Koszul duality theory of Ginzburg--Kapranov~\cite{GinzburgKapranov}),
because one can observe that the (co)bar complex of a (co)operad
only depends on the shuffle (co)operad structure
of our object (when we forget the symmetric group actions).

\medskip
We give brief recollections on the tree structures and on the conventions on trees that we use all along this paper in a preliminary section.
We also briefly review the definition of cooperads in terms of functors defined on trees in this section.
We study the strict Segal cooperad model afterwards, in Section~\ref{sec:strict-segal-cooperads}.
Then we explain our definition of homotopy Segal $E_{\infty}$-Hopf cooperads. We address the study of this notion in Section~\ref{sec:homotopy-segal-cooperads}.
We devote an appendix to brief recollections on the definition of the Barratt--Eccles operad and to the proof of the crucial statements
on algebras over this operad that we use in the definition of homotopy Segal $E_{\infty}$-Hopf cooperads (the weak-equivalence
between coproducts and tensor products, and the compatibility of connections with the algebra structure
of cubical complexes).

\medskip
We work in a category of differential graded modules over an arbitrary ground ring $\kk$ all along this paper, where a differential graded module (a dg module for short)
generally denotes a $\kk$-module $M$ equipped with a decomposition of the form $M = \bigoplus_{*\in\ZZ} M_*$
and with a differential $\delta: M\rightarrow M$ such that $\delta(M_*)\subset M_{*-1}$.
We therefore assume that our dg modules are equipped with a lower grading in general, but we may also consider dg modules
that are naturally equipped with an upper grading $M = \bigoplus_{*\in\ZZ} M^*$
and with a differential such that $\delta(M^*)\subset M^{*+1}$.
We then use the classical equivalence $M_* = M^{-*}$ to convert the upper grading on $M$
into a lower grading.
We equip the category of dg modules with the standard symmetric monoidal structure, given by the tensor product of dg modules,
with a sign in the definition of the symmetry operator that reflects the usual commutation rule
of differential graded algebra.
We call weak-equivalences the quasi-isomorphisms of dg modules and we transfer this class of weak-equivalences
to every category of structured objects (algebras, cooperads)
that we may form within the category of dg modules.

We take the category of algebras over the chain Barratt--Eccles operad, denoted by $\EOp$, as a model for the category of $E_\infty$-algebras in dg modules.
We use the notation $\EAlg$ for this category of dg algebras. We also adopt the notation $\vee$ for the coproduct in $\EAlg$.
We refer to the appendix~(\S\ref{sec:Barratt-Eccles-operad}) for detailed recollections on the definition and properties of the chain Barratt--Eccles operad,
and for a study of the properties of the coproduct, notably the existence of a weak-equivalence $\EM: A\otimes B\xrightarrow{\sim}A\vee B$
that we use in our constructions.
We also use that the normalized cochain complex of a simplicial set $\DGN^*(X)$ inherits the structure of an $\EOp$-algebra.
We refer to~\cite{BergerFresse} for the definition of this $\EOp$-algebra structure. (We give brief recollections on this subject in~\S\ref{sec:Barratt-Eccles-operad}.)
We also adopt the notation $\Sigma_r$ for the symmetric group on $r$ letters
all along the paper.

\thanks{The authors acknowledge support from the Labex CEMPI (ANR-11-LABX-0007-01) and from the FNS-ANR project OCHoTop (ANR-18CE93-0002-01).}

\section{Background}\label{section:background}

The first purpose of this section is to briefly explain the conventions on trees that we use all along the paper.
We also briefly recall the definition of the notion of a shuffle cooperad, which is intermediate between the notion of a non-symmetric cooperad
and the notion of a symmetric cooperad.
The idea of shuffle cooperads, as we already explained in the introduction of the paper, is to retain the symmetries of the composition schemes of cooperads (based on trees),
but to forget about the internal symmetric structure
of our objects.
This construction is possible for cooperads with no term in arity zero, because the trees
that we consider in this case
have a natural planar embedding (and as such, have trivial automorphism groups).

In what follows, we mainly consider shuffle cooperads (rather than shuffle operads). Our main interest for this notion lies in the observation
that the category shuffle cooperads is endowed with a cobar complex functor, which is the same as the cobar complex functor on the category of cooperads
when we forget about the internal symmetric structure
of our objects.

\subsection{Recollections and conventions on the categories of trees}\label{subsection:trees}
In general, we follow the conventions of the book~\cite[Appendix A]{FresseBook} for the definitions that concern the categories of trees.

To summarize, we consider the categories, denoted by $\Tree(r)$ in~\cite[\S A.1]{FresseBook}, whose objects are trees $\ttree$
with $r$ ingoing edges $e_1,\dots,e_r$ numbered from $1$ to $r$ (the leaves), one outgoing edge $e_0$ (the root),
and where each inner edge $e$ is oriented from a source vertex $s(e)$ to a target vertex $t(e)$
so that each vertex $v$ has one outgoing edge and at least one ingoing edge.
The morphisms of $\Tree(r)$ are composed of isomorphisms and of edge contractions,
where we assume that the isomorphisms preserve the numbering of the ingoing edges.
The assumption that each vertex in a tree has at least one ingoing edge implies that our trees have a trivial automorphism group (see~\cite[\S A.1.8]{FresseBook}).
In what follows, we use this observation to simplify our constructions.
Namely, we can forget about isomorphisms by picking a representative in each isomorphism class of trees
and we follow this convention all along the paper.

The set of vertices of a tree $\ttree\in\Tree(r)$ is denoted by $V(\ttree)$
whereas the set of edges is denoted by $E(\ttree)$.
The set of inner edges (the edges which are neither a leaf nor a root) is denoted by $\mathring{E}(\ttree)$.
For a vertex $v\in V(\ttree)$, we use the notation $\rset_v$
for the set of edges $e$ such that $t(e) = v$.

The symmetric group $\Sigma_r$ acts on the category of $r$-trees $\Tree(r)$ by renumbering the ingoing edges.
In what follows, we also consider a version of the category of trees $\Tree(\rset)$ where the ingoing edges of the trees
are indexed by an arbitrary finite set $\rset$,
and not necessarily by an ordinal, so that the mapping $\rset\mapsto\Tree(\rset)$
defines a functor from the category of finite sets and bijections between them
to the category of small categories.
The categories of trees are endowed with composition operations $\circ_i: \Tree(k)\times\Tree(l)\rightarrow\Tree(k+l-1)$,
which provide the collection $\Tree(r)$, $r>0$,
with the structure of an operad in the category of categories.
These composition operations have a unit, the $1$-tree $\downarrow\in\Tree(1)$ with a single edge which is both the root and a leave,
but we put this tree aside actually and we do not consider it in our forthcoming constructions.

Recall that we call $r$-corolla the $r$-tree $\ytree\in\Tree(r)$ with a single vertex $v$, one outgoing edge with this vertex $v$ as a source,
and $r$ ingoing edges targeting to $v$. We also consider corollas $\ytree\in\Tree(\rset)$
whose sets of ingoing edges are indexed by arbitrary finite sets $\rset$.
To each vertex $v$ in a tree $\ttree$, we can associate an $\rset_v$-corolla $\ytree_v\subset\ttree$,
with $v$ as vertex, the ingoing edges of this vertex in $\ttree$
as set of ingoing edges, and the outgoing edge
of $v$ as root.
The existence of a tree morphism $f: \ttree\rightarrow\stree$ is equivalent to the existence of a decomposition $\ttree = \lambda_{\stree}(\sigmatree_v,v\in V(\stree))$,
where $\lambda_{\stree}$ denotes a treewise operadic composition operation, shaped on the tree $\stree$,
of subtrees $\sigmatree_v\subset\ttree$, $v\in V(\stree)$,
that represent the pre-image of the corollas $\ytree_v\subset\stree$, $v\in V(\stree)$,
under our morphism.
For tree with two vertices $\gammatree$, the composition $\ttree = \lambda_{\gammatree}(\sigmatree_u,\sigmatree_v)$
is equivalent to an operadic composition operation $\ttree = \sh_*(\sigmatree_u\circ_i\sigmatree_v)$,
where $\sh_*$ denotes the action of a permutation,
associated to a partition of the form $\{1<\dots<r\} = \{i_1<\dots<\widehat{i_p}<\dots<i_k\}\amalg\{j_1<\dots<j_l\}$,
which reflects the indexing of the leaves in the tree $\gammatree$.
The index $i_p$ is a dummy composition variable which we associate to the inner edge of this tree $\gammatree$.
We can insert this dummy variable at the position such that $i_{p-1}<j_1<i_{p+1}$ inside the ordered set $\{i_1<\dots<\widehat{i_p}<\dots<i_k\}\subset\{1<\dots<r\}$
in order to work out the symmetries of this operation (we go back to this topic in the next paragraph).

Each $r$-tree $\ttree\in\Tree(r)$ has a natural planar embedding, which we determine as follows.
Let $\{e_{\alpha},\alpha\in\rset_v\}$ be the set of ingoing edges of a vertex $v\in V(\ttree)$.
Each edge $e_{\alpha}$ can be connected to a leaf $e_{i_{\alpha}}$
through a chain of edges $e_{\alpha} = e_{\alpha_0}$, $e_{\alpha_1}$, \dots, $e_{\alpha_l} = e_{i_{\alpha}}$,
such that $t(e_{\alpha_k}) = s(e_{\alpha_{k-1}})$, for $k = 1,\dots,l$.
Let $m_{\alpha}\in\{1,\dots,r\}$ be the minimum of the indices $i_{\alpha}$ of these leaves $e_{i_{\alpha}}$
that lie over the edge $e_{\alpha}$
in the tree $\ttree$. We order the set of ingoing edges $e_{\alpha}$, $\alpha\in\rset_v$, of our vertex $v$ by taking $e_{\alpha}<e_{\beta}$
when $m_{\alpha}<m_{\beta}$. We perform this ordering of the set of ingoing edges for each vertex $v\in V(\ttree)$
to get the planar embedding of our tree. We have an obvious generalization of this result for the trees $\ttree\in\Tree(\rset)$
whose ingoing edges are indexed by a set $\rset$
equipped with a total ordering.
The existence of this natural planar embedding reflects the fact that the automorphism group of any object $\ttree$
is trivial in our categories of trees $\Tree(\rset)$.

\subsection{Recollections on the treewise definition of cooperads and of shuffle cooperads}\label{subsection:shuffle-cooperads}
Throughout the paper, we mainly use a definition of cooperads in terms of collections endowed with treewise composition coproducts.
We refer to \cite[Appendix C]{FresseBook} or \cite[Section 5.6]{LodayValletteBook}, for instance, for a detailed account of this combinatorial approach to the definition of a cooperad.

In the paper, we more precisely use that a cooperad $\COp$ is equivalent to a collection of contravariant functors on the categories of trees $\COp: \ttree\mapsto\COp(\ttree)$
such that $\COp(\ttree) = \bigotimes_{v\in V(\ttree)}\COp(\ytree_v)$,
for all $\ttree\in\Tree(r)$
and where the morphism $\rho_{\ttree\rightarrow\stree}: \COp(\stree)\rightarrow\COp(\ttree)$ induced by a tree morphism $f: \ttree\rightarrow\stree$
also admits a decomposition of the form $\rho_{\ttree\rightarrow\stree} = \bigotimes_{v\in V(\stree)}\rho_{\sigmatree_v\rightarrow\ytree_v}$
when we use the relation $\ttree\simeq\lambda_{\stree}(\sigmatree_v,v\in V(\stree))$.
In the standard definition, the definition of the structure of a cooperad is rather expressed in terms of the coproduct operations $\rho_{\sigmatree_v\rightarrow\ytree_v}$
which generate the general operators $\rho_{\ttree\rightarrow\stree}$
associated to the tree morphisms $f: \ttree\rightarrow\stree$. (We review this reduction of the definition later on in this paragraph.)
The consideration of general coproduct operators $\rho_{\ttree\rightarrow\stree}$ in the definition of a cooperad
is motivated by the definition of the category of Segal cooperads
in the next section.

In the definition of a symmetric cooperad, we assume, besides, that the symmetric group $\Sigma_r$ acts on the collection $\COp(\ttree)$
in the sense that a natural transformation $s: \COp(\ttree)\rightarrow\COp(s\ttree)$, $\ttree\in\Tree(r)$,
is associated to each permutation $s\in\Sigma_r$,
where $\ttree\mapsto s\ttree$
denotes the action of this permutation on the category of trees $\Tree(r)$.
Then we require that the decomposition $\COp(\ttree) = \bigotimes_{v\in V(\ttree)}\COp(\ytree_v)$ is, in some natural sense, preserved by the action of the symmetric groups.
Note that we can again extend the definition of the functor underlying a cooperad $\COp$ to the categories of trees $\Tree(\rset)$ whose leaves are indexed by arbitrary finite sets $\rset$.
We then consider an action of the bijections of finite sets to extend the action of the permutations on ordinals.

Recall that $\ytree_v$ denotes the corolla generated by a vertex $v$ in a tree $\ttree$.
The object $\COp(\ytree)$ associated to a corolla $\ytree$ only depends on the number of leaves of the corolla (since all corollas with $r$ leaves are canonically isomorphic).
Thus the decomposition relations of the above definition imply that our functor is fully determined by a sequence of objects $\COp(r)$, $r>0$,
equipped with an action of the symmetric groups $\Sigma_r$, so that $\COp(\ytree) = \COp(r)$ for a corolla with $r$-leaves,
together with coproduct operations $\COp(r) = \COp(\ytree)\rightarrow\COp(\ttree)$,
which we associate to the tree morphisms with values in a corolla $\ttree\rightarrow\ytree$.
Furthermore, these coproduct operations can be generated by coproduct operations with values in a term $\COp(\gammatree)$
such that $\gammatree$ is a tree with two vertices,
because every tree morphism $\ttree\rightarrow\stree$ can be decomposed into a sequence of edge contractions,
which are equivalent to the application of tree morphisms of the form $\gammatree\rightarrow\ytree$
inside the tree $\ttree$.

In this equivalence, we can still consider a collection $\COp(\rset)$ indexed by arbitrary finite sets $\rset$
and take $\COp(\ytree) = \COp(\rset)$ for a corolla $\ytree$ whose set of leaves
is indexed by a finite set $\rset$.
Note that such a consideration is necessary in the expression of the decomposition $\COp(\ttree) = \bigotimes_{v\in V(\ttree)}\COp(\ytree_v)$,
because we then take an arbitrary set to index the edges of the corollas $\ytree_v$,
which correspond to the ingoing edges of the vertices of our tree (but we go back to this observation in the next paragraph).
If we unravel the construction, then we get that the $2$-fold coproducts $\COp(\rset) = \COp(\ytree)\rightarrow\COp(\gammatree)$ are equivalent to coproduct operations
of the form $\circ_{i_p}^*: \COp(\rset)\rightarrow\COp(\rset_u)\otimes\COp(\rset_v)$
with $\rset_u = \{i_1,\dots,i_p,\dots,i_k\}$
and $\rset_v = \{j_1,\dots,j_l\}$ such that $\{i_1,\dots,\widehat{i_p},\dots,i_k\}\amalg\{j_1,\dots,j_l\} = \rset$.
(We then retrieve the dual of the classical partial product operations associated to an operad.)

Recall that we assume by convention that the vertices of our trees have at least one ingoing edge.
(For this reason, we assume that the sequence of objects $\COp(r)$, which underlies an operad, is indexed by positive integers $r>0$.)
This convention enables us to order the ingoing edges of each vertex in a tree whose leaves are indexed by an ordinal $\rset = \{1<\dots<r\}$,
and as a consequence,
to get rid of the actions of the symmetric groups
in the construction
of the tensor product $\COp(\ttree) = \bigotimes_{v\in V(\ttree)}\COp(\ytree_v)$ (since we can use such a canonical ordering of the edges of the corolla $\ytree_v$
to fix a bijection between the indexing set of this set of edges $\rset_v$
and an ordinal $\{1<\dots<r_v\}$).
The idea, explained in~\cite{FressePartitions}, is to order the ingoing edges according to the minimum
of the index of the leaves of the subtree
that lie over each edge.

This observation is used to define the notion of a shuffle cooperad. Indeed, a shuffle cooperad explicitly consists of a collection of contravariant functors
on the categories of trees $\COp: \ttree\mapsto\COp(\ttree)$ with the same operations and structure properties
as the classical symmetric cooperads, but where we forget about the actions
of permutations. If we express the definition in terms of $2$-fold coproducts, then we get that a shuffle cooperad consists of a collection $\COp(r)$, $r>0$,
equipped with coproducts $\circ_{i_p}^*: \COp(\{1<\dots<r\})\rightarrow\COp(\{i_1<\dots<i_p<\dots,i_k\})\otimes\COp(\{j_1<\dots<j_l\})$
associated to partitions $\{i_1<\dots<\widehat{i_p}<\dots<i_k\}\amalg\{j_1<\dots<j_l\} = \{1<\dots<r\}$
such that $i_{p-1}<j_1<i_{p+1}$ (see~\cite{DotsenkoKhoroshkin}).
These partitions are equivalent to the pointed shuffles of~\cite{Hoffbeck}.

Note that we can extend the ordering of the ingoing edges of each vertex in a tree to an ordering of the vertices themselves.
We use this observation in our definition of the forgetful functor from the category of Segal cooperads in $E_\infty$-algebras
to the category of Segal cooperads in dg modules.

\subsection{Counits, connected cooperads and local conilpotence}\label{subsection:conilpotence}
In the standard definition of a cooperad, we assume that the coproducts $\circ_{i_p}^*$ satisfy natural counit relations
with respect to a counit morphism which we associate to our objects,
but we forget about this counit morphism and about the counit conditions in the definition
of the previous paragraph.
The cooperads that we consider are actually equivalent to the coaugmentation coideal of coaugmented cooperads.
If we start with the standard definition of a cooperad (where we have a counit), then we should take components of the coaugmentation coideal of our cooperad $\COp$
in the definition of the treewise tensor products $\COp(\ttree) = \bigotimes_{v\in V(\ttree)}\COp(\ytree_v)$
and of the treewise coproducts $\rho_{\ttree\rightarrow\stree}: \COp(\stree)\rightarrow\COp(\ttree)$.

In the definition of cooperads, one often has to assume the validity of a local conilpotence condition.
In the treewise formalism, this local conilpotence condition asserts that for every element $x\in\COp(\stree)$
in the component of a cooperad $\COp$ associated to a tree $\stree\in\Tree(r)$, we can pick a non-negative integer $N_x\in\NN$
such that $\sharp V(\ttree)\geq N_x\Rightarrow\rho_{\ttree\rightarrow\stree}(x) = 0$,
for every tree $\ttree\in\Tree(r)$.
This condition ensures that the map $\rho: \COp(\stree)\rightarrow\prod_{\ttree\rightarrow\stree}\COp(\ttree)$
induced by the collection of all coproducts $\rho_{\ttree\rightarrow\stree}: \COp(\stree)\rightarrow\COp(\ttree)$
factors through the sum $\bigoplus_{\ttree\rightarrow\stree}\COp(\ttree)$.

In what follows, we may actually need a stronger connectedness condition, which we define by requiring that the components of our object $\COp(\ttree)$
vanish when the tree $\ttree$ contains at least one vertex with a single ingoing edge. For an ordinary cooperad, this requirement is equivalent to the relation $\COp(1) = 0$.
In general, this connectedness condition implies that our object $\COp$ reduces to a structure given by a collection of functors $\ttree\mapsto\COp(\ttree)$
on the subcategories $\widetilde{\Tree}(r)\subset\Tree(r)$ formed by trees $\ttree$
where all the vertices have at least two ingoing edges (in~\cite[\S A.1.12]{FresseBook}
the terminology `reduced tree' is used for this subcategory of trees). The conilpotence of the cooperad $\COp$
then follows from the observation that, for any given tree $\stree\in\widetilde{\Tree}(r)$,
we have finitely many morphisms such that $\ttree\rightarrow\stree$
in $\widetilde{\Tree}(r)$.
We say that a cooperad is connected when it satisfies this connectedness requirement $\COp(1) = 0$, or equivalently, when $\ttree\not\in\widetilde{\Tree}(r)\Rightarrow\COp(\ttree) = 0$.

In what follows, we will similarly say that a Segal cooperad $\COp$ is connected when it satisfies the same treewise condition $\ttree\not\in\widetilde{\Tree}(r)\Rightarrow\COp(\ttree) = 0$.
We mainly use the local conilpotence and the connectedness condition in our study of the $W$-construction of (Segal) shuffle dg cooperads
and in our definition of the cobar construction for homotopy Segal shuffle dg cooperads.

\section{The category of strict Segal $E_\infty$-Hopf cooperads}\label{sec:strict-segal-cooperads}

We study the category of Segal $E_\infty$-Hopf cooperads in this section.
We devote our first subsection to the definition of this category.
We then study strict Segal dg cooperads, which are structures, defined within the category of dg modules, which we obtain by forgetting the $E_\infty$-algebra structures
attached to the definition of a Segal $E_\infty$-Hopf cooperad.
We also explain the definition of an equivalence between our Segal dg cooperads and ordinary dg cooperads.
We devote the second subsection of this section to these topics.
We study the cobar complex of Segal dg cooperads afterwards, in a third subsection.
We eventually explain a correspondence between operads in simplicial sets and Segal $E_\infty$-Hopf cooperads.
We prove that we can retrieve a completion of operads in simplicial sets
from a corresponding Segal $E_\infty$-Hopf cooperad.
We devote the fourth subsection of the section to this subject.

Recall that we use the notation $\EOp$ for the chain Barratt--Eccles operad and that $\EAlg$ denotes the category of algebras in dg modules associated to this operad.

\subsection{The definition of strict Segal $E_\infty$-Hopf cooperads}\label{subsec:strict-segal-cooperads}

We begin our study by defining the objects of our category of Segal $E_\infty$-Hopf cooperads.
We actually define beforehand a notion of Segal $E_\infty$-Hopf pre-cooperad, which consists of objects equipped with all the operations
that underlie the structure of a Segal $E_\infty$-Hopf cooperads (coproducts and facet operators),
and then we just define a Segal $E_\infty$-Hopf cooperad as a Segal $E_\infty$-Hopf pre-cooperad whose facet operators
satisfy an extra homotopy equivalence condition (the Segal condition).
We make these definitions explicit in the first paragraph of this subsection.
We explain the definition of morphisms of Segal $E_\infty$-Hopf cooperads afterwards in order to complete the objectives of this subsection.
We are guided by the combinatorial definition of cooperads in terms of trees, which we briefly recalled in the overview of~\S\ref{subsection:shuffle-cooperads}.

\begin{defn}\label{definition:strict-E-infinity-cooperad}
We call (strict) Segal $E_\infty$-Hopf shuffle pre-cooperad the structure defined by a collection of $\EOp$-algebras
\begin{equation*}
\AOp(\ttree)\in\EAlg,\quad\text{$\ttree\in\Tree(r)$, $r>0$},
\end{equation*}
equipped with
\begin{itemize}
\item
coproduct operators
\begin{equation*}
\rho_{f: \ttree\rightarrow\stree}: \AOp(\stree)\rightarrow\AOp(\ttree),
\end{equation*}
defined as morphisms of $\EOp$-algebras, for all tree morphisms $f: \ttree\rightarrow\stree$,
and which satisfy the following standard functoriality constraints $\rho_{\stree\xrightarrow{=}\stree} = \id_{\AOp(\stree)}$
and $\rho_{\ttree\rightarrow\utree}\circ\rho_{\utree\rightarrow\stree} = \rho_{\ttree\rightarrow\stree}$,
for all pairs of composable tree morphisms $\ttree\rightarrow\utree\rightarrow\stree$,
\item
together with facet operators
\begin{equation*}
i_{\sigmatree,\stree}: \AOp(\sigmatree)\rightarrow\AOp(\stree),
\end{equation*}
also defined as morphisms of $\EOp$-algebras, for all subtree embeddings $\sigmatree\subset\stree$,
and which satisfy the following functoriality relations $i_{\stree,\stree} = \id_{\stree}$ and $i_{\thetatree,\stree}\circ i_{\sigmatree,\thetatree} = i_{\sigmatree,\stree}$,
for all $\sigmatree\subset\thetatree\subset\stree$.
\item
We also assume the verification of a compatibility relation between the facet operators and the coproduct operators.
We express this compatibility relation by the commutativity of the following diagram:
\begin{equation*}
\xymatrixcolsep{5pc}\xymatrix{ A(\stree)\ar[r]^-{\rho_f} & \AOp(\ttree) \\
A(\sigmatree)\ar[u]^{i_{\sigmatree,\stree}}\ar@{.>}[r]^-{\rho_{f|_{f^{-1}\sigmatree}}} &
\AOp(f^{-1}\sigmatree), \ar@{.>}[u]_{i_{f^{-1}\sigmatree,\ttree}} }
\end{equation*}
for all $f: \ttree\rightarrow\stree$ and $\sigmatree\subset\stree$, where $f^{-1}(\sigmatree)\subset\ttree$
denotes the subtree such that $V(f^{-1}\sigmatree) = f^{-1}V(\stree)$
and we consider the obvious restricted morphism $f|_{f^{-1}\sigmatree}: f^{-1}\sigmatree\rightarrow\sigmatree$.
\end{itemize}
We say that a Segal $E_\infty$-Hopf shuffle pre-cooperad $\AOp$ is a Segal $E_\infty$-Hopf shuffle cooperad when it satisfies the following extra condition (the Segal condition):
\begin{enumerate}
\item[(*)]
The facet operators $i_{\sigmatree_v,\ttree}: \AOp(\sigmatree_v)\rightarrow\AOp(\ttree)$
associated to a tree decomposition $\ttree = \lambda_{\stree}(\sigmatree_v, v\in V(\stree))$
induce a weak-equivalence
\begin{equation*}
i_{\lambda_{\stree}(\sigmatree_*)}: \bigvee_{v\in V(\stree)}\AOp(\sigmatree_v)\xrightarrow{\sim}\AOp(\ttree)
\end{equation*}
when we pass to the coproduct of the objects $\AOp(\sigmatree_v)$
in the category of $\EOp$-algebras. We refer to this weak-equivalence $i_{\lambda_{\stree}(\sigmatree_*)}$ as the Segal map
associated to the decomposition $\ttree = \lambda_{\stree}(\sigmatree_v, v\in V(\stree))$.
\end{enumerate}
We finally define a Segal $E_\infty$-Hopf symmetric (pre-)cooperad as a Segal $E_\infty$-Hopf shuffle (pre-)cooperad $\AOp$
equipped with an action of the permutations such that $s^*: \AOp(s\ttree)\rightarrow\AOp(\ttree)$, for $s\in\Sigma_r$ and $\ttree\in\Tree(r)$,
and which intertwine the facets and the coproduct operators attached our object.
\end{defn}

We have the following statement, which enables us to reduce the verification of the Segal condition to particular tree decompositions.

\begin{prop}\label{proposition:Segal-condition}
For any Segal $E_\infty$-Hopf shuffle pre-cooperad $\AOp$, we have an equivalence between the following statements:
\begin{enumerate}
\item The Segal condition holds for all tree decompositions $\ttree = \lambda_{\stree}(\sigmatree_v,v\in V(\stree))$.
\item The Segal condition holds for all tree decompositions of the form $\ttree = \lambda_{\gammatree}(\sigmatree_u,\sigmatree_v) = \sigmatree_u\circ_i\sigmatree_v$,
where we take an operadic composition along a tree with two vertices $\gammatree$ equivalent to the performance of an operadic composition product $\sigmatree_u\circ_i\sigmatree_v$
of a pair of trees $\sigmatree_u,\sigmatree_v\subset\ttree$ (we abusively omit the action of the shuffle permutation
that we associate to general composition operations of this form, see~\S\ref{subsection:trees}).
\item The Segal condition holds for all decompositions of trees into corollas $\ttree = \lambda_{\ttree}(\ytree_v,v\in V(\ttree))$.\qed
\end{enumerate}
\end{prop}

We now define morphisms of strict Segal $E_\infty$-Hopf cooperads.

\begin{defn}\label{definition:E-infinity-cooperad-morphism}
A morphism of Segal $E_\infty$-Hopf shuffle (pre-)cooperads $\phi: \AOp\rightarrow\BOp$ is a collection of $\EOp$-algebra morphisms $\phi_{\ttree}: \AOp(\ttree)\rightarrow\BOp(\ttree)$, $\ttree\in\Tree(r)$, $r>0$,
which preserve the action of the facets and coproduct operators on our objects
in the sense that:
\begin{enumerate}
\item
the diagram
\begin{equation*}
\xymatrix{ \AOp(\stree)\ar[r]^{\phi_{\stree}}\ar[d]_{\rho_{\ttree\rightarrow\stree}} & \BOp(\stree)\ar[d]^{\rho_{\ttree\rightarrow\stree}} \\
\AOp(\ttree)\ar[r]^{\phi_{\ttree}} & \BOp(\ttree) }
\end{equation*}
commutes for all tree morphisms $\ttree\rightarrow\stree$,
\item
the diagram
\begin{equation*}
\xymatrix{ \AOp(\stree)\ar[r]^{\phi_{\stree}} & \BOp(\stree) \\
\AOp(\sigmatree)\ar[r]^{\phi_{\sigmatree}}\ar[u]^{i_{\sigmatree,\stree}} &
\BOp(\sigmatree)\ar[u]_{i_{\sigmatree,\stree}} }
\end{equation*}
commutes for all subtree embeddings $\sigmatree\subset\stree$.
\end{enumerate}
If $\AOp$ and $\BOp$ are Segal $E_\infty$-Hopf symmetric (pre-)cooperads, then $\phi: \AOp\rightarrow\BOp$ is a morphism of Segal $E_\infty$-Hopf symmetric (pre-)cooperads
when $\phi$ preserves the action of permutations on our objects in the sense that
\begin{enumerate}\setcounter{enumi}{2}
\item
the diagram
\begin{equation*}
\xymatrix{ \AOp(s\ttree)\ar[r]^{\phi_{s\ttree}}\ar[d]_{s^*} & \BOp(s\ttree)\ar[d]^{s^*} \\
\AOp(\ttree)\ar[r]^{\phi_{\ttree}} & \BOp(\ttree) }
\end{equation*}
commutes, for all $s\in\Sigma_r$ and $\ttree\in\Tree(r)$.
\end{enumerate}
\end{defn}

The morphisms of Segal $E_\infty$-Hopf shuffle (pre-)cooperads can obviously be composed,
as well as the morphisms of Segal $E_\infty$-Hopf symmetric (pre-)cooperads,
so that we can form a category of Segal $E_\infty$-Hopf shuffle (pre-)cooperads and a category of Segal $E_\infty$-Hopf symmetric (pre-)cooperads.
In what follows, we adopt the notation $\EOp\Hopf\sh\SegOp^c$ for the category of Segal $E_\infty$-Hopf shuffle cooperads
and the notation $\EOp\Hopf\Sigma\SegOp^c$ for the category of Segal $E_\infty$-Hopf symmetric cooperads.

\subsection{The forgetting of $E_\infty$-structures}\label{subsection:forgetful-strict}

To any Segal $E_\infty$-Hopf cooperad $\AOp$, we can associate a Segal cooperad in dg modules by forgetting the $E_\infty$-algebra structure attached to each object $\AOp(\ttree)$.
We examine this construction in this subsection.
We need to assume that the vertices of our trees are totally ordered in order to make the construction
of the forgetful functor from Segal $E_\infty$-Hopf cooperads to Segal dg cooperads
work.
For this reason, we restrict ourselves to Segal shuffle cooperads all along this subsection, though our definition of Segal dg cooperad
makes sense in the symmetric setting.

\begin{defn}\label{definition:tree-shaped-cooperad}
We call Segal shuffle dg pre-cooperad the structure defined by a collection of dg modules
\begin{equation*}
\AOp(\ttree)\in\dg\Mod,\quad\text{$\ttree\in\Tree(r)$, $r>0$},
\end{equation*}
equipped with
\begin{itemize}
\item
coproduct operators
\begin{equation*}
\rho_{f: \ttree\rightarrow\stree}: \AOp(\stree)\rightarrow\AOp(\ttree),
\end{equation*}
defined as morphisms of dg modules, for all tree morphisms $f: \ttree\rightarrow\stree$, and which satisfy the same standard functoriality constraints
as in the case of Segal $E_\infty$-Hopf shuffle cooperads,
\item
together with Segal maps
\begin{equation*}
i_{\lambda_{\stree}(\sigmatree_*)}: \bigotimes_{v\in V(\stree)}\AOp(\sigmatree_v)\rightarrow\AOp(\ttree),
\end{equation*}
defined as morphisms of dg modules, for all tree decompositions $\ttree = \lambda_{\stree}(\sigmatree_v,v\in V(\stree))$,
and such that
for the trivial decomposition $\ttree = \lambda_{\ytree}(\ttree)$, we have $i_{\lambda_{\ytree}(\ttree)} = \id_{\AOp(\ttree)}$,
while for nested decompositions $\ttree = \lambda_{\utree}(\thetatree_u,u\in V(\utree))$ and $\thetatree_u = \lambda_{\stree_u}(\sigmatree_v,v\in V(\stree_u))$, $u\in V(\utree)$,
we have
\begin{equation*}
i_{\lambda_{\utree}(\thetatree_*)}\circ(\bigotimes_{u\in V(\utree)}i_{\lambda_{\stree_u}(\sigmatree_*)}) = i_{\lambda_{\stree}(\sigmatree_*)},
\end{equation*}
where we consider the composite decomposition $\ttree = \lambda_{\stree}(\sigmatree_v,v\in V(\stree))$
with $\stree = \lambda_{\utree}(\stree_u,u\in V(\utree))$.
\item
We still assume the verification of a compatibility relation between the Segal maps and the coproduct operators.
We express this dg module version of the compatibility relation by the commutativity of the following diagram:
\begin{equation*}
\xymatrixcolsep{7pc}\xymatrix{ \AOp(\stree)\ar[r]^-{\rho_f} & \AOp(\ttree) \\
\bigotimes_{u\in V(\utree)}\AOp(\sigmatree_u)\ar[r]_-{\bigotimes_{u\in V(\utree)}\rho_{f|_{f^{-1}\sigmatree_u}}}
\ar[u]^{i_{\lambda_{\utree}(\sigmatree_*)}} &
\bigotimes_{u\in V(\utree)}\AOp(f^{-1}\sigmatree_u)
\ar[u]_{i_{\lambda_{\utree}(f^{-1}\sigmatree_*)}} }
\end{equation*}
for all tree morphisms $f: \ttree\rightarrow\stree$ and decompositions $\stree = \lambda_{\utree}(\sigmatree_v,v\in V(\utree))$,
where we consider the pre-image $f^{-1}\sigmatree_v\subset\ttree$
of the subtrees $\sigmatree_v\subset\stree$.
\end{itemize}
We then say that a Segal shuffle dg pre-cooperad $\AOp$ is a Segal shuffle dg cooperad when the following Segal condition holds:
\begin{enumerate}
\item[(*)]
The Segal map $i_{\lambda_{\stree}(\sigmatree_*)}$ is a weak-equivalence
\begin{equation*}
i_{\lambda_{\stree}(\sigmatree_*)}: \bigotimes_{v\in V(\stree)}\AOp(\sigmatree_v)\xrightarrow{\sim}\AOp(\ttree),
\end{equation*}
for every decomposition $\ttree = \lambda_{\stree}(\sigmatree_v, v\in V(\stree))$.
\end{enumerate}
We still define a morphism of Segal shuffle dg (pre-)cooperads $\phi: \AOp\rightarrow\BOp$
as a collection of dg module morphisms $\phi_{\ttree}: \AOp(\ttree)\rightarrow\BOp(\ttree) $
that preserve the coproduct operators $\rho_{\ttree\rightarrow\stree}$
and the Segal maps $i_{\lambda_{\stree}(\sigmatree_*)}$
in the obvious sense.
We use the notation $\dg\sh\SegOp^c$ for the category of Segal shuffle dg cooperads, which we equip with this notion of morphism.
\end{defn}

The forgetful functor from the category of strict Segal $E_\infty$-Hopf shuffle cooperads to the category of Segal shuffle dg cooperads
essentially ignores the $E_\infty$-structures
attached to our objects.
However, Definition~\ref{definition:strict-E-infinity-cooperad} uses the coproduct of $\EOp$-algebras $\vee$,
while Definition~\ref{definition:tree-shaped-cooperad} uses the tensor product $\otimes$.
To pass from one to the other, we need to use the natural transformation $\EM$ described in Construction~\ref{constr:Barratt-Eccles-diagonal}.

\begin{prop}\label{proposition:forgetful-strict}
Let $\AOp$ be a strict Segal $E_\infty$-Hopf shuffle cooperad, with coproduct operators $\rho_{\ttree\rightarrow\stree}: \AOp(\stree)\rightarrow\AOp(\ttree)$
and facet operators $i_{\sigmatree,\stree}: \AOp(\sigmatree)\rightarrow\AOp(\stree)$.
The collection $\AOp(\ttree)$, $\ttree\in\Tree(r)$, equipped with the coproduct operators inherited from $\AOp$ and the Segal maps
given by the composites
\begin{equation*}
\bigotimes_{v\in V(\stree)}\AOp(\sigmatree_v)\xrightarrow{\EM}\bigvee_{v\in V(\stree)}\AOp(\sigmatree_v)\xrightarrow{i_{\lambda_{\stree}(\sigmatree_*)}}\AOp(\ttree),
\end{equation*}
for all tree decompositions $\ttree = \lambda_{\stree}(\sigmatree_v,v\in V(\stree))$, is a Segal shuffle dg cooperad.
\end{prop}

\begin{proof}
We easily deduce from the associativity of the facet operators in the definition of a Segal $E_\infty$-Hopf shuffle cooperad~(\S\ref{definition:strict-E-infinity-cooperad})
that the Segal maps of $\EOp$-algebras $i_{\lambda_{\stree}(\sigmatree_*)}: \bigvee_{v\in V(\stree)}\AOp(\sigmatree_v)\rightarrow\AOp(\ttree)$
satisfy natural associativity relations, which parallel the associativity relations of the Segal maps
of Segal shuffle cooperads in dg module. (By the universal properties of coproducts, we are left to verifying such a relation on a single summand of our coproducts.)
We use the associativity of the transformation $\EM$ to pass from this associativity relation on coproducts
to the associativity relation on tensor products
which is required in the definition of Segal shuffle dg cooperad.
We eventually deduce from the result of Proposition~\ref{claim:Barratt-Eccles-algebra-coproducts}
that the dg cooperad version of the Segal condition for $\AOp$
is equivalent to the Segal condition of strict Segal $E_\infty$-Hopf cooperads.
\end{proof}

\begin{remark}\label{remark:order-trees}
Note that the definition of the Segal map in the construction of this proposition
requires to pick an order on the vertices of the tree $\stree$
since the transformation $\EM$ is not commutative.
\end{remark}

We immediately see that an ordinary shuffle dg cooperad, in the sense of the definition of~\S\ref{subsection:shuffle-cooperads}, is equivalent to a Segal shuffle dg cooperad
where the Segal maps define isomorphisms $i_{\lambda_{\stree}(\sigmatree_*)}: \bigotimes_{v\in V(\stree)}\AOp(\sigmatree_v)\xrightarrow{\simeq}\AOp(\ttree)$.
We aim to prove that every Segal shuffle dg cooperad is weakly-equivalent to such an ordinary shuffle dg cooperad.
We then need to assume that our Segal shuffle dg cooperad satisfies the connectedness condition of~\S\ref{subsection:conilpotence}.
Namely, we have to assume that $\AOp(\ttree) = 0$ when the tree $\ttree$ is not reduced (contains at least one vertex with a single ingoing edge).
Recall that we say that $\AOp$ is connected when it satisfies this condition.

We use a version of the Boardman--Vogt $W$-construction in order to establish the existence of our equivalences.
The classical Boardman--Vogt construction (see~\cite{BergerMoerdijkResolution,BoardmanVogt})
is defined for ordinary operads (actually for algebraic theories in the original reference~\cite{BoardmanVogt}).
We therefore have to dualize the classical definition in order to deal with cooperads (rather than with operads)
and we have to extend the construction to the context of Segal dg cooperads.
We explain the definition of this $W$-construction of Segal dg cooperads with full details in the next paragraphs.
In a first step, we explain the definition of a covariant functor of cubical cochains on the category of trees.
We will pair this functor with the contravariant functor underlying a Segal shuffle dg cooperad to define our object.

In fact, we do not need the full connectedness condition for the definition of the Boardman--Vogt $W$-construction of a Segal shuffle dg cooperad $\AOp$,
because the definition makes sense as soon as the coproduct operators of our object
fulfill the local conilpotence property of~\S\ref{subsection:conilpotence} (which is implied
by the connectedness condition, but can be satisfied in a broader context). We will explain the definition of the Boardman--Vogt $W$-construction in this setting.

\begin{constr}
Fix $\stree\in\Tree(r)$.
For a tree morphism $\ttree\rightarrow\stree$, equivalent to a treewise decomposition $\ttree = \lambda_{\stree}(\sigmatree_v,v\in V(\stree))$,
we set:
\begin{equation*}
\square^*(\ttree/\stree) = \bigotimes_{e\in\mathring{E}(\sigmatree_v),v\in V(\stree)}\underbrace{\DGN^*(\Delta^1)}_{=: I_e},
\end{equation*}
where we associate a factor $I_e = \DGN^*(\Delta^1)$ to every inner edge of a subtree $\sigmatree_v\subset\ttree$ (recall that $\mathring{E}(\thetatree)$
denotes the set of inner edges, which we associate to any tree $\thetatree$
in our category).

This collection of dg modules $\square^*(\ttree/\stree)$ defines a covariant functor on the over category of tree morphisms $\ttree\rightarrow\stree$, where $\ttree\in\Tree(r)$.
Recall that the cochain complex $\DGN^*(\Delta^1)$ is given by $\DGN^*(\Delta^1) = \kk\underline{0}^{\sharp}\oplus\kk\underline{1}^{\sharp}\kk\underline{01}^{\sharp}$,
with the differential such that $\delta(\underline{0}^{\sharp}) = -\underline{01}^{\sharp}$
and $\delta(\underline{1}^{\sharp}) = \underline{01}^{\sharp}$ (see~\S\ref{constr:cubical-cochain-connection}).
We use that a composite of tree morphisms $\ttree\rightarrow\utree\rightarrow\stree$
is equivalent to a double decomposition $\utree = \lambda_{v\in V(\stree)}(\stree_v,v\in V(\stree))$
and $\ttree = \lambda_{u\in V(\utree)}(\sigmatree_u,u\in V(\utree))$,
which yields $\ttree = \lambda_{v\in V(\stree)}(\thetatree_v,v\in V(\stree))$ with $\thetatree_v = \lambda_{u\in V(\stree_v)}(\sigmatree_u,u\in V(\stree_v))$.
We define
\begin{equation*}
\partial_{\ttree\rightarrow\utree/\stree}: \square^*(\ttree/\stree)\rightarrow\square^*(\utree/\stree)
\end{equation*}
as the morphism of dg modules induced by the identity mapping on the factors $I_e$ associated to the edges such that $e\not\in\mathring{E}(\sigmatree_u)$, for all $u\in V(\utree)$,
and by the map $d_0: I_e\rightarrow\kk$ such that $d_0(\underline{1}^{\sharp}) = 1$ and $d_0(\underline{0}^{\sharp}) = d_0(\underline{01}^{\sharp}) = 0$
on the factors $I_e$
associated to the edges $e$ such that we have $e\in\mathring{E}(\sigmatree_u)$, for some $u\in V(\utree)$ (the edges which collapse when we pass to $\utree$).

The collection $\square^*(\ttree/\stree)$ also defines a contravariant functor on the under category of tree morphisms $\ttree\rightarrow\stree$,
when we fix $\ttree\in\Tree(r)$ instead of $\stree\in\Tree(r)$.
For tree morphisms $\ttree\rightarrow\utree\rightarrow\stree$ as above,
we then consider the map
\begin{equation*}
\rho_{\ttree/\utree\rightarrow\stree}: \square^*(\ttree/\stree)\rightarrow\square^*(\ttree/\utree)
\end{equation*}
induced by the identity mapping on the factors $I_e$
associated to the edges such that $e\in\mathring{E}(\sigmatree_u)$, for some $u\in V(\utree)$,
and by the map $d_1: I_e\rightarrow\kk$ such that $d_1(\underline{0}^{\sharp}) = 1$ and $d_1(\underline{1}^{\sharp}) = d_1(\underline{01}^{\sharp}) = 0$
on the factors $I_e$ associated to the edges $e$
such that $e\not\in\mathring{E}(\sigmatree_u)$, for all $u\in V(\utree)$.

We easily check that the above constructions yield associative covariant and contravariant actions on our collection of dg modules $\square^*(\ttree/\stree)$,
which, in addition, commute to each other. We accordingly get that our mapping $(\ttree\rightarrow\stree)\mapsto\square^*(\ttree/\stree)$
defines a covariant functor on the comma category of tree morphisms $\ttree\rightarrow\stree$.
\end{constr}

\begin{remark}
The application of the face operator $d_0$ for the definition of the covariant functor structure on the collection $\square^*(\ttree/\stree)$
and the application of the face operator $d_1$ for the definition of the contravariant functor structure
is converse to the usual convention for the definition of the $W$-construction.
This choice is motivated by our choices regarding the definition of the cobar construction of homotopy Segal dg cooperads,
which are themselves forced by the definition of connections
in the category of $\EOp$-algebras, and by a seek of coherence between the definition of the $W$-construction
and the definition of the cobar construction of (homotopy) Segal dg cooperads,
which we need in order to be able to compare the $W$-construction
with the cobar construction.
\end{remark}

We now address the definition of the $W$-construction.

\begin{constr}\label{construction:W}
Let $\AOp$ be a Segal shuffle dg pre-cooperad.
We assume that the treewise coproduct operators on $\AOp$ satisfies the local conilpotence condition of~\S\ref{subsection:conilpotence}.
We set:
\begin{equation*}
\DGW^{c}(\AOp)(\stree) = \eq(\xymatrix{\bigoplus_{\ttree\rightarrow\stree}\square^*(\ttree/\stree)\otimes\AOp(\ttree)\ar@<+2pt>[r]^{d^0}\ar@<-2pt>[r]_{d^1}
& \bigoplus_{\ttree\rightarrow\utree\rightarrow\stree}\square^*(\utree/\stree)\otimes\AOp(\ttree)\ar@/_2em/[l]_{s^0} }),
\end{equation*}
where we take the equalizer $\eq$ of the map $d^0$ induced by the covariant action $\square^*(\ttree/\stree)\rightarrow\square^*(\utree/\stree)$ of the tree morphisms $\ttree\rightarrow\utree$
and of the map $d^1$ induced by the contravariant action $\AOp(\utree)\rightarrow\AOp(\ttree)$
on our tensors. (The reflection map $s^0$ is given by the projection onto the summands such that $\ttree = \utree$.)
We may equivalently use the following end-style notation for this equalizer:
\begin{equation*}
\DGW^{c}(\AOp)(\stree) = \int_{\ttree\rightarrow\stree}'\square^*(\ttree/\stree)\otimes\AOp(\ttree),
\end{equation*}
where the notation $\int'$ refers to the consideration of sums (rather than products) in the above equalizer definition of our object.
We note that this additive end is well defined because the local conilpotence condition implies that the contravariant action operations $\AOp(\utree)\rightarrow\AOp(\ttree)$
land in a sum when $\ttree$ varies, while the covariant action operations $\square^*(\ttree/\stree)\rightarrow\square^*(\utree/\stree)$
land in a sum because each tree morphism $\ttree\rightarrow\stree$ has finitely many factorizations $\ttree\rightarrow\utree\rightarrow\stree$.

The objects $\DGW^{c}(\AOp)(\stree)$ inherit natural coproduct operators $\rho^W_{\utree\rightarrow\stree}: \DGW^{c}(\AOp)(\stree)\rightarrow\DGW^{c}(\AOp)(\utree)$
by covariant functoriality of the objects $\square^*(\ttree/\stree)$.
Besides, we have, for each tree decomposition $\stree = \lambda_{\utree}(\sigmatree_u, u\in V(\utree))$,
a natural Segal map
\begin{gather*}
i^W_{\sigmatree_*,\stree}: \bigotimes_{u\in V(\utree)}\DGW^{c}(\AOp)(\sigmatree_v)\rightarrow\DGW^{c}(\AOp)(\stree),
\intertext{induced by the following operators on our additive end}
\bigotimes_{u\in V(\utree)}\square^*(\thetatree_u/\sigmatree_u)\otimes\AOp(\thetatree_u)
\rightarrow\square^*(\lambda_{\utree}(\thetatree_u)/\stree)\otimes\AOp(\lambda_{\utree}(\thetatree_u)),
\end{gather*}
where we fix a collection of tree morphisms $\thetatree_u\rightarrow\sigmatree_u$, $u\in V(\utree)$, which we put together on $\utree$
in order to get a tree morphism $\lambda_{\utree}(\thetatree_u)\rightarrow\lambda_{\utree}(\sigmatree_u) = \stree$,
and we use the obvious isomorphism $\bigotimes_{u\in V(\utree)}\square^*(\thetatree_u/\sigmatree_u)\simeq\square^*(\lambda_{\utree}(\thetatree_u)/\stree)$
together with the Segal map $i_{\lambda_{\utree}(\thetatree_*)}: \bigotimes_{u\in V(\utree)}\AOp(\thetatree_u)\rightarrow\AOp(\ttree)$ on $\AOp$
for the tree $\ttree = \lambda_{\utree}(\thetatree_u)$.
\end{constr}

We check that the above construction provides the object $\DGW^c(\AOp)$ with a well-defined Segal dg cooperad structure later on.
We define, before carrying this verification, a decomposed version of the $W$-construction. The idea is to replace the contravariant functor $\AOp(\ttree)$
in the definition of the $W$-construction by a decomposed version of this functor, which we construct in the next paragraph.

\begin{constr}\label{construction:decomposed-Segal-cooperad-terms}
We again assume $\AOp$ that is a Segal shuffle dg pre-cooperad.
For a tree morphism $\ttree\rightarrow\stree$, equivalent to a treewise decomposition $\ttree = \lambda_{\stree}(\sigmatree_v,v\in V(\stree))$, we set:
\begin{equation*}
\AOp(\ttree/\stree) = \bigotimes_{v\in V(\stree)}\AOp(\sigmatree_v).
\end{equation*}

The collection of these objects inherits the structure of a contravariant functor on the over category of tree morphisms $\ttree\rightarrow\stree$,
where we fix the tree $\stree\in\Tree(r)$
and $\ttree$ varies. We proceed as follows.
Let $f: \utree\rightarrow\ttree$ be a tree morphism, which we compose with the above morphism $\ttree\rightarrow\stree$ to get $\utree\rightarrow\stree$.
In this case, for the decomposition $\utree = \lambda_{\stree}(\thetatree_v,v\in V(\stree))$,
we have $\thetatree_v = f^{-1}\sigmatree_v\subset\utree$, where we consider the pre-image of the subtree $\sigmatree_v\subset\ttree$
under the morphism $f: \utree\rightarrow\ttree$.
Furthermore, we can identify our morphism $f: \utree\rightarrow\ttree$
with the morphism $\lambda_{\stree}(\thetatree_v)\rightarrow\lambda_{\stree}(\sigmatree_v)$
which we obtain by putting together the morphisms $f|_{\thetatree_v}: \thetatree_v\rightarrow\sigmatree_v$
on the tree $\stree$. We just define the decomposed coproduct operator
\begin{equation*}
\rho_{\utree\rightarrow\ttree/\stree}: \AOp(\ttree/\stree)\rightarrow\AOp(\utree/\stree)
\end{equation*}
as the tensor product of the coproduct operators $\rho_{\thetatree_v\rightarrow\sigmatree_v}: \AOp(\sigmatree_v)\rightarrow\AOp(\thetatree_v)$
which we associate to these restrictions $f|_{\thetatree_v}: \thetatree_v\rightarrow\sigmatree_v$.

The collection $\AOp(\ttree/\stree)$ also defines a covariant functor on the under category of tree morphisms $\ttree\rightarrow\stree$,
when we make the tree $\stree$ vary and we fix $\ttree\in\Tree(r)$.
We consider a composable sequence of morphisms $\ttree\rightarrow\utree\rightarrow\stree$.
We write $\utree = \lambda_{v\in V(\stree)}(\stree_v)$ for the decomposition equivalent to the morphism $f: \utree\rightarrow\stree$
and $\ttree = \lambda_{u\in V(\utree)}(\sigmatree_u)$ for the decomposition equivalent to the morphism $\ttree\rightarrow\utree$.
Then the decomposition equivalent to the composite morphism $\ttree\rightarrow\stree$
reads $\ttree = \lambda_{v\in V(\stree)}(\thetatree_v,v\in V(\stree))$ with $\thetatree_v = \lambda_{\stree_v}(\sigmatree_u,u\in V(\stree_v))$,
for $v\in V(\stree)$.
The operator
\begin{equation*}
i_{\ttree/\utree\rightarrow\stree}: \AOp(\ttree/\utree)\rightarrow\AOp(\ttree/\stree)
\end{equation*}
of the covariant action of $f: \utree\rightarrow\stree$ on our collection
is given by the tensor product of the Segal maps $i_{\lambda_{\stree_v}(\sigmatree_*)}: \bigotimes_{u\in V(\stree_v)}\AOp(\sigmatree_u)\rightarrow\AOp(\thetatree_v)$
which we associate to the decompositions $\thetatree_v = \lambda_{\stree_v}(\sigmatree_u,u\in V(\stree_v))$.

We easily check that the above constructions yield associative covariant and contravariant actions on our collection $\AOp(\ttree/\stree)$,
which, in addition, commute to each other. We accordingly get that our mapping $(\ttree\rightarrow\stree)\mapsto\AOp(\ttree/\stree)$
defines a contravariant functor on the comma category of tree morphisms $\ttree\rightarrow\stree$.
\end{constr}

We can now proceed to the definition of our decomposed $W$-construction.

\begin{constr}\label{construction:decomposed-W}
Let $\AOp$ be a Segal shuffle dg cooperad. We assume that $\AOp$ satisfies the local conilpotence condition of~\S\ref{subsection:conilpotence} as in Construction~\ref{construction:W}.
We set:
\begin{equation*}
\DGW^{c}_{dec}(\AOp)(\stree) = \int_{\ttree\rightarrow\stree}'\square^*(\ttree/\stree)\otimes\AOp(\ttree/\stree),
\end{equation*}
where the notation $\int'$ refers to the same additive end construction as in Construction~\ref{construction:W}, and we consider the decomposed contravariant functor $\AOp(\ttree/\stree)$
defined in the previous paragraph.
We note again that this additive end is well defined because the local conilpotence condition implies that the decomposed coproduct operators $\AOp(\ttree/\stree)\rightarrow\AOp(\utree/\stree)$
land in a direct sum when $\utree$ varies, like the coproduct operators $\AOp(\ttree)\rightarrow\AOp(\utree)$
in Construction~\ref{construction:W} (and yet because the covariant action on $\square^*(\ttree/\stree)$
involves finitely many terms on each term of the end).

The coproduct operators $\rho^W_{\utree\rightarrow\stree}: \DGW^{c}_{dec}(\AOp)(\stree)\rightarrow\DGW^{c}_{dec}(\AOp)(\utree)$
are defined by using the covariant functoriality of the objects $\square^*(\ttree/\stree)$
and $\AOp(\ttree/\stree)$.
For each tree decomposition $\stree = \lambda_{\utree}(\sigmatree_u,u\in V(\utree))$, we define the Segal map
\begin{gather*}
i^W_{\sigmatree_*,\stree}: \bigotimes_{u\in V(\utree)}\DGW^{c}(\AOp)(\sigmatree_u)\rightarrow\DGW^{c}(\AOp)(\stree)
\intertext{termwise, by the morphisms}
\bigotimes_{u\in V(\utree)}\square^*(\thetatree_u/\sigmatree_u)\otimes\AOp(\thetatree_u/\sigmatree_u)
\rightarrow\square^*(\lambda_{\utree}(\thetatree_u)/\stree)\otimes\AOp(\lambda_{\utree}(\thetatree_u)/\stree),
\end{gather*}
associated to the collections of tree morphisms $\thetatree_u\rightarrow\sigmatree_u$, $u\in V(\utree)$,
which we get by tensoring the same isomorphisms $\bigotimes_{u\in V(\utree)}\square^*(\thetatree_u/\sigmatree_u)\simeq\square^*(\lambda_{\utree}(\thetatree_u)/\stree)$
as in Construction~\ref{construction:W}
with parallel isomorphisms $\bigotimes_{u\in V(\utree)}\AOp(\thetatree_u/\sigmatree_u)\simeq\AOp(\lambda_{\utree}(\thetatree_u)/\stree)$,
which we associate to the objects of Construction~\ref{construction:decomposed-Segal-cooperad-terms}.
\end{constr}

We have the following observation, which can be used to give a reduced description of both the $W$-construction $\DGW^{c}(\AOp)$
and the decomposed $W$-construction $\DGW^{c}_{dec}(\AOp)$.

\begin{lemm}\label{lemma:W-construction-splitting}
The additive end equalizers in the definition of the $W$-construction $\DGW^{c}(\AOp)$ in Construction~\ref{construction:W}
and in the definition of the decomposed $W$-construction $\DGW^{c}_{dec}(\AOp)$
in Construction~\ref{construction:decomposed-W}
split when we forget about differentials, so that the terms of these objects have a reduced description
of the form:
\begin{align*}
\DGW^{c}(\AOp)(\stree) & = \bigoplus_{\ttree\rightarrow\stree}L\square^*(\ttree/\stree)\otimes\AOp(\ttree), \\
\DGW^{c}_{dec}(\AOp)(\stree) & = \bigoplus_{\ttree\rightarrow\stree}L\square^*(\ttree/\stree)\otimes\AOp(\ttree/\stree),
\end{align*}
where, for a tree morphism $\ttree\rightarrow\stree$ equivalent to a tree decomposition such that $\ttree = \lambda_{v\in V(\stree)}(\sigmatree_v,v\in V(\stree))$,
we define $L\square^*(\ttree/\stree)\subset\square^*(\ttree/\stree)$
by the tensor product:
\begin{equation*}
L\square^*(\ttree/\stree) = \bigotimes_{e\in\mathring{E}(\sigmatree_v),v\in V(\stree)}\underbrace{(\kk\underline{0}^{\sharp}\oplus\kk\underline{01}^{\sharp})}_{=: \mathring{I}_e}.
\end{equation*}
(Thus, we just drop the factors $\underline{1}^{\sharp}$ from the normalized cochain complexes $I_e = N^*(\Delta^1)$ in the expression of the object $\square^*(\ttree/\stree)$.)
\end{lemm}

\begin{proof}
This lemma readily follows from the fact that the terms $\varpi = \sigma\otimes\alpha\in\square^*(\ttree/\stree)\otimes\AOp(\ttree)$
(respectively, $\varpi = \sigma\otimes\alpha\in\square^*(\ttree/\stree)\otimes\AOp(\ttree/\stree)$)
with $\underline{1}^{\sharp}$ factors in the additive end definition of the object $\DGW^{c}(\AOp)(\stree)$ (respectively, $\DGW^{c}_{dec}(\AOp)(\stree)$)
are determined by the equalizer relations,
which identify such terms with the image of tensors of the form $\varpi' = \sigma'\otimes\alpha'\in L\square^*(\ttree'/\stree)\otimes\AOp(\ttree')$
(respectively, $\varpi' = \sigma'\otimes\alpha'\in L\square^*(\ttree'/\stree)\otimes\AOp(\ttree'/\stree)$)
under the action of coproduct operators,
where $\sigma'$ is defined by withdrawing the factors $\underline{1}^{\sharp}$ from $\sigma\in\square^*(\ttree/\stree)$
and $\ttree'$ is the tree obtained by contracting the edges $e$
that correspond to such factors in $\ttree$.
Indeed, we then have $\sigma' = \partial_{\ttree\rightarrow\ttree'/\stree}(\sigma)$ and, in the case of $W$-construction $\DGW^{c}(\AOp)$,
from the zigzag of morphisms
\begin{equation*}
\xymatrix{ \square^*(\ttree/\stree)\otimes\AOp(\ttree)\ar[dr]_{\partial_{\ttree\rightarrow\ttree'/\stree}\otimes\id} &&
\square^*(\ttree'/\stree)\otimes\AOp(\ttree')\ar[dl]^{\id\otimes\rho_{\ttree\rightarrow\ttree'}} \\
& \square^*(\ttree'/\stree)\otimes\AOp(\ttree) & },
\end{equation*}
which we extract from our equalizer, we see that we have the identity $\sigma'\otimes\alpha = \sigma'\otimes\rho_{\ttree\rightarrow\ttree'}(\alpha')\Rightarrow\alpha = \rho_{\ttree\rightarrow\ttree'}(\alpha')$.
We argue similarly in the case of the decomposable $W$-construction.
\end{proof}

We now check the validity of the definition of our Segal shuffle dg pre-cooperad structure on the $W$-construction $\DGW^{c}(\AOp)$ in Construction~\ref{construction:W}
and on the decomposed $W$-construction $\DGW^{c}_{dec}(\AOp)$ in Construction~\ref{construction:decomposed-W}.
We use the following straightforward observation (see \cite[Appendix A]{FresseBook}).

\begin{lemm}\label{lemma:tree-decomposition-chain}
Let $\stree = \lambda_{\utree}(\sigmatree_u,u\in V(\utree))$ be a tree decomposition.
There is a bijection between the set of collections of composable pairs of tree morphisms $\{\thetatree_u\rightarrow\ttree_u\rightarrow\sigmatree_u,u\in V(\utree)\}$
indexed by $V(\utree)$
and the set of composable pairs of tree morphisms $\thetatree\rightarrow\ttree\rightarrow\stree$.
This bijection associates any such collection $\{\thetatree_u\rightarrow\ttree_u\rightarrow\sigmatree_u,u\in V(\utree)\}$
with the morphisms $\lambda_{\utree}(\thetatree_u,u\in V(\utree))\rightarrow\lambda_{\utree}(\ttree_u,u\in V(\utree))\rightarrow\lambda_{\utree}(\sigmatree_u,u\in V(\utree)) = \stree$.\qed
\end{lemm}

\begin{thm-defn}\label{lemma:W-construction-cooperad}
The objects $\DGW^{c}(\AOp)$ and $\DGW^{c}_{dec}(\AOp)$, equipped with the coproduct operators and the Segal maps defined in Construction~\ref{construction:W},
form Segal shuffle dg pre-cooperads, to which we respectively refer as the $W$-construction and the decomposed $W$-construction
of the Segal shuffle dg (pre-)cooperad $\AOp$.

For the decomposed $W$-construction $\DGW^{c}_{dec}(\AOp)$, we get in addition that the Segal maps
define isomorphisms
\begin{equation*}
i_{\lambda_{\stree}(\sigmatree_*)}: \bigotimes_{v\in V(\stree)}\DGW^{c}_{dec}(\AOp)(\sigmatree_v)\xrightarrow{\simeq}\DGW^{c}_{dec}(\AOp)(\ttree),
\end{equation*}
for all tree decompositions $\ttree = \lambda_{\stree}(\sigmatree_v,v\in V(\stree))$, so that $\DGW^{c}_{dec}(\AOp)$
is identified with a shuffle dg cooperad in the ordinary sense.
\end{thm-defn}

\begin{proof}
The associativity of the coproduct operators on $\DGW^{c}(\AOp)$ and $\DGW^{c}_{dec}(\AOp)$
is immediate from the definition of these morphisms
in terms of associative actions on the terms
of our additive ends
in Construction~\ref{construction:W} and Construction~\ref{construction:decomposed-W}.
We similarly check the validity of the associativity condition for the Segal maps that we attach to our objects.
We also deduce the compatibility between the Segal maps and the coproduct operators from termwise counterparts of this relation.

To establish that Segal maps define isomorphisms in the case of the decomposed $W$-construction, we use the reduced expression of Lemma~\ref{lemma:W-construction-splitting},
the fact that the tensor products of this reduced expression have a factorization
of the form
\begin{multline*}
L\square^*(\ttree/\stree)\otimes\AOp(\ttree/\stree) = (\bigotimes_{e\in\mathring{E}(\sigmatree_v),v\in V(\stree)}\mathring{I}_e)\otimes(\bigotimes_{v\in V(\stree)}\AOp(\sigmatree_v))\\
\simeq\bigotimes_{v\in V(\stree)}(\underbrace{\bigotimes_{e\in\mathring{E}(\sigmatree_v)}\mathring{I}_e}_{= L\square^*(\sigmatree_v/\ytree_v)}
\otimes\underbrace{\AOp(\sigmatree_v)}_{= \AOp(\sigmatree_v/\ytree_v)})
\end{multline*}
and the bijective correspondence of Lemma~\ref{lemma:tree-decomposition-chain}.
\end{proof}

\begin{thm}\label{proposition:quasi-isomorphism-bar-strict-cooperads}
If $\AOp$ is a connected Segal shuffle dg pre-cooperad $\AOp$ (where we use the connectedness condition of~\S\ref{subsection:conilpotence}),
then we have a zigzag of natural transformations of Segal shuffle dg pre-cooperads
\begin{equation*}
\AOp\xrightarrow{\sim}\DGW^c(\AOp)\leftarrow\DGW^c_{dec}(\AOp)
\end{equation*}
where the morphism on the left-hand side is a weak-equivalence termwise.

If $\AOp$ satisfies the Segal condition (and, therefore, is a Segal shuffle dg cooperad $\AOp$),
then the morphism on the right-hand side of this zigzag is also a weak-equivalence
and the $W$-construction $\DGW^c(\AOp)$ also satisfies the Segal condition,
so that $\AOp$ is, as a Segal shuffle dg cooperad, weakly-equivalent to a shuffle dg cooperad in the ordinary sense.
\end{thm}

\begin{proof}
We address the definition of the morphism $\beta: \AOp\rightarrow\DGW^{c}(\AOp)$ first.
We consider the dg module morphisms $\beta_{\ttree\rightarrow\stree}: \AOp(\stree)\rightarrow\square^*(\ttree/\stree)\otimes\AOp(\ttree)$
defined by pairing the coproducts $\rho_{\ttree\rightarrow\stree}: \AOp(\stree)\rightarrow\AOp(\ttree)$
associated to the morphisms $\ttree\rightarrow\stree$
with the unit morphisms $\eta: \kk\rightarrow I_e$ of the cochain algebras $I_e = \DGN^*(\Delta^1)$
in $\square^*(\ttree/\stree)$.
We readily check that these morphisms $\beta_{\ttree\rightarrow\stree}: \AOp(\stree)\rightarrow\square^*(\ttree/\stree)\otimes\AOp(\ttree)$
induce a morphism with values in the additive end of Construction~\ref{construction:W}. (We just note that the local conilpotence condition
implies again that the collection of these morphisms land in the sum of the objects $\square^*(\ttree/\stree)\otimes\AOp(\ttree)$ when $\ttree$ varies.)
We accordingly get a dg module morphism $\beta: \AOp(\stree)\rightarrow\DGW^{c}(\AOp)(\stree)$, for each tree $\stree$.
We easily deduce from the associativity of the coproduct operators that these morphisms commute with the coproduct operators on $\AOp$
and on $\DGW^{c}(\AOp)(\stree)$.
We similarly prove that our morphisms preserve the Segal maps by reducing the verification of this claim
to a termwise relation.
We therefore get a well defined morphism of Segal shuffle dg cooperads $\beta: \AOp\rightarrow\DGW^{c}(\AOp)$
as requested.

We now check that this morphism defines a termwise weak-equivalence $\beta: \AOp(\stree)\xrightarrow{\sim}\DGW^{c}(\AOp)(\stree)$.
For this purpose, we use the reduced expression of the $W$-construction given in Lemma~\ref{lemma:W-construction-splitting},
and we take a filtration of our object by the number of vertices of the trees $\ttree$
in this expansion.
We see that the terms of the differential given by the map $\delta(\underline{0}^{\sharp}) = -\underline{01}^{\sharp}$
in the normalized cochain complexes $I_e = \DGN^*(\Delta^1)$
preserve this grading, as well as the term of the differential induced by the internal differential
of the dg modules $\AOp(\ttree)$,
but the terms of the differential given by the map $\delta(\underline{1}^{\sharp}) = \underline{01}^{\sharp}$
increase the number of vertices
when we pass to the reduced expansion.
We just take the spectral sequence associated to our filtration to neglect the latter terms and to reduce the differential of our object
to the terms given by the maps $\delta(\underline{0}^{\sharp}) = -\underline{01}^{\sharp}$
and the internal differential of the dg modules $\AOp(\ttree)$.
The acyclicity of the cochain complex $\mathring{I}_e = (\kk\underline{0}^{\sharp}\oplus\kk\underline{01}^{\sharp},\delta(\underline{0}^{\sharp}) = -\underline{01}^{\sharp})$
implies that all terms of our reduced expansion have a trivial homology, except the term associated to the identity morphism $\ttree = \stree\rightarrow\stree$
for which we have $L\square^*(\stree/\stree) = \kk$.
Hence, our map $\beta: \AOp(\stree)\rightarrow\DGW^{c}(\AOp)(\stree)$ induces an isomorphism on the first page of the spectral sequence associated to our filtration,
and we conclude from this result that $\beta: \AOp(\stree)\rightarrow\DGW^{c}(\AOp)(\stree)$
defines a weak-equivalence of dg modules,
as requested.
Note simply that the connectedness assumption of the theorem implies that the object $\DGW^{c}(\AOp)(\stree)$ reduces to a finite sum, for any given tree $\stree$,
because we have only finitely many morphism $\ttree\rightarrow\stree$ such that $\ttree$ is reduced,
and for which the object $\AOp(\ttree)$ does not vanish (see~\S\ref{subsection:conilpotence}).
This observation ensures that no convergence issue occurs in this spectral argument.

We obviously define our second morphism $\alpha: \DGW^c_{dec}(\AOp)\rightarrow\DGW^c(\AOp)$ termwise,
by taking the tensor product of the identity map on $\square^*(\ttree/\stree)$
with the morphism
\begin{equation}\tag{$*$}\label{eqn:termwise_segal_maps}
\AOp(\ttree/\stree) = \bigotimes_{v\in V(\stree)}\AOp(\sigmatree_v)\xrightarrow{i_{\lambda_{\stree}(\sigmatree_*)}}\AOp(\ttree)
\end{equation}
given by the Segal map on $\AOp$, where we consider the decomposition $\ttree = \lambda_{v\in V(\stree)}(\sigmatree_v)$
equivalent to the morphism $\ttree\rightarrow\stree$.
We just check that these maps~(\ref{eqn:termwise_segal_maps}) define a morphism of bifunctors on the comma category of tree morphisms $\ttree\rightarrow\stree$
to obtain that they induce a well-defined map on our additive end. This map also commutes with the action of our coproduct operators
on $\DGW^c_{dec}(\AOp)$ and $\DGW^c(\AOp)$.
We easily deduce, from the associativity relations of the Segal maps, that the maps~(\ref{eqn:termwise_segal_maps})
intertwine the action of the Segal operators on $\DGW^c_{dec}(\AOp)$ and $\DGW^c(\AOp)$
and hence, define a morphism of shuffle dg pre-cooperad.

We again use the spectral sequence determined by the filtration by the number of vertices of trees
in the reduced expansions of Lemma~\ref{lemma:W-construction-splitting}
to study the effect of our map in homology.
We use that the first page of the spectral sequence is given by these reduced expansions
with the differential inherited from our objects $\AOp(\ttree/\stree)$ and $\AOp(\ttree)$
and from the terms $\delta(\underline{0}^{\sharp}) = -\underline{01}^{\sharp}$
of the differential on the factors $\mathring{I}_e\subset I_e = \DGN^*(\Delta^1)$.
We immediately deduce that our morphism induces an isomorphism on the first page of this spectral sequence as our maps~(\ref{eqn:termwise_segal_maps}),
which we pair with the identity of the object $L\square^*(\ttree/\stree)$ in our reduced expansions,
are weak-equivalences when we assume that $\AOp$ satisfies the Segal condition.
We conclude from this observation that our morphism $\alpha: \DGW^c_{dec}(\AOp)\rightarrow\DGW^c(\AOp)$
defines a termwise weak-equivalence of Segal shuffle dg pre-cooperads.
(We again use the connectedness assumption to ensure that no convergence issue occurs in this spectral sequence argument.)

Finally, the existence of such a weak-equivalence of Segal shuffle dg pre-cooperads $\alpha: \DGW^c_{dec}(\AOp)\xrightarrow{\sim}\DGW^c(\AOp)$
immediately implies that $\DGW^c_{dec}(\AOp)$ fulfills the Segal condition,
and hence, defines a Segal shuffle dg cooperad,
since the Segal maps for $\DGW^c(\AOp)$ are weakly-equivalent to the Segal maps for $\DGW^c_{dec}(\AOp)$,
which are isomorphisms by the result of Proposition~\ref{lemma:W-construction-cooperad}.
\end{proof}

\subsection{The application of cobar complexes}\label{subsection:strict-Segal-cooperad-cobar}

In this subsection, we describe a cobar construction for Segal shuffle dg (pre-)cooperads. This is a construction that, from the structure of a Segal shuffle dg pre-cooperad $\AOp$,
produces an operad in dg modules $\DGB^c(\AOp)$.
We make explicit the definition of the structure operations of the cobar construction $\DGB^c(\AOp)$
in the next paragraph.
We check the validity of these definitions afterwards.
We eventually make a formal statement to record the definition of this dg operad $\DGB^c(\AOp)$.

\begin{constr}\label{construction:bar-complex}
We define the graded modules $\DGB^c(\AOp)(r)$, which form the components of the cobar construction $\DGB^c(\AOp)$,
by the following formula:
\begin{equation*}
\DGB^c(\AOp)(r) = \bigoplus_{\ttree\in\Tree(r)}\DGSigma^{-\sharp V(\ttree)}\AOp(\ttree),
\end{equation*}
where $\DGSigma$ is the suspension functor on graded modules.
We also have an identity $\DGSigma^{-\sharp V(\ttree)}\AOp(\ttree) = \bigl(\bigotimes_{v\in V(\ttree)}\underline{01}^{\sharp}_v\bigr)\otimes\AOp(\ttree)$,
where we associate a factor of cohomological degree one $\underline{01}^{\sharp}_v$ to every vertex $v\in V(\ttree)$.
We just need to fix an ordering choice on the vertices of our tree to get this representation, as a permutation of factors produces a sign
in the tensor product $\bigotimes_{v\in V(\ttree)}\underline{01}^{\sharp}_v$.

To every pair $(\ttree,e)$, with $\ttree\in\Tree(r)$ and $e\in\mathring{E}(\ttree)$,
we associate the map
\begin{equation*}
\partial_{(\ttree,e)}: \DGSigma^{-\sharp V(\ttree)+1}\AOp(\ttree/e)\rightarrow\DGSigma^{-\sharp V(\ttree)}\AOp(\ttree)
\end{equation*}
given by the coproduct operator $\partial_{(\ttree,e)} = \rho_{\ttree\rightarrow\ttree/e}$ on $\AOp(\ttree/e)$, where $\ttree/e$ is defined by contracting the edge $e$ in $\ttree$,
together with the mapping $\underline{01}_{u\equiv v}\mapsto\underline{01}_u\otimes\underline{01}_v$
in the tensor product $\underline{01}_{u\equiv v}\otimes\bigotimes_{x\in V(\ttree/e)\setminus\{u\equiv v\}}\underline{01}^{\sharp}_x$,
where $u$ and $v$ respectively denote the source and the target of the edge $e$ in $\ttree$,
while $u\equiv v$ denotes the result of the fusion of these vertices in $\ttree/e$. (Note that $V(\ttree/e) = V(\ttree)\setminus\{u,v\}\amalg\{u\equiv v\}$.)
To perform this construction, we put the factor $\underline{01}_{u\equiv v}$ associated to the merged vertex $u\equiv v$
in front position of the tensor product $\bigotimes_{x\in V(\ttree/e)}\underline{01}^{\sharp}_x$,
using a tensor permutation if necessary (recall simply that such a tensor permutation
produces a sign). The other option is to transport the map $\underline{01}_{u\equiv v}\mapsto\underline{01}_u\otimes\underline{01}_v$
over some tensors in order to reach the factor $\underline{01}_{u\equiv v}$
inside the tensor product $\bigotimes_{x\in V(\ttree/e)}\underline{01}^{\sharp}_x$.
This operation produces a sign too (because this map $\underline{01}_{u\equiv v}\mapsto\underline{01}_u\otimes\underline{01}_v$
has degree $1$ and therefore the permutation of this map with a tensor returns a sign).
Both procedures yield equivalent results.
Then we take
\begin{equation*}
\partial = \sum_{(\ttree,e)}\partial_{(\ttree,e)},
\end{equation*}
the sum of these maps $\partial_{(\ttree,e)}$ associated to the pairs $(\ttree,e)$.

The next lemma implies that this map $\partial$ defines a twisting differential on $\DGB^c(\AOp)(r)$,
so that we can provide $\DGB^c(\AOp)(r)$ with a dg module structure with the sum $\delta+\partial: \DGB^c(\AOp)(r)\rightarrow\DGB^c(\AOp)(r)$
as total differential, where $\delta: \DGB^c(\AOp)(r)\rightarrow\DGB^c(\AOp)(r)$ denotes the differential
induced by the internal differential of the objects $\AOp(\ttree)$
in $\DGB^c(\AOp)(r)$.

We now fix a pointed shuffle decomposition $\{1<\dots<r\} = \{i_1<\dots<\widehat{i_p}<\dots<i_k\}\amalg\{j_1<\dots<j_l\}$
associated to the composition scheme of an operadic composition product $\circ_{i_p}$,
which we can also represent by a tree with two vertices $\gammatree$.
For a pair of trees $\ttree\in\Tree(\{i_1<\dots<i_k\})$ and $\stree\in\Tree(\{j_1<\dots<j_l\})$,
we consider the map
\begin{equation*}
\circ_{i_p}^{\stree,\ttree}: \DGSigma^{-\sharp V(\stree)}\AOp(\stree)\otimes\DGSigma^{-\sharp V(\ttree)}\AOp(\ttree)
\rightarrow\DGSigma^{-\sharp V(\stree\circ_{i_p}\ttree)}\AOp(\stree\circ_{i_p}\ttree)
\end{equation*}
yielded by the Segal map $i_{\stree\circ_{i_p}\ttree}: \AOp(\stree)\otimes\AOp(\ttree)\rightarrow\AOp(\stree\circ_{i_p}\ttree)$
associated to the decomposition of the grafting $\stree\circ_{i_p}\ttree = \lambda_{\gammatree}(\stree,\ttree)$,
together with the obvious tensor identity $\bigotimes_{u\in V(\stree)}\underline{01}^{\sharp}_u\otimes\bigotimes_{v\in V(\ttree)}\underline{01}^{\sharp}_v
= \bigotimes_{x\in V(\stree\circ_{i_p}\ttree)}\underline{01}^{\sharp}_x$,
which we deduce from the relation $V(\stree\circ_{i_p}\ttree) = V(\stree)\amalg V(\ttree)$.
Then we consider the composition product
\begin{equation*}
\circ_{i_p}: \DGB^c(\AOp)(\{i_1<\dots<i_k\})\otimes\DGB^c(\AOp)(\{j_1<\dots<j_l\})\rightarrow\DGB^c(\AOp)(\{1<\dots<r\})
\end{equation*}
defined by the sum of these maps $\circ_{i_p}^{\stree,\ttree}$ associated to the pairs $(\stree,\ttree)$. We check in a forthcoming lemma
that these operations preserve our differentials, and hence, provide our object with well-defined dg module operations.
\end{constr}

We first check the validity of the definition of our twisting differential on the cobar construction.
We have the following more precise statement.

\begin{lemm}\label{lemma:differential-bar-construction}
We have the relations $\partial^2 = \delta\partial + \partial\delta = 0$ on $\DGB^c(\AOp)(r)$, where $\delta: \DGB^c(\AOp)(r)\rightarrow\DGB^c(\AOp)(r)$
denotes the differential induced by the internal differential of the objects $\AOp(\ttree)$ (as we explained in the above construction).
\end{lemm}

\begin{proof}
The identity $\delta\partial +\partial\delta = 0$ reduces to the termwise relations $\delta\partial_{(\ttree,e)} + \partial_{(\ttree,e)}\delta = 0$,
which in turn, are reformulation of the fact that the coproduct operators $\rho_{\ttree\rightarrow\ttree/e}$
are morphisms of dg modules.

We check that the relation $\partial^2 = 0$ holds on each summand $\DGSigma^{-\sharp V(\ttree)}\AOp(\ttree)\subset\DGB^c(\AOp)$
in the target of our map.
Note that $\mathring{E}(\ttree/e) = \mathring{E}(\ttree)\setminus\{e\}$.
Hence the components of $\partial^2$ which land in $\DGSigma^{-\sharp V(\ttree)}\AOp(\ttree)$
are defined on a summand $\DGSigma^{-\sharp V(\stree)}\AOp(\stree)\subset\DGB^c(\AOp)$
such that $\stree = \ttree/\{e,f\}$,
for some $e,f\in\mathring{E}(\ttree)$, $e\not=f$, and are given by the composites:
\begin{equation*}
\xymatrix{ & \DGSigma^{-\sharp V(\ttree)+1}\AOp(\ttree/e)\ar[dr]^{\partial_{(\ttree,e)}} & \\
\DGSigma^{-\sharp V(\ttree)+2}\AOp(\ttree/\{e,f\})\ar[ur]^{\partial_{(\ttree/e,f)}}\ar[dr]_{\partial_{(\ttree/f,e)}} && \DGSigma^{-\sharp V(\ttree)}\AOp(\ttree) \\
& \DGSigma^{-\sharp V(\ttree)+1}\AOp(\ttree/e)\ar[ur]_{\partial_{(\ttree,f)}} & }.
\end{equation*}
Both composites are given by composite coproduct operators
such that
\begin{equation*}
\rho_{\ttree\rightarrow\ttree/e}\circ\rho_{\ttree/\{e,f\}\rightarrow\ttree/e} = \rho_{\ttree/\{e,f\}\rightarrow\ttree} = \rho_{\ttree\rightarrow\ttree/f}\circ\rho_{\ttree/\{e,f\}\rightarrow\ttree/f}
\end{equation*}
by functoriality of the object $\AOp$. We just note that the above composites differ by an order of tensor factors $\underline{01}$ to conclude that these composite coproducts
occur with opposite signs in $\partial^2$, and hence, cancel each other. The conclusion follows.
\end{proof}

We check the validity of our definition of the composition products on the cobar construction. We have the following more precise statement.

\begin{lemm}\label{lemma:bar-construction-product-operad}
The composition products $\circ_{i_p}$ of Construction~\ref{construction:bar-complex} preserve both the cobar differential $\partial$ and the differential $\delta$
induced by the internal differential of the object $\AOp$ on the cobar construction $\DGB^c(\AOp)$,
and hence, form morphisms of dg modules.
\end{lemm}

\begin{proof}
We prove that $\circ_{i_p}$ commutes with the differentials on each summand $\AOp(\stree)\otimes\AOp(\ttree)$
of the tensor product $\DGB^c(\AOp)(\{i_1<\dots<i_k\})\otimes\DGB^c(\AOp)(\{j_1<\dots<j_l\})$,
where $\stree\in\Tree(\{i_1<\dots<i_k\})$, $\ttree\in\Tree(\{j_1<\dots<j_l\})$.
We just use that the Segal maps, which induce our composition product componentwise, are morphisms of dg modules to conclude that $\circ_{i_p}$
preserves the internal differentials on our objects.
We focus on the verification that $\circ_{i_p}$ preserves the cobar differential.
We set $\thetatree = \stree\circ_{i_p}\ttree$ and we consider a tree $\thetatree'$ such that $\thetatree = \thetatree'/e$ for an edge $e\in\mathring{E}(\thetatree')$
which is contracted in a vertex in $\thetatree$. This vertex comes either from $\stree$
or from $\ttree$. In the first case, we have $\thetatree' = \stree'\circ_{i_p}\ttree$, for a tree $\stree'\subset\thetatree'$ such that $e\in\mathring{E}(\stree')$
and $\stree'/e = \stree$ (we can actually identify $\stree'\subset\thetatree'$
with the pre-image of the subtree $\stree\subset\thetatree$ under the morphism $\thetatree'\rightarrow\thetatree'/e = \thetatree$).
In the second case, we have $\thetatree' = \stree\circ_{i_p}\ttree'$,
for a tree $\ttree'\subset\thetatree'$ such that $e\in\mathring{E}(\ttree')$ and $\ttree'/e = \ttree$ (we can then identify $\ttree'\subset\thetatree'$
with the pre-image of the subtree $\ttree\subset\thetatree$
under our edge contraction morphism $\thetatree'\rightarrow\thetatree'/e = \thetatree$).
The compatibility between the Segal maps and the coproduct operators imply that the following diagrams commute:
\begin{equation*}
\xymatrix{ \AOp(\stree'/e)\otimes\AOp(\ttree)\ar[d]_-{\rho_{\stree'\rightarrow\stree'/e}\otimes\id}\ar[r]^{i_{\stree'/e\circ_{i_p}\ttree}} &
\AOp(\stree'/e\circ_{i_p}\ttree)\ar[d]^-{\rho_{\stree'\rightarrow\stree'/e}} \\
\AOp(\stree')\otimes\AOp(\ttree)\ar[r]^{i_{\stree'\circ_{i_p}\ttree}} &
\AOp(\stree'\circ_{i_p}\ttree) },
\quad\xymatrix{ \AOp(\stree)\otimes\AOp(\ttree'/e)\ar[d]_-{\id\otimes\rho_{\ttree'\rightarrow\ttree'/e}}\ar[r]^{i_{\stree\circ_{i_p}\ttree'/e}} &
\AOp(\stree\circ_{i_p}\ttree'/e)\ar[d]^-{\rho_{\ttree'\rightarrow\ttree'/e}} \\
\AOp(\stree)\otimes\AOp(\ttree')\ar[r]^{i_{\stree'\circ_{i_p}\ttree}} &
\AOp(\stree\circ_{i_p}\ttree') }.
\end{equation*}
This yields the relation $\partial_{(\stree'\circ_{i_p}\ttree,e)}\circ\circ_{i_p}^{\stree'/e,\ttree} = \circ_{i_p}^{\stree',\ttree}\circ(\partial_{(\stree',e)}\otimes\id)$
in the first case
and the relation $\partial_{(\stree\circ_{i_p}\ttree',e)}\circ\circ_{i_p}^{\stree,\ttree'/e} = \circ_{i_p}^{\stree,\ttree'}\circ(\id\otimes\partial_{(\ttree',e)})$
in the second case.
By summing these equalities with suitable suspension factors we get that $\partial$
defines a derivation with respect to $\circ_{i_p}$.
\end{proof}

We immediately deduce from the associativity of the Segal maps that the composition products of Construction~\ref{construction:bar-complex}
satisfy the associativity relations of the composition products of an operad.
We therefore get the following concluding statement:

\begin{thm-defn}
The collection $\DGB^c(\AOp) = \{\DGB^c(\AOp)(r),r>0\}$ equipped with the differential and structure operations
defined in Construction~\ref{construction:bar-complex}
forms a shuffle operad in dg modules.
This operad $\DGB^c(\AOp)$ is the cobar construction of the Segal shuffle dg (pre-)cooperad $\AOp$.\qed
\end{thm-defn}


\begin{remark}
In the case of an ordinary dg cooperad, we just retrieve the cobar construction functor of the classical theory of operads $\COp\mapsto\DGB^c(\COp)$,
which goes from the category of locally conilpotent (coaugmented) dg cooperads
to the category of (augmented) dg operads. (Recall simply that we drop units from our definitions so that our cooperads are equivalent to the coaugmentation coideal of coaugmented cooperads,
whereas our cobar construction is equivalent to the augmentation ideal of the classical unital cobar construction.)
By classical operad theory, we also have a bar construction functor $\POp\mapsto\DGB(\POp)$, which goes in the converse direction, from the category of (augmented) dg operads
to the usual category of locally conilpotent (coaugmented) dg cooperads.
For a locally conilpotent Segal shuffle dg pre-cooperad $\AOp$, we actually have an identity $\DGW^c_{dec}(\AOp) = \DGB\DGB^c(\AOp)$,
where we consider the decomposed $W$-construction of the previous subsection and the cobar operad $\DGB^c(\AOp)$,
such as defined in this subsection.
\end{remark}

\subsection{The strict Segal $E_\infty$-Hopf cooperad associated to an operad in simplicial sets}

In this subsection, we study a correspondence between the category of operads in simplicial sets
and the category of Segal $E_\infty$-Hopf cooperads.
In a first step, we explain the construction of the structure of a strict Segal $E_\infty$-Hopf cooperad
on the normalized cochain complex of a simplicial operad.

\begin{constr}\label{definition:cooperad-from-simplicial-operad}
Let $\POp$ be a (symmetric) operad in the category of simplicial sets $\simp\Set$.
Let $\DGN^*: \simp\Set^{op}\rightarrow\EAlg$ be the normalized cochain complex functor, where we consider on $\DGN^*(-)$ the $\EOp$-algebra structure
defined in~\cite{BergerFresse} (see also our overview in~\S\ref{sec:Barratt-Eccles-operad}).
For a tree $\ttree\in\Tree(r)$, we set $\POp(\ttree) = \prod_{v\in V(\ttree)}\POp(\rset_v)$.
Then we set:
\begin{equation*}
\AOp_{\POp}(\ttree) = \DGN^*(\POp(\ttree)).
\end{equation*}
We equip this collection with the coproduct operators $\rho_{\ttree\rightarrow\stree}: \AOp_{\POp}(\stree)\rightarrow\AOp_{\POp}(\ttree)$
induced by the treewise composition products on~$\POp$
and with the facet operators $i_{\sigmatree,\stree}: \AOp_{\POp}(\sigmatree)\rightarrow\AOp_{\POp}(\ttree)$
induced by the projection maps $\prod_{v\in V(\ttree)}\POp(\rset_v)\rightarrow\prod_{v\in V(\sigmatree)}\POp(\rset_v)$,
for $\sigmatree\subset\ttree$.
We also have operators $s^*: \AOp_{\POp}(s\ttree)\rightarrow\AOp_{\POp}(\ttree)$,
associated to the permutations $s\in\Sigma_r$,
which are induced by the corresponding action of the permutations $s_*: \POp(\ttree)\rightarrow\POp(s\ttree)$
on our simplicial sets.
\end{constr}

\begin{prop}\label{proposition:cooperad-from-simplicial-set}
Let $\POp$ be an operad in $\simp\Set$. The collection $\AOp_{\POp}(\ttree)$, equipped with the coproduct operators and facet operators defined in the above construction,
and with our action of permutations,
defines a strict Segal $E_\infty$-Hopf symmetric cooperad. This construction is functorial in $\POp$.
\end{prop}

\begin{proof}
The associativity of the coproduct maps is a direct consequence of the associativity of the products in $\POp$ and of the functoriality of the normalized cochain complex $\DGN^*(-)$.
The associativity of the facet operators and their compatibility with the coproducts also follows from the counterpart of these relations at the simplicial set level
and from the functoriality of the normalized cochain complex $\DGN^*(-)$,
and similarly as regards the compatibility between the action of permutations, the coproduct operators and the facet operators.

We only need to prove the Segal condition. By Proposition \ref{proposition:Segal-condition}, we can reduce our verifications to the case
of a decomposition $\ttree = \lambda_{\gammatree}(\sigmatree_u,\sigmatree_v)$
over a tree with two internal vertices $\gammatree$.
We observe that, as our facet operators are induced by simplicial projections, there is a commutative diagram:
\begin{equation*}
\xymatrixcolsep{2pc}\xymatrix{ \AOp_{\POp}(\sigmatree_u)\vee\AOp_{\POp}(\sigmatree_v) = \DGN^*(\POp(\sigmatree_u))\vee\DGN^*(\POp(\sigmatree_v))
\ar[r]^-{i_{\lambda_{\gammatree}(\sigmatree_u,\sigmatree_v)}} &
\AOp_{\POp}(\ttree) = \DGN^*(\POp(\sigmatree_u)\times\POp(\sigmatree_v)) \\
& \DGN^*(\POp(\sigmatree_u))\otimes\DGN^*(\POp(\sigmatree_v))\ar[lu]^-{\EM}\ar[u]_-{\AW} },
\end{equation*}
where $\AW$ denotes the classical Alexander--Whitney product $\AW: \DGN^*(X)\otimes\DGN^*(Y)\rightarrow\DGN^*(X\times Y)$.
Both maps $\EM$ and $\AW$ are weak-equivalences (the map $\EM$ by Proposition~\ref{claim:Barratt-Eccles-algebra-coproducts}).
Hence $i_{\lambda_{\gammatree}(\sigmatree_u,\sigmatree_v)}$ is also a weak-equivalence.
The conclusion follows.
\end{proof}

The normalized cochain complex functor $\DGN^*: \simp\Set^{op}\rightarrow\EAlg$ admits a left adjoint $\DGG: \EAlg\rightarrow\simp\Set^{op}$,
which is defined by the formula $\DGG(A) = \Mor_{\EAlg}(A,\DGN^*(\Delta^{\bullet}))$,
for all $A\in\EAlg$ (see~\cite{Mandell} for a variant of this construction).
This pair of adjoint functors also defines a Quillen adjunction between the Kan model category of simplicial sets
and the model category of $\EOp$-algebras (we refer to~\cite{BergerFresse} for the definition of the model structure on the category of $\EOp$-algebras).
By the main result of~\cite{Mandell}, if we take $\kk = \bar{\FF}_p$ as ring of coefficients, then the object $L\DGG(\DGN^*(X))$,
which we obtain by applying the normalized cochain complex functor to a connected simplicial sets $X\in\simp\Set$
and by going back to simplicial sets by using the derived functor $L\DGG(-)$
of our left adjoint $\DGG(-) = \Mor_{\EAlg}(-,\DGN^*(\Delta^{\bullet}))$,
is weakly-equivalent to the $p$-completion of the space~$X$ (under standard nilpotence and cohomological finiteness assumptions).

We aim to examine the application of this adjoint functor construction to our Segal $E_\infty$-Hopf cooperads.
We need to form resolutions of our objects in the category of Segal $E_\infty$-Hopf symmetric cooperads
with cofibrant $E_\infty$-algebras
as components to give a sense to the application of the derived functor $\DGG$
to our objects. We rely on the following observation.

\begin{lemm}
Let $R: \EAlg\rightarrow\EAlg$ be a functorial cofibrant replacement functor on the model category of $\EOp$-algebras (which exists because
the category of $\EOp$-algebras is cofibrantly generated).
If $\AOp\in\EOp\Hopf\Sigma\SegOp^c$ is a Segal $E_\infty$-Hopf symmetric pre-cooperad, then the collection $\ROp(\ttree) = R\AOp(\ttree)\in\EAlg$, $\ttree\in\Tree(r)$, $r>0$,
inherits a Segal $E_\infty$-Hopf symmetric pre-cooperad structure
by functoriality. If $\AOp$ satisfies the Segal condition and is such that the objects $\AOp(\ttree)$ are cofibrant as dg modules,
then $\ROp = R\AOp$ also satisfies the Segal condition,
and hence, forms a Segal $E_\infty$-Hopf symmetric cooperad.
\end{lemm}

\begin{proof}
The first claim, that the collection $\ROp(\ttree) = R\AOp(\ttree)\in\EAlg$, $\ttree\in\Tree(r)$, $r>0$,
inherits a Segal $E_\infty$-Hopf symmetric pre-cooperad structure,
is immediate.
To establish that $R\AOp$ also satisfies the Segal condition, we just use the result of Proposition~\ref{claim:Barratt-Eccles-algebra-coproducts}
(we merely need the extra assumption that the objects $\AOp(\ttree)$ are cofibrant as dg modules, like the cofibrant $\EOp$-algebras $R\AOp(\ttree)$,
to get that the weak-equivalences $R\AOp(\ttree)\xrightarrow{\sim}\AOp(\ttree)$
induce a weak-equivalence when we pass to a tensor product).
\end{proof}

The application of the functor $\DGG(-)$ to the Segal $E_\infty$-Hopf symmetric cooperad $R\AOp$
defined in this proposition returns a Segal symmetric operad in the category of simplicial sets,
where a Segal symmetric operad is the obvious counterpart,
in the category of simplicial sets,
of our notion of a Segal $E_\infty$-Hopf symmetric cooperad (we just have to dualize the definition).

We mention in the introduction of this paper that such structures are close to Cisinski--Moerdijk's notion of a dendroidal Segal space (see~\cite{CisinskiMoerdijkI}).
We can also check that every Segal symmetric operad
is weakly-equivalent to an operad
in the ordinary sense.
We can proceed along the lines of Cisinski--Moerdijk's dendroidal nerve construction or use a counterpart, in the category of simplicial sets,
of the $W$-constructions of~\S\ref{subsection:forgetful-strict}.
We then get that every Segal symmetric operad in the category of simplicial sets $\POp$
is connected to an ordinary symmetric operad $\DGW_{dec}(\POp)$
by a zigzag of natural weak-equivalences of Segal symmetric operads:
\begin{equation*}
\POp\xleftarrow{\sim}\DGW(\POp)\xrightarrow{\sim}\DGW_{dec}(\POp).
\end{equation*}
We record the following straightforward consequence of our statements to conclude this subsection.

\begin{prop}
We assume $\kk = \bar{\FF}_p$ and we consider normalized cochains with coefficients in this field, so that $\DGN^*(X,\bar{\FF}_p)$
represents the Mandell model of the space $X$ in the category of $\EOp$-algebras.
We use the notation $\DGG(-)$ for the adjoint functor of $\DGN^*(-)$, from the category of $\EOp$-algebras
to the category of simplicial sets.
\begin{enumerate}
\item
Let $\POp$ be a symmetric operad in simplicial sets. We consider the Segal $E_\infty$-Hopf symmetric cooperad $\AOp_{\POp} = \DGN^*(\POp(-))$
given by the resolution of Proposition~\ref{proposition:cooperad-from-simplicial-set} and its resolution $R\AOp$
in the category of Segal $E_\infty$-Hopf symmetric cooperads.
If $\POp$ consists of connected nilpotent spaces with a cohomology $\DGH^*(-,\bar{\FF}_p)$ of finite dimension degreewise,
then the object $\DGG(R\AOp_{\POp})$ defines a Segal symmetric operad in simplicial sets such that $\DGG(R\AOp_{\POp})(\ttree) = \POp(\ttree)^{\wedge}_p$,
for each tree $\ttree\in\Tree(r)$, $r>0$, where we consider the $p$-completion of the space $\POp(\ttree)$.
Thus, if we apply our functor $\DGW_{dec}(-)$ to this Segal symmetric operad $\DGG(R\AOp_{\POp})$
then we get a model of the $p$-completion $\POp(-)^{\wedge}_p$
in the category of ordinary operads.
\item
If we have a pair of symmetric operads in simplicial sets $\POp$ and $\QOp$ which satisfy these connectedness, nilpotence and cohomological finiteness assumptions,
then the existence of a weak-equivalence $\AOp_{\POp}\sim\AOp_{\QOp}$ at the level of the category of Segal $E_\infty$-Hopf symmetric cooperads
implies the existence of a weak-equivalence $\POp(-)^{\wedge}_p\sim\QOp(-)^{\wedge}_p$
at the level of the $p$-completion of our operads.\qed
\end{enumerate}
\end{prop}

\section{The category of homotopy Segal $E_\infty$-Hopf cooperads}\label{sec:homotopy-segal-cooperads}

We study the homotopical version of Segal $E_\infty$-Hopf cooperads in this section. The informal idea is to require that the coproduct operators
satisfy the associativity relation $\rho_{\ttree\rightarrow\utree}\circ\rho_{\utree\rightarrow\stree} = \rho_{\ttree\rightarrow\stree}$
only up to homotopy. We construct a model to shape the homotopies that govern these associativity relations
and the compatibility relation between these homotopies and the facet operators.
We explain our definition of homotopy Segal $E_\infty$-Hopf cooperad with full details in the first subsection of this section.
We then examine the forgetting of $E_\infty$-algebra structures in the definition of this structure, as in a study of strict homotopy $E_\infty$-Hopf cooperads,
and we examine the application of the cobar construction to homotopy Segal cooperads. We devote our second and third subsection to these topics.
We study a homotopy version of morphisms of Segal $E_\infty$-Hopf cooperads afterwards, in a fourth subsection,
and we eventually prove, in the fifth subsection, that every homotopy Segal $E_\infty$-Hopf cooperad
is weakly equivalent to a strict Segal $E_\infty$-Hopf cooperad.

We still take the category of algebras $\EAlg$ associated to the chain Barratt--Eccles operad $\EOp$
as a model of a category of $E_\infty$-algebras in dg modules
all along this section.

\subsection{The definition of homotopy Segal $E_\infty$-Hopf cooperads}

We define homotopy Segal $E_\infty$-Hopf cooperads essentially along the same plan as strict Segal $E_\infty$-Hopf cooperads.
We still define a notion of homotopy Segal $E_\infty$-Hopf pre-cooperad beforehand, as a structure equipped with operations that underlie the definition of a homotopy Segal $E_\infty$-Hopf cooperad.
We define the notion of a homotopy Segal $E_\infty$-Hopf cooperad afterwards, as a Segal $E_\infty$-Hopf pre-cooperad equipped with facet operators that satisfy the Segal condition.
In comparison to the definition of strict Segal $E_\infty$-Hopf cooperads, we mainly consider higher coproduct operators,
which we use to govern the homotopical associativity of the composition of coproduct operators.
We use a complex of cubical cochains, which we define from the cochain algebra of the interval $I = \DGN^*(\Delta^1)$, in our definition scheme
of this model of higher coproduct operators.
We explain the definition of the structure of this complex of cubical cochains in the next preliminary construction.
We address the definition of homotopy Segal $E_\infty$-Hopf cooperads afterwards.
We explain a definition of strict morphism of homotopy Segal $E_\infty$-Hopf cooperads to complete the objectives of this subsection. (These strict morphisms
are particular cases of the homotopy morphisms that we define in the fourth subsection of the section.)

\begin{constr}\label{constr:cubical-cochain-algebras}
We define our cubical cochain algebras $I^k$, $k\in\NN$, as the tensor powers of the cochain algebra of the interval:
\begin{equation*}
I^k = \underbrace{I\otimes\dots\otimes I}_k,\quad I = \DGN^*(\Delta^1).
\end{equation*}
We make the Barratt--Eccles operad acts on these tensor products through its diagonal $\Delta: \EOp\rightarrow\EOp\otimes\EOp$
and its action on each factor $I = \DGN^*(\Delta^1)$
so that our objects $I^k$ inherit a natural $\EOp$-algebra structure (see Construction~\ref{constr:Barratt-Eccles-diagonal}
and Construction~\ref{constr:cubical-cochain-connection}).

We can identify the object $I^k$ with the cellular cochain complex of the $k$-cube $\square^k$.
We define face operators $d^i_{\epsilon}: I^k\rightarrow I^{k-1}$, for $1\leq i\leq k$ and $\epsilon\in\{0,1\}$,
and degeneracy operators $s^i: I^{k-1}\rightarrow I^k$, for $0\leq i\leq k$,
which reflect classical face and degeneracy operations
on the topological cubes $\square^k$.
We determine these operators from the maps
\begin{equation*}
\DGN^*(\Delta^1)\xrightarrow{d_{\epsilon}}\DGN^*(\Delta^0) = \kk,\quad\epsilon\in\{0,1\},
\quad\text{and}
\quad\kk = \DGN^*(\Delta^0)\xrightarrow{s_0}\DGN^*(\Delta^1),
\end{equation*}
induced by simplicial coface and codegeneracy operators of the $1$-simplex $d^{\epsilon}: \Delta^0\rightarrow\Delta^1$
and $s^0: \Delta^1\rightarrow\Delta^0$,
and from the connection of Construction~\ref{constr:cubical-cochain-connection}
\begin{equation*}
\DGN^*(\Delta^1)\xrightarrow{\nabla^*}\DGN^*(\Delta^1)\otimes\DGN^*(\Delta^1),
\end{equation*}
which we associate to the simplicial map $\min: \Delta^1\times\Delta^1\rightarrow\Delta^1$
such that $\min: (s,t)\mapsto\min(s,t)$ on topological realization.

Recall briefly that we have an identity $\DGN^*(\Delta^1) = \kk\underline{0}^{\sharp}\oplus\kk\underline{1}^{\sharp}\oplus\kk\underline{01}^{\sharp}$,
where $\underline{0}^{\sharp}$ and $\underline{1}^{\sharp}$ denote elements of (cohomological) degree $0$,
dual to the classes of the vertices $\underline{0},\underline{1}\in\Delta^1$
in the normal chain complex $\DGN_*(\Delta^1)$,
whereas $\underline{01}^{\sharp}$ denotes an element of (cohomological) degree $1$
dual to the class of the fundamental simplex $\underline{01}\in\Delta^1$. This cochain complex $\DGN^*(\Delta^1)$
is equipped with the differential such that $\delta(\underline{0}^{\sharp}) = - \underline{01}^{\sharp}$
and $\delta(\underline{1}^{\sharp}) = - \underline{01}^{\sharp}$.
The face operators $d_0,d_1: \DGN^*(\Delta^1)\rightarrow\DGN^*(\Delta^0) = \kk$
are defined by the formulas $d_0(\underline{0}^{\sharp}) = 0$, $d_0(\underline{1}^{\sharp}) = 1$, $d_1(\underline{0}^{\sharp}) = 1$, $d_1(\underline{1}^{\sharp}) = 0$,
and $d_0(\underline{01}^{\sharp}) = d_1(\underline{01}^{\sharp}) = 0$,
while the degeneracy operator $s_0: \kk = \DGN^*(\Delta^0)\rightarrow\DGN^*(\Delta^1)$ is defined by the formula $s_0(1) = \underline{0}^{\sharp}+\underline{1}^{\sharp}$.
We refer to Construction~\ref{constr:cubical-cochain-connection} for the explicit definition of the connection $\nabla^*$ on basis elements.
We mainly use that this connection satisfies relations of the form $(d_1\otimes\id)\circ\nabla^* = s_0 d_1 = (\id\otimes d_1)\circ\nabla^*$
and $(d_0\otimes\id)\circ\nabla^* = \id = (\id\otimes d_0)\circ\nabla^*$,
which reflect the identities $\min\circ(d^1\times\id) = d^1 s^0 = (\min\circ\id)\times d^1$
and $\min\circ(d^0\times\id) = \id = (\min\circ\id)\times d^0$
at the topological level.

We number our factors from right to left in our cubical cochain algebras
\begin{equation*}
I^k = \underset{k}{\DGN^*(\Delta^1)}\otimes\dots\otimes\underset{1}{\DGN^*(\Delta^1)}
\end{equation*}
and we use this numbering convention to index our face and degeneracy operators (the superscript in the notation of the faces $d^i_{\epsilon}$
indicates the factor of this tensor product on which we apply a simplicial face operator $d_{\epsilon}$
and the superscript in the notation of the degeneracies $s^j$ similarly indicates the factor of the tensor product $I^{k-1}$
on which we apply a simplicial degeneracy or a connection operator).
We precisely define the face operators of our cubical cochain algebras $d^i_{\epsilon}: I^k\rightarrow I^{k-1}$, $i = 1,\dots,k$, $\epsilon = 0,1$, by the formula:
\begin{align*}
d^i_{\epsilon} & = \id^{\otimes k-i}\otimes d_{\epsilon}\otimes\id^{\otimes i-1}\\
\intertext{and our degeneracy operators $s^j: I^{k-1}\rightarrow I^k$, $j = 0,\dots,k$, by the formulas:}
s^0 & = \id^{\otimes k-1}\otimes s_0, \\
s^j & = \id^{\otimes k-j-1}\otimes\nabla^*\otimes\id^{\otimes j-1}\quad\text{for $j = 1,\dots,k-1$}, \\
s^k & = s_0\otimes\id^{\otimes k-1},
\end{align*}
\end{constr}

We are now ready to define our main objects.

\begin{defn}\label{definition:homotopy-E-infinity-cooperad}
We call homotopy Segal $E_\infty$-Hopf shuffle pre-cooperad the structure defined by a collection of $\EOp$-algebras
\begin{equation*}
\AOp(\ttree)\in\EAlg,\quad\text{$\ttree\in\Tree(r)$, $r>0$},
\end{equation*}
equipped with
\begin{itemize}
\item
homotopy coproduct operators
\begin{equation*}
\rho_{\ttree\rightarrow\ttree_k\rightarrow\dots\rightarrow\ttree_1\rightarrow\stree}: \AOp(\stree)\rightarrow\AOp(\ttree)\otimes I^k,
\end{equation*}
defined as morphisms of $\EOp$-algebras, for all composable sequences of tree morphisms $\ttree\rightarrow\ttree_k\rightarrow\dots\rightarrow\ttree_1\rightarrow\stree$, $k\geq 0$,
and which satisfy the face and degeneracy relations depicted in Figure~\ref{homotopy-E-infinity-cooperad:0-faces}-\ref{homotopy-E-infinity-cooperad:degeneracies},
\item
together with facet operators
\begin{equation*}
i_{\sigmatree,\stree}: \AOp(\sigmatree)\rightarrow\AOp(\stree),
\end{equation*}
defined as morphisms of $\EOp$-algebras again, for all subtree embeddings $\sigmatree\subset\stree$,
and which satisfy the usual functoriality relations $i_{\stree,\stree} = \id_{\stree}$
and $i_{\thetatree,\stree}\circ i_{\sigmatree,\thetatree} = i_{\sigmatree,\stree}$ for all $\sigmatree\subset\thetatree\subset\stree$,
\item
and where we also assume the verification of compatibility relations between the facet operators and the coproduct operators,
which we express by the commutativity of the diagram of Figure~\ref{homotopy-E-infinity-cooperad:facet-compatibility}.
\end{itemize}
We again say that a Segal $E_\infty$-Hopf shuffle pre-cooperad $\AOp$ is a Segal $E_\infty$-Hopf shuffle cooperad when the facet operators satisfy a Segal condition,
which reads exactly as in the context of strict Segal $E_\infty$-Hopf shuffle cooperads:
\begin{enumerate}
\item[(*)]
The facet operators $i_{\sigmatree_v,\ttree}: \AOp(\sigmatree_v)\rightarrow\AOp(\ttree)$
associated to a tree decomposition $\ttree = \lambda_{\stree}(\sigmatree_v, v\in V(\stree))$
induce a weak-equivalence
\begin{equation*}
i_{\lambda_{\stree}(\sigmatree_*)}: \bigvee_{v\in V(\stree)}\AOp(\sigmatree_v)\xrightarrow{\sim}\AOp(\ttree)
\end{equation*}
when we pass to the coproduct of the objects $\AOp(\sigmatree_v)$
in the category of $\EOp$-algebras.
\end{enumerate}
We also define a homotopy Segal $E_\infty$-Hopf symmetric (pre-)cooperad as a homotopy Segal $E_\infty$-Hopf shuffle pre-cooperad $\AOp$
equipped with an action of the permutations such that $s^*: \AOp(s\ttree)\rightarrow\AOp(\ttree)$,
for $s\in\Sigma_r$ and $\ttree\in\Tree(r)$,
and which intertwine the facets and the homotopy coproduct operators attached our object.
\end{defn}

Note that the statement of Proposition \ref{proposition:Segal-condition}, where we reduce the verification of the Segal condition to various particular cases,
obviously holds for homotopy Segal $E_\infty$-Hopf cooperads too.

\begin{figure}[p]
\ffigbox
{\caption{The compatibility of homotopy coproducts with $0$-faces.
The diagram commutes for all sequences of composable tree morphisms $\ttree\rightarrow\ttree_k\rightarrow\dots\rightarrow\ttree_1\rightarrow\stree$
and for all $1\leq i\leq k$, where $\widehat{\ttree_i}$ means that we delete the node $\ttree_i$
and we replace the morphisms $\ttree_{i+1}\rightarrow\ttree_i\rightarrow\ttree_{i-1}$
by their composite $\ttree_{i+1}\rightarrow \ttree_{i-1}$.}\label{homotopy-E-infinity-cooperad:0-faces}}
{\centerline{\xymatrixcolsep{7pc}\xymatrix{ \AOp(\stree)
\ar[r]^{\rho_{\ttree\rightarrow\ttree_k\rightarrow\dots\rightarrow\ttree_1\rightarrow\stree}}
\ar[dr]_{\rho_{\ttree\rightarrow\dots\widehat{\ttree_i}\dots\rightarrow\stree}} &
\AOp(\ttree)\otimes I^k\ar[d]^{\id\otimes d^i_0} \\
& \AOp(\ttree)\otimes I^{k-1} }
}}
\ffigbox
{\caption{The compatibility of homotopy coproducts with $1$-faces.
The diagram commutes for all sequences of composable tree morphisms $\ttree\rightarrow\ttree_k\rightarrow\dots\rightarrow\ttree_1\rightarrow\stree$
and for all $1\leq i\leq k$.}\label{homotopy-E-infinity-cooperad:1-faces}}
{\centerline{\xymatrixcolsep{7pc}\xymatrix{ \AOp(\stree)
\ar[r]^{\rho_{\ttree\rightarrow\ttree_k\rightarrow\dots\rightarrow\ttree_1\rightarrow\stree}}
\ar[dd]_{\rho_{\ttree_i\rightarrow\dots\rightarrow\ttree_1\rightarrow\stree}} &
\AOp(\ttree)\otimes I^k\ar[d]^{\id\otimes d^i_1} \\
& \AOp(\ttree)\otimes I^{k-1} \\
\AOp(\ttree_i)\otimes I^{i-1}\ar[r]_-{\rho_{\ttree\rightarrow\ttree_k\rightarrow\dots\rightarrow\ttree_i}\otimes\id} &
\AOp(\ttree)\otimes I^{k-i}\otimes I^{i-1}\ar[u]_{\simeq} }
}}
\ffigbox
{\caption{The compatibility of homotopy coproducts with degeneracies.
The diagram commutes for all sequences of composable tree morphisms $\ttree\rightarrow\ttree_{k-1}\rightarrow\dots\rightarrow\ttree_1\rightarrow\stree$
and for all $0\leq j\leq k$ (with the convention that $\ttree_0 = \stree$ in the case $j = 0$
and $\ttree_k = \ttree$ in the case $j = k$).}\label{homotopy-E-infinity-cooperad:degeneracies}}
{\centerline{\xymatrixcolsep{7pc}\xymatrix{ \AOp(\stree)
\ar[r]^-{\rho_{\ttree\rightarrow\ttree_{k-1}\rightarrow\dots\rightarrow\ttree_1\rightarrow\stree}} \ar[rd]_{\rho_{\ttree\rightarrow\ttree_{k-1}\rightarrow\dots\ttree_j=\ttree_j\dots\rightarrow\ttree_1\rightarrow\stree}} &
\AOp(\ttree)\otimes I^{k-1}\ar[d]^{\id\otimes s^j} \\
& \AOp(\ttree)\otimes I^k }
}}
\ffigbox
{\caption{The compatibility between coproducts and facet operators.
The diagram commutes for all sequences of composable tree morphisms $\ttree\xrightarrow{f_k}\ttree_k\xrightarrow{f_{k-1}}\cdots\xrightarrow{f_1}\ttree_1\xrightarrow{f_0}\stree$
and for all subtrees $\sigmatree\subset\stree$.}\label{homotopy-E-infinity-cooperad:facet-compatibility}}
{\centerline{\xymatrixcolsep{15pc}\xymatrix{ \AOp(\stree)\ar[r]^{\rho_{\ttree\rightarrow\ttree_k\rightarrow\dots\rightarrow\ttree_1\rightarrow\stree}} &
\AOp(\ttree)\otimes I^k \\
\AOp(\sigmatree)
\ar@{.>}[r]_-{\rho_{(f_0\dots f_k)^{-1}(\sigmatree)\rightarrow(f_0\dots f_{k-1})^{-1}(\sigmatree)\rightarrow\dots\rightarrow f_0^{-1}(\sigmatree)\rightarrow\sigmatree}}
\ar[u]^{i_{\sigmatree,\stree}} &
\AOp((f_0\dots f_k)^{-1}(\sigmatree))\otimes I^k\ar@{.>}[u]^{i_{(f_0\dots f_k)^{-1}(\sigmatree),\ttree}\otimes\id} }
}}
\end{figure}

\begin{figure}[p]
\ffigbox
{\caption{The coherence of the face and degeneracy relations of the coproduct operators with respect to the relation $s^j d^i_0 = \id$
between the cubical face and degeneracy operators
for $i = j,j+1$.}\label{homotopy-E-infinity-cooperad:0-face-coherence}}
{\centerline{\xymatrixcolsep{10pc}\xymatrix{ & \AOp(\ttree)\otimes I^{k-1}\ar[d]_{\id\otimes s^j}\ar@/^2pc/[dd]^{\id} \\
\AOp(\stree)
\ar[ru]^{\rho_{\ttree\rightarrow\ttree_{k-1}\rightarrow\dots\rightarrow\ttree_1\rightarrow\stree}} \ar[r]_{\rho_{\ttree\rightarrow\dots\rightarrow\ttree_j=\ttree_j\rightarrow\dots\rightarrow\stree}} \ar[rd]_{\rho_{\ttree\rightarrow\ttree_{k-1}\rightarrow\dots\rightarrow\ttree_1\rightarrow\stree}} &
\AOp(\ttree)\otimes I^k\ar[d]_{\id\otimes d^i_0} \\
& \AOp(\ttree)\otimes I^{k-1} }
}}
\ffigbox
{\caption{The coherence of the face and degeneracy relations of the coproduct operators with respect to the relation $s^j d^i_1 = s^0\otimes\id_{I^j}$
for $i = j$.}\label{homotopy-E-infinity-cooperad:1-face-coherence-left}}
{\centerline{\xymatrixcolsep{10pc}\xymatrix{ \AOp(\stree)
\ar@/_3pc/[ddd]_{\rho_{\ttree\rightarrow\dots\rightarrow\stree}}
\ar[r]^{\rho_{\ttree\rightarrow\ttree_{k-1}\rightarrow\dots\rightarrow\ttree_1\rightarrow\stree}} \ar[rd]_{\rho_{\ttree\rightarrow\dots\rightarrow\ttree_j=\ttree_j\rightarrow\dots\rightarrow\stree}}
\ar[dd]^{\rho_{\ttree_j\rightarrow\dots\rightarrow\ttree_1\rightarrow\stree}} &
\AOp(\ttree)\otimes I^{k-1}\ar[d]^{\id\otimes s^j} \\
& \AOp(\ttree)\otimes I^k \ar[d]^{\id\otimes d^j_1} \\
\AOp(\ttree_j)\otimes I^{j-1}
\ar[r]^{\rho_{\ttree\rightarrow\cdots\rightarrow\ttree_j=\ttree_j}\otimes\id}
\ar[dr]_{\rho_{\ttree\rightarrow\dots\rightarrow\ttree_j}} &
\AOp(\ttree)\otimes I^{k-1} \\
\AOp(\ttree)\otimes I^{k-1}\ar[r]_{\id\otimes d^j_1} &
\AOp(\ttree)\otimes I^{k-j-1}\otimes I^{j-1}\ar[u]_{\id\otimes s^0\otimes\id} }
}}
\ffigbox
{\caption{The coherence of the face and degeneracy relations of the coproduct operators with respect to the relation $s^j d^i_1 = \id_{I^{k-j}}\otimes s^j$
for $i = j+1$.}\label{homotopy-E-infinity-cooperad:1-face-coherence-right}}
{\centerline{\xymatrixcolsep{10pc}\xymatrix{ \AOp(\stree)
\ar@/_4pc/[dddd]_{\rho_{\ttree\rightarrow\dots\rightarrow\stree}}
\ar@/_2pc/[ddd]|{\rho_{\ttree_j\rightarrow\dots\rightarrow\stree}}
\ar[r]^{\rho_{\ttree\rightarrow\ttree_{k-1}\rightarrow\dots\rightarrow\ttree_1\rightarrow\stree}} \ar[rd]_{\rho_{\ttree\rightarrow\dots\rightarrow\ttree_j=\ttree_j\rightarrow\dots\rightarrow\stree}}
\ar[dd]^{\rho_{\ttree_j=\ttree_j\rightarrow\dots\rightarrow\ttree_1\rightarrow\stree}} &
\AOp(\ttree)\otimes I^{k-1}\ar[d]^{\id\otimes s^j} \\
& \AOp(\ttree)\otimes I^k\ar[d]^{\id\otimes d^{j+1}_1} \\
\AOp(\ttree_j)\otimes I^j\ar[r]^{\rho_{\ttree\rightarrow\dots\rightarrow\ttree_j}\otimes\id} &
\AOp(\ttree)\otimes I^{k-1} \\
\AOp(\ttree_j)\otimes I^{j-1}
\ar[r]_{\rho_{\ttree\rightarrow\dots\rightarrow\ttree_j}\otimes\id}
\ar[u]_{\id\otimes s^j} &
\AOp(\ttree)\otimes I^{k-j-1}\otimes I^{j-1}\ar[u]_{\id\otimes\id\otimes s^j} \\
\AOp(\ttree)\otimes I^{k-1}\ar[ru]_{\id\otimes d^j_1} }
}}
\end{figure}

\afterpage{\clearpage}

\begin{remark}\label{remark:compatibility-relations}
The faces and degeneracy operators on our complex of cubical cochain algebras
satisfy the following system of identities:
\begin{align*}
d^j_{\epsilon} d^i_{\eta} & = d^i_{\eta} d^{j-1}_{\epsilon},\quad\text{for $i<j$, $\epsilon,\eta\in\{0,1\}$}, \\
s^j d^i_{\epsilon} & = \begin{cases} d^i_{\epsilon} s^{j-1}, & \text{for $i<j$, $\epsilon\in\{0,1\}$}, \\
\id, & \text{for $i = j,j+1$ and $\epsilon = 0$}, \\
s^0\otimes\id_{I^j}, & \text{for $i = j$, $\epsilon = 1$, using $I^{k-1} = I^{k-j-1}\otimes I^j$}, \\
\id_{I^{k-j}}\otimes s^j, & \text{for $i = j+1$, $\epsilon = 1$, using $I^{k-1} = I^{k-j}\otimes I^{j-1}$}, \\
d^{i-1}_{\epsilon} s^j, & \text{for $i>j+1$, $\epsilon\in\{0,1\}$},
\end{cases} \\
s^j s^i & = s^i s^{j+1},\quad\text{for $i\leq j$}.
\end{align*}
We easily check that a multiple application of face and degeneracy relations of Figure~\ref{homotopy-E-infinity-cooperad:0-faces}-\ref{homotopy-E-infinity-cooperad:degeneracies}
for the coproduct operators of homotopy Segal $E_\infty$-Hopf shuffle cooperads lead to coherent results
with respect to these identities.
For instance, we have the double face relation
\begin{multline*}
(\id\otimes d^j_0 d^i_0)\circ\rho_{\ttree\rightarrow\ttree_k\rightarrow\dots\rightarrow\ttree_1\rightarrow\stree}
= \rho_{\ttree\rightarrow\ttree_k\rightarrow\cdots\widehat{\ttree_j}\cdots\widehat{\ttree_i}\cdots\rightarrow\ttree_1\rightarrow\stree}\\
= (\id\otimes d^i_0 d^{j-1}_0)\circ\rho_{\ttree\rightarrow\ttree_k\rightarrow\dots\rightarrow\ttree_1\rightarrow\stree},
\end{multline*}
which expresses the coherence of the face relations of Figure~\ref{homotopy-E-infinity-cooperad:0-faces}
with respect the double $0$-face identity $d^j_0 d^i_0 = d^i_0 d^{j-1}_0$
in our complex of cubical cochain algebras.
We get similar easy results for the other double face identities $d^j_{\epsilon} d^i_{\eta} = d^i_{\eta} d^{j-1}_{\epsilon}$
and the other commutation relations between the face and degeneracy operators $s^j d^i_{\epsilon} = d^i_{\epsilon} s^{j-1}$, $s^j d^i_{\epsilon} = d^{i-1}_{\epsilon} s^j$,
and $s^j s^i = s^i s^{j+1}$,
while the coherence of the face and degeneracy relations of the coproduct operators
with respect to the relations $s^j d^i_0 = \id$, for $i = j,j+1$,
and $s^j d^j_1 = s^0\otimes\id_{I^j}$, $s^j d^{j+1}_1 = \id_{I^{k-j}}\otimes s^j$
follows from the commutativity of the diagrams of Figure~\ref{homotopy-E-infinity-cooperad:0-face-coherence}-\ref{homotopy-E-infinity-cooperad:1-face-coherence-right}.
\end{remark}

We have the following obvious notion of strict morphism of homotopy Segal $E_\infty$-Hopf shuffle (pre-)cooperads
and of homotopy Segal $E_\infty$-Hopf symmetric (pre-)cooperads,
which generalize the definition of morphism of strict Segal $E_\infty$-Hopf shuffle (pre-)cooperads
and strict Segal $E_\infty$-Hopf symmetric (pre-)cooperads
in~\S\ref{sec:strict-segal-cooperads}

\begin{defn}\label{definition:homotopy-E-infinity-cooperad-morphism}
A (strict) morphism of homotopy Segal $E_\infty$-Hopf shuffle (pre-)cooperads $\phi: \AOp\rightarrow\BOp$
is a collection of $\EOp$-algebra morphisms
\begin{equation*}
\phi_{\ttree}: \AOp(\ttree)\rightarrow\BOp(\ttree),\quad\text{$\ttree\in\Tree(r)$, $r>0$},
\end{equation*}
which
\begin{enumerate}
\item
preserve the action of all coproduct operators on our objects,
in the sense that the diagram
\begin{equation*}
\xymatrix{ \AOp(\stree)\ar[r]^{\phi_{\stree}}\ar[d]_{\rho_{\ttree\rightarrow\ttree_k\rightarrow\dots\rightarrow\ttree_1\rightarrow\stree}} &
\BOp(\stree)\ar[d]^{\rho_{\ttree\rightarrow\ttree_k\rightarrow\dots\rightarrow\ttree_1\rightarrow\stree}} \\
\AOp(\ttree)\otimes I^k\ar[r]^{\phi_{\ttree}\otimes\id} & \BOp(\ttree)\otimes I^k }
\end{equation*}
commutes, for all sequences of composable tree morphisms $\ttree\rightarrow\ttree_k\rightarrow\dots\rightarrow\ttree_1\rightarrow\stree$,
\item
and the action of facet operators, so that we have the same commutative diagram as in the case of morphisms strict Segal $E_\infty$-Hopf shuffle (pre-)cooperads:
\begin{equation*}
\xymatrix{ \AOp(\stree)\ar[r]^{\phi_{\stree}} & \BOp(\stree) \\
\AOp(\sigmatree)\ar[r]^{\phi_{\sigmatree}}\ar[u]^{i_{\sigmatree,\stree}} &
\BOp(\sigmatree)\ar[u]_{i_{\sigmatree,\stree}} },
\end{equation*}
for all subtree embeddings $\sigmatree\subset\stree$.
\end{enumerate}
If $\AOp$ and $\BOp$ are homotopy Segal $E_\infty$-Hopf symmetric (pre-)cooperads, then we say that $\phi: \AOp\rightarrow\BOp$ is a morphism of Segal $E_\infty$-Hopf symmetric (pre-)cooperads
when $\phi$ also preserves the action of permutations on our objects (we express this condition by the same commutative diagram as in the case of strict $E_\infty$-Hopf cooperads).
\end{defn}

These strict morphisms of homotopy Segal $E_\infty$-Hopf shuffle (pre-)cooperads can obviously be composed,
as well the strict morphisms of homotopy Segal $E_\infty$-Hopf symmetric (pre-)cooperads,
so that we can form a category of homotopy Segal $E_\infty$-Hopf shuffle (pre-)cooperads
and a category of homotopy Segal $E_\infty$-Hopf symmetric (pre-)cooperads
as morphisms.
In what follows, we adopt the notation $\EOp\Hopf\sh\hSegOp^c$ for the category of homotopy Segal $E_\infty$-Hopf shuffle cooperads
and the notation $\EOp\Hopf\Sigma\hSegOp^c$ for the category of homotopy Segal $E_\infty$-Hopf symmetric cooperads.

We can modify the above definition to assume that the preservation of coproduct operators holds up to homotopy only, just as we assume
that the coproduct operators satisfy associativity relations up to homotopy in a homotopy Segal $E_\infty$-Hopf pre-cooperad.
This idea gives the notion of homotopy morphism, which we study in Subsection~\ref{subsection:homotopy-morphisms}.

\subsection{The forgetting of $E_\infty$-structures}

We now study the structure in dg modules that we obtain by forgetting the $E_\infty$-algebra structure attached to each object
in the definition of a homotopy Segal $E_\infty$-Hopf cooperad.
We follow the same plan as in Subsection~\ref{subsec:strict-segal-cooperads},
where we examine the parallel forgetting of $E_\infty$-algebra structures
in strict Segal $E_\infty$-Hopf cooperads.
In the previous subsection, we assume that the associativity of the coproduct operators only holds up to homotopy for the definition of a homotopy Segal $E_\infty$-Hopf cooperad,
but we keep the same notion of facet operators as in the case of strict Segal $E_\infty$-cooperads.
In the case of homotopy Segal cooperads in dg modules, we retain the homotopy associativity relation of the coproduct operators of homotopy Segal $E_\infty$-Hopf cooperads,
and we retrieve the expression of the Segal map that we obtained from the structure of the facet operators
in the definition of strict Segal dg cooperads.

We again need to assume that the vertices of our trees are totally ordered in order to make the construction
of the forgetful functor from homotopy Segal $E_\infty$-Hopf cooperads
to homotopy Segal dg cooperads work.
For this reason, we restrict ourselves to Segal shuffle cooperads all along this subsection (as in our study of strict Segal cooperads in dg modules).

\begin{defn}\label{definition:homotopy-tree-shaped-cooperad}
We call homotopy Segal shuffle dg pre-cooperad the structure defined by a collection of dg modules
\begin{equation*}
\AOp(\ttree)\in\dg\Mod,\quad\text{$\ttree\in\Tree(r)$, $r>0$},
\end{equation*}
equipped with
\begin{itemize}
\item
homotopy coproduct operators
\begin{equation*}
\rho_{\ttree\rightarrow\ttree_k\rightarrow\dots\rightarrow\ttree_1\rightarrow\stree}: \AOp(\stree)\rightarrow\AOp(\ttree)\otimes I^k,
\end{equation*}
defined as morphisms of dg modules, for all composable sequences of tree morphisms $\ttree\rightarrow\ttree_k\rightarrow\dots\rightarrow\ttree_1\rightarrow\stree$, $k\geq 0$,
and which satisfy the same face and degeneracy relations,
expressed by the commutative diagrams of Figure~\ref{homotopy-E-infinity-cooperad:0-faces}-\ref{homotopy-E-infinity-cooperad:degeneracies},
as the homotopy coproduct operators of Segal $E_\infty$-Hopf shuffle cooperads,
\item
together with Segal maps
\begin{equation*}
i_{\lambda_{\stree}(\sigmatree_*)}: \bigotimes_{v\in V(\stree)}\AOp(\sigmatree_v)\rightarrow\AOp(\ttree),
\end{equation*}
defined as morphisms of dg modules, for all tree decompositions $\ttree = \lambda_{\stree}(\sigmatree_v,v\in V(\stree))$,
and which satisfy the same functoriality relations
as the Segal maps of strict Segal cooperads in dg modules.
Namely, for the trivial decomposition $\ttree = \lambda_{\ytree}(\ttree)$, we have $i_{\lambda_{\ytree}(\ttree)} = \id_{\AOp(\ttree)}$,
and for nested decompositions $\ttree = \lambda_{\utree}(\thetatree_u,u\in V(\utree))$ and $\thetatree_u = \lambda_{\stree_u}(\sigmatree_v,v\in V(\stree_u))$, $u\in V(\utree)$,
we have
\begin{equation*}
i_{\lambda_{\utree}(\thetatree_*)}\circ(\bigotimes_{u\in V(\utree)}i_{\lambda_{\stree_u}(\sigmatree_*)}) = i_{\lambda_{\stree}(\sigmatree_*)},
\end{equation*}
where we again consider the composite decomposition $\ttree = \lambda_{\stree}(\sigmatree_v,v\in V(\stree))$
with $\stree = \lambda_{\utree}(\stree_u,u\in V(\utree))$.
\item
We now assume the verification of the compatibility relations depicted in Figure~\ref{homotopy-dg-cooperad:Segal-coproduct-compatibility},
for the Segal maps and the higher coproduct operators.
\end{itemize}
We also say that a homotopy Segal shuffle dg pre-cooperad $\AOp$ is a homotopy Segal shuffle dg cooperad when the Segal maps
satisfy the following Segal condition (the same Segal condition as in the case of strict Segal dg cooperads):
\begin{enumerate}
\item[(*)]
The Segal map $i_{\lambda_{\stree}(\sigmatree_*)}$ is a weak-equivalence
\begin{equation*}
i_{\lambda_{\stree}(\sigmatree_*)}: \bigotimes_{v\in V(\stree)}\AOp(\sigmatree_v)\xrightarrow{\sim}\AOp(\ttree),
\end{equation*}
for every decomposition $\ttree = \lambda_{\stree}(\sigmatree_v, v\in V(\stree))$.
\end{enumerate}
We still consider a notion of strict morphism of homotopy Segal shuffle dg (pre-)cooperads, which we obviously define
as a collection of dg module morphisms $\phi_{\ttree}: \AOp(\ttree)\rightarrow\BOp(\ttree) $
that preserve all homotopy coproduct operators $\rho_{\ttree\rightarrow\ttree_k\rightarrow\dots\rightarrow\ttree_1\rightarrow\stree}$
and the Segal maps $i_{\lambda_{\stree}(\sigmatree_*)}$.
We use the notation $\dg\sh\hSegOp^c$ for the category of Segal shuffle dg cooperads, which we equip with this notion of morphism.
\end{defn}

\begin{figure}[t]
\ffigbox
{\caption{The compatibility between coproducts and Segal maps.
The diagram commutes for all sequences of composable tree morphisms $\ttree\xrightarrow{f_k}\ttree_k\xrightarrow{f_{k-1}}\cdots\xrightarrow{f_1}\ttree_1\xrightarrow{f_0}\stree$
and for all decompositions $\stree = \lambda_{\utree}(\sigmatree_v, v \in V(\utree))$.
The map $\mu: (\bigotimes_{v\in V(\utree)}I^k)\rightarrow I^k$ on the right hand side vertical arrow is given by the associative product
of the $\EOp$-algebra $I^k$.}\label{homotopy-dg-cooperad:Segal-coproduct-compatibility}}
{\centerline{\xymatrixcolsep{15pc}\xymatrix{ \AOp(\stree)
\ar[r]^{\rho_{\ttree\rightarrow\ttree_k\rightarrow\dots\rightarrow\ttree_1\rightarrow\stree}} &
\AOp(\ttree)\otimes I^k \\
& \left(\bigotimes_{v\in V(\utree)}\AOp(\sigmatree_v)\right)\otimes\left(\bigotimes_{v\in V(\utree)}I^k\right)
\ar@{.>}[u]_{i_{\lambda_{\utree}((f_0 \dots f_k)^{-1}(\sigmatree_*))}\otimes\mu} \\
\bigotimes_{v\in V(\utree)}\AOp(\sigmatree_v)
\ar@{.>}[r]_-{\bigotimes_v\rho_{(f_0\dots f_k)^{-1}(\sigmatree_v)\rightarrow(f_0\dots f_{k-1})^{-1}(\sigmatree_v)\rightarrow\dots\rightarrow\sigmatree_v}}
\ar[uu]_-{i_{\lambda_{\utree}(\sigmatree_*)}} &
\bigotimes_{v\in V(\utree)}\left(\AOp((f_0 \dots f_k)^{-1}(\sigmatree_v))\otimes I^k\right)\ar@{.>}[u]_{\simeq} }
}}
\end{figure}

There is a forgetful functor from the category of homotopy Segal $E_\infty$-Hopf cooperads to the category of homotopy Segal shuffle dg cooperads,
which is similar to the one that we have in the strict case.
To be more precise, we have the following proposition, which represents the homotopy counterpart of the result of Proposition~\ref{proposition:forgetful-strict}.

\begin{prop}\label{proposition:forgetful-homotopy}
Let $\AOp$ be a homotopy Segal $E_\infty$-Hopf shuffle cooperad,
with coproduct operators $\rho_{\ttree\rightarrow\ttree_k\rightarrow\dots\rightarrow\ttree_1\rightarrow\stree}: \AOp(\stree)\rightarrow\AOp(\ttree)\otimes I^k$
and facet operators $i_{\sigmatree,\stree}: \AOp(\sigmatree)\rightarrow\AOp(\stree)$.
The collection $\AOp(\ttree)$, $\ttree\in\Tree(r)$, equipped with the coproduct operators inherited from $\AOp$ and the Segal maps
given by the following composites
\begin{equation*}
\bigotimes_{v\in V(\stree)}\AOp(\sigmatree_v)\xrightarrow{\EM}\bigvee_{v\in V(\stree)}\AOp(\sigmatree_v)\xrightarrow{i_{\lambda_{\stree}(\sigmatree_*)}}\AOp(\ttree)
\end{equation*}
(the same composites as in Proposition~\ref{proposition:forgetful-strict}),
for all tree decompositions $\ttree = \lambda_{\stree}(\sigmatree_v,v\in V(\stree))$,
is a homotopy Segal shuffle dg cooperad.
\end{prop}

\begin{proof}
The face and degeneracy relations for the homotopy coproduct operators are directly inherited from the homotopy Segal $E_\infty$-Hopf shuffle cooperad $\AOp$,
as we do not change the coproduct operators in our forgetful operation.
The functoriality relations of the Segal maps
and the Segal condition
follows from the same arguments as in Proposition \ref{proposition:forgetful-strict}
since the definition of the Segal maps
is the same as in the strict case.
The compatibility between the homotopy coproduct operators and the Segal maps follows from the commutativity of the following diagram,
\begin{equation*}
\xymatrix{ \bigotimes_{v \in V(\utree)}\AOp(\sigmatree_v)
\ar[d]|{\bigotimes_{v \in V(\utree)}\rho_{(f_0\dots f_k)^{-1}(\sigmatree_v)\rightarrow\dots\rightarrow f_0^{-1}(\sigmatree_v)\rightarrow\sigmatree_v}}
\ar[r]^{\EM} &
\bigvee_{v\in V(\utree)} \AOp(\sigmatree_v)
\ar[d]|{\bigvee_{v \in V(\utree)}\rho_{(f_0\dots f_k)^{-1}(\sigmatree_v)\rightarrow\dots\rightarrow f_0^{-1}(\sigmatree_v)\rightarrow\sigmatree_v}} \\
\bigotimes_{v \in V(\utree)}\left(\AOp((f_0\dots f_k)^{-1}(\sigmatree_v)\otimes I^k)\right)
\ar[r]^{\EM}\ar[d]_{\simeq} &
\bigvee_{v\in V(\utree)}\left(\AOp((f_0\dots f_k)^{-1}(\sigmatree_v))\otimes I^k\right)
\ar[d]|{\sum_v\left(i_{(f_0\dots f_k)^{-1}(\sigmatree_v),\ttree}\otimes\id\right)} \\
\left(\bigotimes_{v\in V(\utree)}\AOp((f_0\dots f_k)^{-1}(\sigmatree_v))\right)\otimes\left(\bigotimes_{v\in V(\utree)}I^k\right)
\ar[d]_{\EM\otimes\EM} &
\AOp(\ttree)\otimes I^k \\
\left(\bigvee_{v\in V(\utree)}\AOp((f_0\dots f_k)^{-1}(\sigmatree_v))\right)\otimes\left(\bigvee_{v\in V(\utree)}I^k\right)
\ar[ru]_-{i_{\lambda_{\utree}(\sigmatree_*)}\otimes\nabla} }
\end{equation*}
where $\nabla = \sum_v\id: \bigvee_v I^k\rightarrow I^k$ denotes the codiagonal of the $\EOp$-algebra $I^k$,
we use that the composite $\nabla\circ\EM: \bigotimes_v I^k\rightarrow I^k$ is identified with the associative product $\mu: \bigotimes_v I^k\rightarrow I^k$
and we consider the Segal map
of $\EOp$-algebras $i_{\lambda_{\utree}(\sigmatree_*)}: \bigvee_{v\in V(\utree)}\AOp((f_0\dots f_k)^{-1}(\sigmatree_v))\rightarrow\AOp(\ttree)$.
\end{proof}

The normalized cochain complex of the $k$-cube $I^k = \DGN^*(\Delta^1)^{\otimes k}$ is identified with the sum of the top dimensional element $\underline{01}^\sharp{}^{\otimes k}$
with the image of cubical face operators. We use this observation to determine the homotopy coproducts of homotopy Segal dg cooperads
in the following lemma.
For a morphism of dg modules $\alpha: X\rightarrow Y\otimes I^k$, we let $\alpha^{\square}: X\rightarrow Y$ be the homomorphism of graded modules of degree $k$
given by the component of the map $\alpha$ with values in $Y\otimes\underline{01}^{\sharp}{}^{\otimes k}\subset Y\otimes I^k$,
where we consider the top dimensional element $\underline{01}^{\sharp}{}^{\otimes k}\in\DGN^*(\Delta^1)^{\otimes k}$
of the cubical cochain complex $I^k = \DGN^*(\Delta^1)^{\otimes k}$,
so that we have:
\begin{equation*}
\alpha(x) = (-1)^{k\deg(x)}\alpha^{\square}(x)\otimes\underline{01}^{\sharp}{}^{\otimes k} + \text{tensors with a factor of dimension $<k$ in $I^k$},
\end{equation*}
for any $x\in X$. We also write $\delta(\alpha^{\square}) = \delta\alpha^{\square} - (-1)^k\alpha^{\square}\delta$
for the differential of this homomorphism.

\begin{lemm}\label{lemma:homotopy-cooperad-top-component}
\begin{enumerate}
\item\label{homotopy-cooperad-top-component:expression}
Let $\AOp$ be a homotopy Segal shuffle dg cooperad. The graded homomorphism of degree $k$
\begin{equation*}
\rho^{\square}_{\ttree\rightarrow\ttree_k\rightarrow\dots\rightarrow\ttree_1\rightarrow\stree}: \AOp(\stree)\rightarrow\AOp(\ttree)
\end{equation*}
that we associate to the dg module morphism $\rho_{\ttree\rightarrow\ttree_k\rightarrow\dots\rightarrow\ttree_1\rightarrow\stree}: \AOp(\stree)\rightarrow\AOp(\ttree)\otimes I^k$,
for any sequence of composable tree morphisms $\ttree\rightarrow\ttree_k\rightarrow\dots\rightarrow\ttree_1\rightarrow\stree$,
satisfies the relation
\begin{equation}\tag{*}\label{equation:differential-top-component}
\delta(\rho^{\square}_{\ttree\rightarrow\ttree_k\rightarrow\dots\rightarrow\ttree_1\rightarrow\stree})
\begin{aligned}[t] & = \sum_{i=1}^k(-1)^{i-1}\rho^{\square}_{\ttree\rightarrow\ttree_k\rightarrow\dots\rightarrow\ttree_i}
\circ\rho^{\square}_{\ttree_i\rightarrow\dots\rightarrow\ttree_1\rightarrow\stree}\\
& - \sum_{i=1}^k(-1)^{i-1}\rho^{\square}_{\ttree\rightarrow\dots\rightarrow\hat{\ttree}_i\rightarrow\dots\rightarrow\stree}.
\end{aligned}
\end{equation}
Moreover, if we have a degeneracy $\ttree_j = \ttree_{j+1}$ in our sequence of tree morphisms, for some $0\leq j\leq k$ (with the convention that $\ttree_{k+1} = \ttree$ and $\ttree_0 = \stree$),
then we have the relation
\begin{equation}\tag{**}\label{equation:degeneracy-top-component}
\rho^{\square}_{\ttree\rightarrow\ttree_{k}\rightarrow\dots\ttree_{j+1}=\ttree_j\cdots\rightarrow\ttree_1\rightarrow\stree} = 0.
\end{equation}
\item\label{homotopy-cooperad-top-component:generation}
Let now $\AOp(\ttree)$, $\ttree\in\Tree(r)$, $r>0$, be any given collection of dg modules
equipped with graded homomorphisms $\rho^{\square}_{\ttree\rightarrow\ttree_k\rightarrow\dots\rightarrow\ttree_1\rightarrow\stree}: \AOp(\stree)\rightarrow\AOp(\ttree)$,
of degree $k$,
and which satisfy the relations (\ref{equation:differential-top-component})-(\ref{equation:degeneracy-top-component})
of the previous statement.
Then there is a unique collection of morphisms of dg modules $\rho_{\ttree\rightarrow\ttree_k\rightarrow\dots\rightarrow\ttree_1\rightarrow\stree}: \AOp(\stree)\rightarrow\AOp(\ttree)\otimes I^k$,
which extend these maps on the summands $\AOp(\ttree)\otimes\underline{01}^{\sharp}{}^{\otimes k}$
and satisfy the face and degeneracy relations of homotopy coproduct operators
of Figure~\ref{homotopy-E-infinity-cooperad:0-faces}-\ref{homotopy-E-infinity-cooperad:degeneracies}.
\end{enumerate}
\end{lemm}

\begin{proof}[Proof of assertion~\ref{homotopy-cooperad-top-component:expression}]
We use the identities
\begin{gather*}
\begin{aligned}
(I^k)_{-k} & = \kk\,\underline{01}^{\sharp}{}^{\otimes k}\\
(I^k)_{1-k} & = \bigoplus_{i=1}^k\bigl(\bigoplus_{\epsilon\in\{0,1\}}\kk\,\underline{01}^{\sharp}{}^{\otimes k-i}\otimes\underline{\epsilon}^\sharp\otimes\underline{01}^{\sharp}{}^{\otimes i-1}\bigr),
\end{aligned}
\intertext{and the relations}
d^i_{\epsilon}(\underline{01}^{\sharp}{}^{\otimes k-j}\otimes\underline{\eta}^\sharp\otimes\underline{01}^{\sharp}{}^{\otimes j-1})
= \begin{cases}\underline{01}^{\sharp}{}^{\otimes k-1}, & \text{if $i = j$ and $\epsilon\equiv\eta + 1\mymod 2$}, \\
0, & \text{otherwise}, \end{cases}
\end{gather*}
which implies that we have a formula of the form:
\begin{multline*}
\alpha(x) = (-1)^{k\deg(x)}\alpha^{\square}(x)\otimes\underline{01}^{\sharp}{}^{\otimes k} \\
+ \sum_{i=1}^k\left(\sum_{\epsilon\in\{0,1\}}(-1)^{(k-1)\deg(x)}((\id\otimes d^i_{\epsilon+1})\circ\alpha)^{\square}(x)\otimes\underline{01}^{\sharp}{}^{\otimes k-i}\otimes\underline{\epsilon}^\sharp\otimes\underline{01}^{\sharp}{}^{\otimes i-1}\right)\\
+ \text{tensors with a factor of dimension $<k-1$ in $I^k$},
\end{multline*}
for any morphism $\alpha: X\rightarrow Y\otimes I^k$ and any $x\in X$,
where we consider the homomorphism of graded modules $((\id\otimes d^i_{\epsilon+1})\circ\alpha)^{\square}: X\rightarrow Y$
associated to the composite $(\id\otimes d^i_{\epsilon+1})\circ\alpha: X\rightarrow Y\otimes I^{k-1}$ (and we obviously take the face operator $d^i_{\epsilon+1}$
indexed by the residue class of $\epsilon+1$ mod $2$).
We deduce from the differential formula
\begin{equation*}
\delta(\underline{01}^{\sharp}{}^{\otimes k-i}\otimes\underline{\epsilon}^\sharp\otimes\underline{01}^{\sharp}{}^{\otimes i-1}) = (-1)^{k-i+\epsilon+1}\underline{01}^{\sharp}{}^{\otimes k}
\end{equation*}
that the projection of the relation $\delta(\alpha(x)) = \alpha(\delta(x))$ onto $Y\otimes\underline{01}^{\sharp}{}^{\otimes k}\subset Y\otimes I^k$
is equivalent to the following relation:
\begin{equation*}
\delta(\alpha^{\square}(x)) + \sum_{i=1}^k\bigl(\sum_{\epsilon\in\{0,1\}}(-1)^{i+\epsilon}((\id\otimes d^i_{\epsilon+1})\alpha)^{\square}(x)\bigr)
= (-1)^k\alpha^{\square}(\delta x).
\end{equation*}
The relation of our statement (\ref{equation:differential-top-component}) then follows from the compatibility of coproducts with the face operators
(the relations of Figure~\ref{homotopy-E-infinity-cooperad:0-faces}-\ref{homotopy-E-infinity-cooperad:1-faces}).

If we have $\ttree_j = \ttree_{j+1}$ for some $j$, then the morphism $\rho_{\ttree\rightarrow\ttree_k\rightarrow\dots\rightarrow\ttree_1\rightarrow\stree}$
factors through the image of a degeneracy $\AOp(\ttree)\otimes s^j: \AOp(\ttree)\otimes I^{k-1}\rightarrow\AOp(\ttree)\otimes I^k$
by the compatibility of the coproducts with the degeneracies (the relations of Figure~\ref{homotopy-E-infinity-cooperad:degeneracies})
and this requirement implies the vanishing relation $\rho_{\ttree\rightarrow\ttree_k\rightarrow\dots\rightarrow\ttree_1\rightarrow\stree}^{\square} = 0$
since $I^{k-1}$ is null in dimension $k$.
\end{proof}

\begin{proof}[Proof of assertion~\ref{homotopy-cooperad-top-component:generation}]
The operators $\rho_{\ttree\rightarrow\ttree_k\rightarrow\dots\rightarrow\ttree_1\rightarrow\stree}: \AOp(\stree)\rightarrow\AOp(\ttree)\otimes I^k$
are defined by induction on $k$.
If $k = 0$, then we take $\rho_{\ttree\rightarrow\stree} = \rho_{\ttree\rightarrow\stree}^{\square}$.
The compatibility conditions of Figure~\ref{homotopy-E-infinity-cooperad:0-faces}-\ref{homotopy-E-infinity-cooperad:degeneracies}
are tautological in this case.

If $k > 0$, then defining a morphism of dg modules $\rho_{\ttree\rightarrow\ttree_k\rightarrow\dots\rightarrow\ttree_1\rightarrow\stree}: \AOp(\stree)\rightarrow\AOp(\ttree)\otimes I^k$
amounts to defining a morphism of dg modules
$\rho'_{\ttree\rightarrow\ttree_k\rightarrow\dots\rightarrow\ttree_1\rightarrow\stree}: \AOp(\stree)\otimes\DGN_*(\Delta^1)^{\otimes k}\rightarrow\AOp(\ttree)$.

Let $d_i^{\epsilon}: \DGN_*(\Delta^1)^{\otimes k-1}\rightarrow\DGN_*(\Delta^1)^{\otimes k}$ denote the coface operators attached to the cubical complex $\DGN_*(\Delta^1)^{\otimes k}$
dual to the face operators $d^i_{\epsilon}: \DGN^*(\Delta^1)^{\otimes k}\rightarrow\DGN^*(\Delta^1)^{\otimes k-1}$
that we considered so far. We let $\rho'_{\ttree\rightarrow\ttree_k\rightarrow\dots\rightarrow\ttree_1\rightarrow\stree}(x\otimes\underline{01}^{\otimes r})
= \rho^{\square}_{\ttree\rightarrow\ttree_k\rightarrow\dots\rightarrow\ttree_1\rightarrow\stree}(x)$,
and we define inductively:
\begin{align*}
\rho'_{\ttree\rightarrow\ttree_k\rightarrow\dots\rightarrow\ttree_1\rightarrow\stree}(x\otimes d_i^0(\underline{\sigma})) &
= \rho'_{\ttree\rightarrow\dots\rightarrow\widehat{\ttree_i}\rightarrow\dots\rightarrow\stree}(x\otimes\underline{\sigma}), \\
\rho'_{\ttree\rightarrow\ttree_k\rightarrow\dots\rightarrow\ttree_1\rightarrow\stree}(x\otimes d_i^1(\underline{\sigma})) &
= \rho'_{\ttree\rightarrow\dots\rightarrow\ttree_i}(\rho'_{\ttree_i\rightarrow\dots\rightarrow\stree}\otimes\id(x\otimes\underline{\sigma})),
\end{align*}
where we use the factorization $\underline{\sigma}\in\DGN_*(\Delta^1)^{\otimes k-1}\Rightarrow\underline{\sigma} = \underline{\sigma}'\otimes\underline{\sigma}''\in\DGN_*(\Delta^1)^{\otimes k-i}\otimes\DGN_*(\Delta^1)^{\otimes i-1}$.
We deduce from the compatibility of the homotopy coproducts with the face operators that our coproduct operator is necessarily given by these formulas
and this observation proves the uniqueness of the coproduct operators
extending our maps $\rho^{\square}_{\ttree\rightarrow\ttree_k\rightarrow\dots\rightarrow\ttree_1\rightarrow\stree}$.
We may also use the observations of Remark \ref{remark:compatibility-relations} to check that our maps satisfy the compatibility relations
of Figure~\ref{homotopy-E-infinity-cooperad:0-faces}-\ref{homotopy-E-infinity-cooperad:degeneracies}
on the whole dg modules $\AOp(\ttree)\otimes I^k$.

We only need to prove that the map $\rho_{\ttree\rightarrow\ttree_k\rightarrow\dots\rightarrow\ttree_1\rightarrow\stree}$
or equivalently the map $\rho'_{\ttree\rightarrow\ttree_k\rightarrow\dots\rightarrow\ttree_1\rightarrow\stree}$
commutes with the differential.
We proceed by induction on $k$. For $k = 0$ the statement is obviously equivalent to relation~(\ref{equation:differential-top-component}).
For $k > 0$, the relation $\delta(\rho'_{\ttree\rightarrow\ttree_k\rightarrow\dots\rightarrow\ttree_1\rightarrow\stree}(x\otimes\underline{\sigma}))
= \rho'_{\ttree\rightarrow\ttree_k\rightarrow\dots\rightarrow\ttree_1\rightarrow\stree}\delta(x\otimes\underline{\sigma})$
follows from the induction assumption
in the case $\underline{\sigma} = (d_i^{\epsilon})(\underline{\tau})$, for some $1\leq i\leq k$, $\epsilon\in\{0,1\}$, $\underline{\tau}\in\DGN_*(\Delta^1)^{\otimes k-1}$,
and reduces to relation~(\ref{equation:differential-top-component})
when $\underline{\sigma} = \underline{01}^{\otimes k}$.
The conclusion follows.
\end{proof}

\subsection{The application of cobar complexes}\label{subsection:homotopy-Segal-cooperad-cobar}

The aim of this Subsection is to extend the cobar construction of Subsection~\ref{subsection:strict-Segal-cooperad-cobar}
to homotopy Segal dg cooperads. We follow the same plan.

We make explicit the definition of the structure operations of the cobar construction $\DGB^c(\AOp)$
associated to a homotopy Segal shuffle dg cooperad $\AOp$
in the next paragraph.
We check the validity of these definitions afterwards and we record the definition of this dg operad $\DGB^c(\AOp)$
to conclude this subsection.
We assume all along this subsection that $\AOp$ is a homotopy Segal shuffle dg cooperad.
We also assume that $\AOp$ is connected in the sense that we have $\AOp(\ttree) = 0$ when the tree $\ttree$ is not reduced (has at least one vertex with a single ingoing edge).
This condition implies that every object $\AOp(\stree)$
is the source of finitely many nonzero homotopy coproducts $\rho_{\ttree\rightarrow\dots\rightarrow\stree}: \AOp(\stree)\rightarrow\AOp(\ttree)\otimes I^k$,
because the set of tree morphisms $\ttree\rightarrow\stree$ with $\ttree$ reduced and $\stree$ fixed is finite.

\begin{constr}\label{construction:bar-complex-homotopy}
We define the graded modules $\DGB^c(\AOp)(r)$ which form the components of the cobar construction $\DGB^c(\AOp)$
by the same formula as in the strict case
\begin{equation*}
\DGB^c(\AOp)(r) = \bigoplus_{\ttree\in\Tree(r)}\DGSigma^{-\sharp V(\ttree)}\AOp(\ttree),
\end{equation*}
but we now take a cobar differential that involves terms given by higher coproduct operators and which we associate to multiple edge contractions.
We precisely consider the set of pairs $(\ttree,\underline{e})$, where $\ttree\in\Tree(r)$ and $\underline{e} = (e_1,\dots,e_m)$
is an ordered collection of pairwise distinct edges $e_i\in\mathring{E}(\ttree)$.
To any such pair, we associate the sequence of tree morphisms such that
\begin{gather*}
\sigma(\ttree,\underline{e}) = \{\ttree\rightarrow\ttree/{e_1}\rightarrow\ttree/{\{e_1,e_2\}}\rightarrow\dots\rightarrow\ttree/\{e_1,e_2,\dots,e_m\}\}.
\intertext{and the map of degree $-1$}
\partial_{(\ttree,\underline{e})}: \DGSigma^{-\sharp V(\ttree)+m}\AOp(\ttree/\{e_1,\dots,e_m\})\rightarrow\DGSigma^{-\sharp V(\ttree)}\AOp(\ttree)
\intertext{given by top component of the homotopy coproduct $\rho_{\sigma(\ttree,\underline{e})}$}
\partial_{(\ttree,\underline{e})} = \rho^{\square}_{\sigma(\ttree,\underline{e})},
\end{gather*}
such as defined in Lemma~\ref{lemma:homotopy-cooperad-top-component}.

Recall that the object $\DGSigma^{-\sharp V(\ttree)}\AOp(\ttree)$
is identified with a tensor product
\begin{equation*}
\DGSigma^{-\sharp V(\ttree)}\AOp(\ttree) = \bigl(\bigotimes_{v\in V(\ttree)}\underline{01}^{\sharp}_v\bigr)\otimes\AOp(\ttree),
\end{equation*}
where we associate a factor of cohomological degree one $\underline{01}^{\sharp}_v$ to every vertex $v\in V(\ttree)$.
In the definition of our map $\partial_{(\ttree,e)}$, we also perform blow-up operations
\begin{equation*}
\underline{01}_{u\equiv v}^{\sharp}\mapsto\underline{01}_{u}^{\sharp}\otimes\underline{01}_{v}^{\sharp},
\end{equation*}
for each edge contraction step $\ttree/{\{e_1,\dots,e_{i-1}\}}\rightarrow\ttree/{\{e_1,\dots,e_{i-1},e_i\}}$,
where $u$ and $v$ represent the vertices of the edge $e_i$ in the tree $\ttree/{\{e_1,\dots,e_{i-1}\}}$,
and $u\equiv v$ represents the result of the fusion of these vertices
in $\ttree/{\{e_1,\dots,e_{i-1}\}}$.
The performance of this sequence of blow-up operations, in parallel to the application of the map $\rho^{\square}_{\sigma(\ttree,\underline{e})}$,
enables us to pass
from the tensor product $\DGSigma^{-\sharp V(\ttree)+m}\AOp(\ttree/\{e_1,\dots,e_m\})
= \bigl(\bigotimes_{x\in V(\ttree/\{e_1,\dots,e_m\})}\underline{01}^{\sharp}_x\bigr)\otimes\AOp(\ttree/\{e_1,\dots,e_m\})$
to $\DGSigma^{-\sharp V(\ttree)}\AOp(\ttree) = \bigl(\bigotimes_{x\in V(\ttree)}\underline{01}^{\sharp}_x\bigr)\otimes\AOp(\ttree)$.
This operation may involve a sign, which we determine as in the strict case in Construction~\ref{construction:bar-complex}.

Finally, we take:
\begin{equation*}
\partial_m = \sum_{(\ttree,(e_1,\dots,e_m))}\partial_{(\ttree,(e_1,\dots,e_m))},\quad\text{for $m\geq 1$},
\quad\text{and}\quad\partial = \sum_{m\geq 1}\partial_m.
\end{equation*}
(We just use the connectedness condition on $\AOp$ to ensure that this sum reduces to a finite number of terms
on each summand $\DGSigma^{-\sharp V(\ttree)}\AOp(\ttree)$
of our object $\DGB^c(\AOp)(r)$.)
We prove next that this map $\partial$ also defines a twisting differential on $\DGB^c(\AOp)(r)$, so that we can again provide $\DGB^c(\AOp)(r)$
with a dg module structure with the sum $\delta+\partial: \DGB^c(\AOp)(r)\rightarrow\DGB^c(\AOp)(r)$ as total differential,
where $\delta: \DGB^c(\AOp)(r)\rightarrow\DGB^c(\AOp)(r)$ denotes the differential induced by the internal differential of the objects $\AOp(\ttree)$
in $\DGB^c(\AOp)(r)$ (as in the case of the cobar construction of strict Segal dg cooperads).

We equip the object $\DGB^c(\AOp)$ with composition products
\begin{equation*}
\circ_{i_p}: \DGB^c(\AOp)(\{i_1<\dots<i_k\})\otimes\DGB^c(\AOp)(\{j_1<\dots<j_l\})\rightarrow\DGB^c(\AOp)(\{1<\dots<r\}),
\end{equation*}
which we define exactly as in the strict case.
We explicitly define $\circ_{i_p}$ as the sum of the maps
\begin{equation*}
\circ_{i_p}^{\stree,\ttree}: \DGSigma^{-\sharp V(\stree)}\AOp(\stree)\otimes\DGSigma^{-\sharp V(\ttree)}\AOp(\ttree)
\rightarrow\DGSigma^{-\sharp V(\stree\circ_{i_p}\ttree)}\AOp(\stree\circ_{i_p}\ttree)
\end{equation*}
yielded by the Segal map $i_{\stree\circ_{i_p}\ttree}: \AOp(\stree)\otimes\AOp(\ttree)\rightarrow\AOp(\stree\circ_{i_p}\ttree)$
associated to the composition operation $\stree\circ_{i_p}\ttree = \lambda_{\gammatree}(\stree,\ttree)$
in the category of trees, where $\stree\in\Tree(\{i_1<\dots<i_k\})$ and $\ttree\in\Tree(\{j_1<\dots<j_l\})$.
In a forthcoming lemma, we also check that these operations preserve the above differential, and hence, provide our object with well-defined operations in the category of dg modules.
\end{constr}

We first check the validity of the definition of the twisting differential announced in our construction. This result is a consequence of the following more precise lemma.

\begin{lemm}\label{lemma:differential-bar-construction-homotopy}
We have the relation $\delta\partial_m + \partial_m\delta + \sum_{i=1}^{m-1}\partial_i\partial_{m-i} = 0$, for each $m\geq 1$.
\end{lemm}

\begin{proof}
If $m = 1$, then the statement reduces to $\delta\partial_1 + \partial_1\delta = 0$, and we readily check, as in the case of the cobar construction of Segal dg cooperads,
that this relation is equivalent to the fact that the coproducts of degree zero $\rho_{\ttree\rightarrow\stree} = \rho_{\ttree\rightarrow\stree}^{\square}$
are morphisms of dg modules.
We now prove the statement for $m>1$. From Equation~(\ref{equation:differential-top-component}) of Lemma~\ref{lemma:homotopy-cooperad-top-component},
we see that, for every tree $\ttree$ and every sequence of internal edges $\underline{e} = (e_1,\dots,e_m)$,
we have:
\begin{align*}
\delta\partial_{(\ttree,\underline{e})} + \partial_{(\ttree,\underline{e})}\delta &
= \sum_{i=1}^{m-1}\pm\rho^{\square}_{\ttree\rightarrow\dots\rightarrow\widehat{\ttree/\{e_1,\dots,e_i\}}\rightarrow\dots\rightarrow\ttree/\{e_1,\dots,e_m\}} \\
& + \sum_{i=1}^{m-1}\pm\rho^{\square}_{\ttree\rightarrow\dots\rightarrow\ttree/\{e_1,\dots,e_i\}}\circ\rho^{\square}_{\ttree/\{e_1,\dots,e_i\}\rightarrow\dots\rightarrow\ttree/\{e_1,\dots,e_m\}}
\end{align*}
Then, by taking the sum of these expressions over the set of pairs $(\ttree,\underline{e})$, we obtain the formula:
\begin{align*}
\delta\partial_m + \partial_m\delta &
= \sum_{\substack{(\ttree,\underline{e})\\i=1,\dots,m}}\pm\rho^{\square}_{\ttree\rightarrow\dots\rightarrow\widehat{\ttree/\{e_1,\dots,e_i\}}\rightarrow\dots\rightarrow\ttree/\{e_1,\dots,e_m\}} \\
& + \sum_{i=1}^{m-1}\bigl(\underbrace{\sum_{(\ttree,\underline{e})}\pm\rho^{\square}_{\ttree\rightarrow\dots\rightarrow\ttree/\{e_1,\dots,e_i\}}
\circ\rho^{\square}_{\ttree/\{e_1,\dots,e_i\}\rightarrow\dots\rightarrow\ttree/\{e_1,\dots,e_m\}}}_{= \partial_i\partial_{m-i}}\bigr).
\end{align*}
In the first sum of this formula, the term that corresponds to the removal of the node $\ttree/\{e_1,\dots,e_{i-1},e_i\}$
and the term that corresponds to the removal of the node $\ttree/\{e_1,\dots,e_{i-1},e_{i+1}\}$
for the pair $(\ttree,(e_1,\dots,e_{i-1},e_{i+1},e_i,e_{i+2},\dots,e_m))$
with $e_i$ and $e_{i+1}$ switched
are equal up to a sign. We readily check that these signs are opposite, so that these terms cancel out in our sum. The conclusion of the lemma follows.
\end{proof}

We still check the validity of our definition of the composition products.

\begin{lemm}\label{lemma:bar-construction-product-operad-homotopy}
The twisting map $\partial$ and the differential $\delta$ induced by the internal differential of the object $\AOp$ on the cobar construction $\DGB^c(\AOp)$
form derivations with respect the composition products of Construction~\ref{construction:bar-complex-homotopy},
so that these operations $\circ_{i_p}$ define morphisms of dg modules.
\end{lemm}

\begin{proof}
We generalize the arguments used in Lemma \ref{lemma:bar-construction-product-operad}.
We again prove that $\circ_{i_p}$ commutes with the differential on each summand $\DGSigma^{-\sharp V(\stree)}\AOp(\stree)\otimes\DGSigma^{-\sharp V(\ttree)}\AOp(\ttree)$
of the tensor product $\DGB^c(\AOp)(\{i_1<\dots<i_k\})\otimes\DGB^c(\AOp)(\{j_1<\dots<j_l\})$,
where $\stree\in\Tree(\{i_1<\dots<i_k\})$, $\ttree\in\Tree(\{j_1<\dots<j_l\})$.
We can still use that the Segal maps, which induce our composition product componentwise, are morphisms of dg modules to conclude that $\circ_{i_p}$
preserves the internal differentials on our objects.
We therefore focus on the verification that $\circ_{i_p}$ preserves the twisting differential $\partial$.

We set $\thetatree = \stree\circ_{i_p}\ttree$ and we consider a tree $\thetatree'$ such that $\thetatree = \thetatree'/\{e_1,\dots,e_m\}$
for an sequence of edges $e_1,\dots,e_m\in\mathring{E}(\thetatree')$.
We then have $\thetatree' = \stree'\circ_{i_p}\ttree'$, where $\stree'$ and $\ttree'$ represent the pre-image of the subtrees $\stree\subset\thetatree$ and $\ttree\subset\thetatree$
under the map $\thetatree'\rightarrow\thetatree'/\{e_1,\dots,e_m\}$, and $\stree = \stree'/\{e_{\alpha_1},\dots,e_{\alpha_r}\}$, $\ttree = \ttree'/\{e_{\beta_1},\dots,e_{\beta_s}\}$,
for the partition $\{e_{\alpha_1},\dots,e_{\alpha_r}\}\amalg\{e_{\beta_1},\dots,e_{\beta_s}\} = \{e_1,\dots,e_m\}$
such that $e_{\alpha_1},\dots,e_{\alpha_r}\in\mathring{E}(\stree')$ and $e_{\beta_1},\dots,e_{\beta_s}\in\mathring{E}(\ttree')$.
We may note that one the component of this partition can be empty (we take by convention $r = 0$ or $s = 0$ in this case).

The compatibility between the Segal maps and the coproduct operators in Figure~\ref{homotopy-dg-cooperad:Segal-coproduct-compatibility}
together with the degeneracy relations of Figure~\ref{homotopy-E-infinity-cooperad:degeneracies}
imply the commutativity of the following diagram:
\begin{equation*}
\xymatrixcolsep{5pc}\xymatrix{ \AOp(\stree)\otimes\AOp(\ttree)
\ar[r]^-{\rho_{\sigma(\stree',\underline{e}|_{\stree'})}\otimes\rho_{\sigma(\ttree',\underline{e}|_{\ttree'})}}
\ar[ddd]|{i_{\stree\circ_{i_p}\ttree}} &
\AOp(\stree')\otimes I^{r-1}\otimes\AOp(\ttree')\otimes I^{s-1}
\ar[d]|{(\id\otimes s^{m-\beta_*})\otimes(\id\otimes s^{m-\alpha_*})} \\
& \AOp(\stree')\otimes I^{m-1}\otimes\AOp(\ttree)\otimes I^{m-1}\ar[d]^{\simeq} \\
& \AOp(\stree')\otimes\AOp(\ttree')\otimes I^{m-1}\otimes I^{m-1}\ar[d]^{i_{\stree',\ttree',\thetatree}\otimes\mu} \\
\AOp(\thetatree)\ar[r]_{\rho_{\sigma(\thetatree',\underline{e})}} &
\AOp(\thetatree')\otimes I^{m-1} },
\end{equation*}
where we set $\underline{e}|_{\stree'} = (e_{\alpha_1},\dots,e_{\alpha_r})$ and $\underline{e}|_{\ttree'} = (e_{\beta_1},\dots,e_{\beta_s})$ for short,
and $s^{m-\alpha_*} = s^{m-\alpha_1} s^{m-\alpha_2}\cdots s^{m-\alpha_r}$, $s^{m-\beta_*} = s^{m-\beta_1} s^{m-\beta_2}\cdots s^{m-\beta_s}$.
(These composites correspond to the positions of the degeneracies
when we take the pre-image of the subtrees $\stree,\ttree\subset\thetatree$
under the sequence of tree morphisms $\thetatree'\rightarrow\thetatree'/\{e_1\}\rightarrow\dots\rightarrow\thetatree'/\{e_1,\dots,e_m\} = \thetatree$.)

Note that we may have $r = 0$ or $s = 0$ and our diagram is still valid in these cases. (We then take $I^{-1} = \kk$ by convention and $s^0: I^{-1}\rightarrow I^0$
denotes the identity map.)
We actually have three possible cases:
\begin{itemize}
\item $r = 0$:
In this case all the edges of our collection $\underline{e}$ belong to $\ttree'$, we have $\stree = \stree'$,
and the commutativity of the diagram implies that $\partial_{\theta',\underline{e}}\circ i_{\stree\circ_{i_p}\ttree}
= i_{\stree\circ_{i_p}\ttree'}\circ(\id\otimes\partial_{\ttree',\underline{e}})$.
\item $s = 0$:
In this mirror case, all the edges of our collection $\underline{e}$ belong to $\stree'$, we have $\ttree = \ttree'$,
and we get $\partial_{\theta',\underline{e}}\circ i_{\stree\circ_{i_p}\ttree}
= i_{\stree'\circ_{i_p}\ttree}\circ(\partial_{\stree',\underline{e}}\otimes\id)$.
\item $r,s\geq 1$: in this case, the composite of the vertical morphisms on the right hand side of the diagram
does not meet $\AOp(\stree'\circ_{i_p}\ttree')\otimes\underline{01}^{\sharp}{}^{\otimes m-1}$,
because the product of degeneracies carries $I^{r-1}\otimes I^{s-1}$ to a submodule of $I^{m-1}\otimes I^{m-1}$ concentrated in dimension $<m-1$,
whose image under the product can not meet $\underline{01}^{\sharp}{}^{\otimes m-1}$,
so that we have $\partial_{\theta',\underline{e}}\circ i_{\stree\circ_{i_p}\ttree} = 0$.
\end{itemize}
From these identities, we obtain the derivation relation $\partial_m\circ\circ_{i_p} = \circ_{i_p}\circ(\partial_m\otimes\id + \id\otimes\partial_m)$,
valid for each $m\geq 1$. The conclusion follows.
\end{proof}

We still immediately deduce from the associativity of the Segal maps that the composition products of Construction~\ref{construction:bar-complex-homotopy}
satisfy the associativity relations of the composition products of an operad.
We therefore get the following concluding statement:

\begin{thm-defn}
The collection $\DGB^c(\AOp) = \{\DGB^c(\AOp)(r),r>0\}$ equipped with the differential and structure operations defined in Construction~\ref{construction:bar-complex-homotopy}
forms a shuffle operad in dg modules.
This operad $\DGB^c(\AOp)$ is the cobar construction of the connected homotopy Segal shuffle dg (pre-)cooperad $\AOp$.\qed
\end{thm-defn}

\subsection{The definition of homotopy morphisms}\label{subsection:homotopy-morphisms}

We devote this section to the study of homotopy morphisms of Segal cooperads.
We always assume that our target object is equipped with a strict Segal cooperad structure for technical reasons,
but our source object can be equipped with a general homotopy Segal cooperad structure.
We explain the definition of these homotopy morphisms in the context of $E_\infty$-Hopf cooperads first.
We examine the forgetting of $E_\infty$-structures afterwards and then we study the application
of homotopy morphisms to the cobar construction.

\begin{defn}\label{definition:homotopy-morphisms}
We assume that $\BOp$ is a strict Segal $E_\infty$-Hopf shuffle cooperad while $\AOp$ can be any homotopy Segal $E_\infty$-Hopf shuffle cooperad.
We then define a homotopy morphism of homotopy Segal $E_\infty$-Hopf shuffle cooperads $\phi: \AOp\rightarrow\BOp$
as a collection of $\EOp$-algebra morphisms
\begin{equation*}
\phi_{\ttree}: \AOp(\ttree)\rightarrow\BOp(\ttree),\quad\text{$\ttree\in\Tree(r)$, $r>0$},
\end{equation*}
referred to as the underlying maps of our homotopy morphism, together with a collection of higher morphism operators
\begin{equation*}
\phi_{\ttree\rightarrow\ttree_k\rightarrow\dots\rightarrow\ttree_1\rightarrow\stree}: \AOp(\stree)\rightarrow\BOp(\ttree)\otimes I^{k+1},
\end{equation*}
defined in the category of $\EOp$-algebras as well and associated to sequences of composable tree morphisms $\ttree\rightarrow\ttree_k\rightarrow\dots\rightarrow\ttree_1\rightarrow\stree$,
so that the compatibility relations with the face and degeneracy operators expressed by the diagrams of Figure~\ref{homotopy-morphisms:0-faces}-\ref{homotopy-morphisms:degeneracies}
hold, as well the compatibility relations with the facet operators
expressed by the diagrams of Figure~\ref{homotopy-morphisms:facets}-\ref{homotopy-morphisms:homotopy-facets}.
(For the underlying maps of our homotopy morphism, we just retrieve the relation of Definition~\ref{definition:E-infinity-cooperad-morphism}.)
When $\AOp$ and $\BOp$ are symmetric cooperads, we say that $\phi: \AOp\rightarrow\BOp$
defines a homotopy morphism of homotopy Segal $E_\infty$-Hopf symmetric cooperads
if we have also the compatibility relations with the action of permutations
expressed by the diagrams of Figure~\ref{homotopy-morphisms:0-faces}-\ref{homotopy-morphisms:degeneracies}.
(For the underlying maps of our homotopy morphism, we just retrieve the relation of Definition~\ref{definition:E-infinity-cooperad-morphism}.)
\end{defn}

We also have a version of this definition for homotopy Segal shuffle dg cooperads without $E_\infty$-structure.

\begin{defn}\label{definition:homotopy-morphisms-forgetful}
We assume that $\BOp$ be a strict Segal shuffle dg cooperad while $\AOp$ can be any homotopy Segal shuffle dg cooperad.
We then define a homotopy morphism of homotopy Segal shuffle dg cooperads $\phi: \AOp\rightarrow\BOp$
as a collection of morphisms dg modules
\begin{equation*}
\phi_{\ttree}: \AOp(\ttree)\rightarrow\BOp(\ttree),\quad\text{$\ttree\in\Tree(r)$, $r>0$},
\end{equation*}
to which we again refer as the underlying maps of our homotopy morphism, together with a collection of higher morphism operators
\begin{equation*}
\phi_{\ttree\rightarrow\ttree_k\rightarrow\dots\rightarrow\ttree_1\rightarrow\stree}: \AOp(\stree)\rightarrow\BOp(\ttree)\otimes I^{k+1},
\end{equation*}
defined in the category of dg modules as well and associated to sequences of composable tree morphisms $\ttree\rightarrow\ttree_k\rightarrow\dots\rightarrow\ttree_1\rightarrow\stree$,
so that we have the compatibility relations with respect to the face and degeneracy operators expressed by the diagrams
of Figure~\ref{homotopy-morphisms:0-faces}-\ref{homotopy-morphisms:degeneracies} (as in the case of homotopy morphisms $E_\infty$-Hopf cooperads),
together with the compatibility relations with respect to the Segal maps
expressed by the diagrams of Figure~\ref{homotopy-dg-morphisms:Segal-maps}-\ref{homotopy-dg-morphisms:homotopy-Segal-maps}.
\end{defn}

\begin{figure}[p]
\ffigbox
{\caption{The compatibility of homotopy morphisms with $0$-faces.
The diagrams commute for all sequences of composable tree morphisms $\ttree\rightarrow\ttree_k\rightarrow\dots\rightarrow\ttree_1\rightarrow\stree$
and for all $1\leq i\leq k$, where $\widehat{\ttree_i}$ means that we delete the node $\ttree_i$.}\label{homotopy-morphisms:0-faces}}
{\centerline{\xymatrixcolsep{10pc}\xymatrix{ \AOp(\stree)
\ar[r]^{\phi_{\ttree\rightarrow\ttree_k\rightarrow\dots\rightarrow\ttree_1\rightarrow\stree}}
\ar[d]^{\phi_{\ttree_k\rightarrow\dots\rightarrow\ttree_1\rightarrow\stree}} &
\BOp(\ttree)\otimes I^{k+1}\ar[d]^{\id\otimes d^{k+1}_0} \\
\BOp(\ttree_k)\otimes I^k
\ar[r]_{\rho^{\BOp}_{\ttree\rightarrow\ttree_k}\otimes\id} &
\BOp(\ttree)\otimes I^k \\
\AOp(\stree)
\ar[r]^{\phi_{\ttree\rightarrow\ttree_k\rightarrow\dots\rightarrow\ttree_1\rightarrow\stree}}
\ar[dr]_{\phi_{\ttree\rightarrow\cdots\widehat{\ttree_i}\cdots\rightarrow\stree}} &
\BOp(\ttree)\otimes I^{k+1} \ar[d]^{\id\otimes d_0^i} \\
& \BOp(\ttree)\otimes I^k }
}}
\ffigbox
{\caption{The compatibility of homotopy morphisms with $1$-faces.
The diagrams commute for all sequences of composable tree morphisms $\ttree\rightarrow\ttree_k\rightarrow\dots\rightarrow\ttree_1\rightarrow\stree$
and for all $1\leq i\leq k$.}
\label{homotopy-morphisms:1-faces}}
{\centerline{\xymatrixcolsep{4pc}\xymatrix{ \AOp(\stree)
\ar[rr]^{\phi_{\ttree\rightarrow\ttree_k\rightarrow\dots\rightarrow\ttree_1\rightarrow\stree}}
\ar[d]^{\rho^{\AOp}_{\ttree\rightarrow\ttree_k\rightarrow\dots\rightarrow\stree}} &&
\BOp(\ttree)\otimes I^{k+1}
\ar[d]^{\id\otimes d_1^{k+1}} \\
\AOp(\ttree)\otimes I^k
\ar[rr]_{\phi_{\ttree}\otimes\id} &&
\BOp(\ttree)\otimes I^k \\
\AOp(\stree)
\ar[rr]^{\phi_{\ttree\rightarrow\ttree_k\rightarrow\dots\rightarrow\ttree_1\rightarrow\stree}}
\ar[d]^{\rho^{\AOp}_{\ttree_i\rightarrow\dots\rightarrow\ttree_1\rightarrow\stree}} &&
\BOp(\ttree)\otimes I^{k+1}
\ar[d]^{\id\otimes d_1^i} \\
\AOp(\ttree_i)\otimes I^{i-1}
\ar[r]_-{\phi_{\ttree\rightarrow\ttree_k\rightarrow\dots\rightarrow\ttree_i}\otimes\id} &
\BOp(\ttree)\otimes I^{k-i+1}\otimes I^{i-1}\ar[r]_-{\simeq} &
\BOp(\ttree)\otimes I^k }
}}
\ffigbox
{\caption{The compatibility of homotopy morphisms with degeneracies.
The diagrams commute for all sequences of composable tree morphisms $\ttree\rightarrow\ttree_k\rightarrow\dots\rightarrow\ttree_1\rightarrow\stree$
and for all $0\leq j\leq k+1$.}\label{homotopy-morphisms:degeneracies}}
{\centerline{\xymatrixcolsep{10pc}\xymatrix{ \AOp(\stree)
\ar[r]^{\phi_{\stree}}
\ar[dr]_{\phi_{\stree = \stree}} &
\BOp(\stree)
\ar[d]^{\id\otimes s^0} \\
& \BOp(\stree)\otimes I^1 \\
\AOp(\stree)
\ar[r]^{\phi_{\ttree\rightarrow\ttree_k\rightarrow\dots\rightarrow\ttree_1\rightarrow\stree}}
\ar[rd]_{\phi_{\ttree\rightarrow\ttree_k\rightarrow\cdots\ttree_j = \ttree_j\cdots\rightarrow\ttree_1\rightarrow\stree}} &
\BOp(\ttree)\otimes I^{k+1}
\ar[d]^{\id\otimes s^j} \\
& \BOp(\ttree)\otimes I^{k+2} }
}}
\end{figure}

\begin{figure}[p]
\ffigbox
{\caption{The preservation of facet operators by the underlying maps of homotopy morphisms.
The diagram commutes for all subtrees $\sigmatree\subset\stree$.}
\label{homotopy-morphisms:facets}}
{\centerline{\xymatrixcolsep{5pc}\xymatrix{ \AOp(\stree)\ar[r]^{\phi_{\stree}} & \BOp(\stree) \\
\AOp(\sigmatree)\ar[r]^{\phi_{\sigmatree}}\ar[u]^{i_{\sigmatree,\stree}} &
\BOp(\sigmatree)\ar[u]_{i_{\sigmatree,\stree}} }
}}
{\caption{The compatibility of homotopy morphisms with facet operators.
The diagram commutes for all subtrees $\sigmatree\subset\stree$
and for all sequences of composable tree morphisms $\ttree\xrightarrow{f_k}\ttree_k\xrightarrow{f_{k-1}}\dots\xrightarrow{f_1}\ttree_1\xrightarrow{f_0}\stree$.}
\label{homotopy-morphisms:homotopy-facets}}
{\centerline{\xymatrixcolsep{10pc}\xymatrix{ \AOp(\stree)\ar[r]^-{\phi_{\ttree\rightarrow\ttree_k\rightarrow\dots\rightarrow\ttree_1\rightarrow\stree}} &
\BOp(\ttree)\otimes I^{k+1} \\
\AOp(\sigmatree)
\ar@{.>}[r]_-{\phi_{(f_k\dots f_0)^{-1}(\sigmatree)\rightarrow\dots\rightarrow f_0^{-1}(\sigmatree)\rightarrow\sigmatree}}
\ar[u]^{i^{\AOp}_{\sigmatree,\stree}} &
\BOp((f_k\dots f_0)^{-1}(\sigmatree))\otimes I^{k+1}
\ar@{.>}[u]^{i^{\BOp}_{(f_k\dots f_0)^{-1}(\sigmatree),\ttree}\otimes\id} }
}}
\ffigbox
{\caption{The preservation of the action of permutations by the underlying maps of homotopy morphisms.
The diagram commutes for all $s\in\Sigma_r$ and $\ttree\in\Tree(r)$.}
\label{homotopy-morphisms:permutations}}
{\centerline{\xymatrixcolsep{5pc}\xymatrix{ \AOp(s\ttree)\ar[r]^{\phi_{s\ttree}}\ar[d]_{s^*} & \BOp(s\ttree)\ar[d]^{s^*} \\
\AOp(\ttree)\ar[r]^{\phi_{\ttree}} & \BOp(\ttree) }
}}
\ffigbox
{\caption{The compatibility of homotopy morphisms with facet operators.
The diagram commutes for all subtrees $\sigmatree\subset\stree$
and for all sequences of composable tree morphisms $\ttree\xrightarrow{f_k}\ttree_k\xrightarrow{f_{k-1}}\dots\xrightarrow{f_1}\ttree_1\xrightarrow{f_0}\stree$.}
\label{homotopy-morphisms:homotopy-permutations}}
{\centerline{\xymatrixcolsep{10pc}\xymatrix{ \AOp(\stree)\ar[r]^-{\phi_{\ttree\rightarrow\ttree_k\rightarrow\dots\rightarrow\ttree_1\rightarrow\stree}} &
\BOp(\ttree)\otimes I^{k+1} \\
\AOp(\sigmatree)
\ar@{.>}[r]_-{\phi_{(f_k\dots f_0)^{-1}(\sigmatree)\rightarrow\dots\rightarrow f_0^{-1}(\sigmatree)\rightarrow\sigmatree}}
\ar[u]^{i^{\AOp}_{\sigmatree,\stree}} &
\BOp((f_k\dots f_0)^{-1}(\sigmatree))\otimes I^{k+1}
\ar@{.>}[u]^{i^{\BOp}_{(f_k\dots f_0)^{-1}(\sigmatree),\ttree}\otimes\id} }
}}
\end{figure}


\begin{figure}[t]
\ffigbox
{\caption{The preservation of Segal maps by the underlying map of homotopy morphisms of homotopy Segal dg cooperads.
The diagram commutes for all tree decompositions $\stree = \lambda_{\utree}(\sigmatree_v,v\in V(\utree))$.}
\label{homotopy-dg-morphisms:Segal-maps}}
{\centerline{\xymatrixcolsep{5pc}\xymatrix{ \AOp(\stree)\ar[r]^{\phi_{\stree}} & \BOp(\stree) \\
\bigotimes_{v\in V(\utree)}\AOp(\sigmatree_v)
\ar[r]^{\bigotimes_{v\in V(\utree)}\phi_{\sigmatree_v}}
\ar[u]^{i^{\AOp}_{\sigmatree_*,\stree}} &
\bigotimes_{v\in V(\utree)}\BOp(\sigmatree_v)
\ar[u]_{i^{\BOp}_{\sigmatree_*,\stree}} }
}}
\ffigbox
{\caption{The compatibility of homotopy morphisms of homotopy Segal dg cooperads with the Segal maps.
The diagram commutes for all tree decompositions $\stree = \lambda_{\utree}(\sigmatree_v,v\in V(\utree))$
and for all sequences of composable tree morphisms $\ttree\xrightarrow{f_k}\ttree_k\xrightarrow{f_{k-1}}\dots\xrightarrow{f_1}\ttree_1\xrightarrow{f_0}\stree$.}
\label{homotopy-dg-morphisms:homotopy-Segal-maps}}
{\centerline{\xymatrixcolsep{10pc}\xymatrix{ \AOp(\stree)
\ar[r]^-{\phi_{\ttree\rightarrow\ttree_k\rightarrow\dots\rightarrow\ttree_1\rightarrow\stree}} &
\BOp(\ttree)\otimes I^{k+1} \\
& \bigotimes_{v\in V(\utree)}\BOp((f_k\dots f_0)^{-1}(\sigmatree_v))\otimes\left(\bigotimes_{v\in V(\utree)}I^{k+1}\right)
\ar@{.>}[u]_{i^{\BOp}_{(f_k\dots f_0)^{-1}(\sigmatree_*),\ttree}\otimes\mu} \\
\bigotimes_{v\in V(\utree)}\AOp(\sigmatree_v)
\ar@{.>}[r]^-{\bigotimes_v \phi_{(f_k\dots f_0)^{-1}(\sigmatree_v)\rightarrow\dots\rightarrow f_0^{-1} (\sigmatree_v)\rightarrow\sigmatree_v}}
\ar[uu]^{i^{\AOp}_{\sigmatree_*,\stree}} &
\bigotimes_{v \in V(\utree)}\left(\BOp((f_k\dots f_0)^{-1}(\sigmatree_v))\otimes I^{k+1}\right)
\ar@{.>}[u]_{\simeq} }
}}
\end{figure}

\afterpage{\clearpage}

We have the following statement, which is the homotopy version of the result of Lemma~\ref{lemma:homotopy-cooperad-top-component},
and which can be proved by the same arguments.
We still write $\alpha^{\square}: X\rightarrow Y$ for the homomorphism of graded modules of degree $k$
associated to a morphism of dg modules $\alpha: X\rightarrow Y \otimes I^k$
such that $\alpha(x) = (-1)^{k\deg(x)}\alpha^{\square}(x)\otimes\underline{01}^{\sharp}{}^{\otimes k} + \text{tensors with a factor of dimension $<k$ in $I^k$}$.

\begin{lemm}\label{lemma:homotopy-morphism-top-component}
\begin{enumerate}
\item
Let $\phi: \AOp\rightarrow\BOp$ be a homotopy morphism of homotopy Segal shuffle dg cooperads,
where we still assume that $\BOp$ is a strict Segal shuffle dg cooperad
as in Definition~\ref{definition:homotopy-morphisms-forgetful}.
The graded homomorphism of degree $k+1$
\begin{equation*}
\phi^{\square}_{\ttree\rightarrow\ttree_k\rightarrow\dots\rightarrow\ttree_1\rightarrow\stree}: \AOp(\stree)\rightarrow\AOp(\ttree)
\end{equation*}
that we associate to the dg module morphism $\phi_{\ttree\rightarrow\ttree_k\rightarrow\dots\rightarrow\ttree_1\rightarrow\stree}: \AOp(\stree)\rightarrow\AOp(\ttree)\otimes I^{k+1}$,
for any sequence of composable tree morphisms $\ttree\rightarrow\ttree_k\rightarrow\dots\rightarrow\ttree_1\rightarrow\stree$,
satisfies the relation
\begin{equation}\tag{*}\label{equation:homotopy-morphism-top-component}
\delta(\phi^{\square}_{\ttree\rightarrow\ttree_k\rightarrow\dots\rightarrow\ttree_1\rightarrow\stree})
\begin{aligned}[t]
& = (-1)^{k+1}\rho^{\BOp}_{\ttree_k\rightarrow\stree}\phi^{\square}_{\ttree_k\rightarrow\dots\rightarrow\ttree_1\rightarrow\stree} \\
& + \sum_{i=1}^k(-1)^i\phi^{\square}_{\ttree\rightarrow\dots\rightarrow\hat{\ttree}_i\rightarrow\dots\rightarrow\stree} \\
& + (-1)^k\phi_{\ttree}\circ\rho^{\AOp\square}_{\ttree\rightarrow\ttree_k\rightarrow\dots\rightarrow\ttree_1\rightarrow\stree} \\
& - \sum_{i=1}^k(-1)^i\phi^{\square}_{\ttree\rightarrow\ttree_k\rightarrow\dots\rightarrow\ttree_i}\circ\rho^{\AOp\square}_{\ttree_i\rightarrow\dots\rightarrow\ttree_1\rightarrow\stree}.
\end{aligned}
\end{equation}
Moreover, if we have a degeneracy $\ttree_j = \ttree_{j+1}$ in our sequence of tree morphisms, for some $0\leq j\leq k$ (with the convention that $\ttree_{k+1} = \ttree$ and $\ttree_0 = \stree$),
then we have the relation
\begin{equation}\tag{**}\label{equation:homotopy-morphism-degeneracy-top-component}
\phi^{\square}_{\ttree\rightarrow\ttree_{k}\rightarrow\dots\ttree_{j+1}=\ttree_j\cdots\rightarrow\ttree_1\rightarrow\stree} = 0.
\end{equation}
\item
In the converse direction, if we have a collection of dg module morphisms $\phi_{\ttree}: \AOp(\ttree)\rightarrow\BOp(\ttree)$, $\ttree\in\Tree(r)$, $r>0$,
together with a collection of dg graded homomorphisms $\phi^{\square}_{\ttree\rightarrow\ttree_k\rightarrow\dots\rightarrow\ttree_1\rightarrow\stree}: \AOp(\stree)\rightarrow\AOp(\ttree)$,
of degree $k+1$,
which satisfy the relations (\ref{equation:homotopy-morphism-top-component})-(\ref{equation:homotopy-morphism-degeneracy-top-component})
of the previous statement,
then there is a unique collection of morphisms of dg modules $\phi_{\ttree\rightarrow\ttree_k\rightarrow\dots\rightarrow\ttree_1\rightarrow\stree}: \AOp(\stree)\rightarrow\AOp(\ttree)\otimes I^k$,
which extend these maps on the summands $\AOp(\ttree)\otimes\underline{01}^{\sharp}{}^{\otimes k+1}$
and satisfy the face and degeneracy relations of homotopy morphism operators
of Figure~\ref{homotopy-morphisms:0-faces}-\ref{homotopy-morphisms:degeneracies}.\qed
\end{enumerate}
\end{lemm}

We now prove that any homotopy morphism of homotopy Segal shuffle dg cooperads $\phi: \AOp\rightarrow\BOp$, as in Definition~\ref{definition:homotopy-morphisms-forgetful},
gives rise to an induced morphism on the cobar construction $\phi_*: \DGB^c(\AOp)\rightarrow\DGB^c(\BOp)$.
We address the definition of this morphism in the next paragraph.
We assume all along this study that a homotopy morphism $\phi: \AOp\rightarrow\BOp$ is fixed, with $\AOp$ a homotopy Segal shuffle dg cooperad
and $\BOp$ a strict Segal shuffle dg cooperad.
We need to assume that the object $\AOp$ is connected in order to give a sense to the cobar construction $\DGB^c(\AOp)$ (see~\S\ref{subsection:homotopy-Segal-cooperad-cobar}).
We also need to assume that $\BOp$ is connected in the construction of our morphisms.
We therefore assume that these connectedness conditions hold in the rest of this subsection.

\begin{constr}\label{construction:bar-complex-homotopy-morphism}
The underlying maps of our homotopy morphism $\phi: \AOp\rightarrow\BOp$ induce morphisms of graded modules between the components of the cobar construction:
\begin{equation*}
\phi_{\ttree}: \DGSigma^{-\sharp V(\ttree)}\AOp(\ttree)\rightarrow\DGSigma^{-\sharp V(\ttree)}\AOp(\ttree).
\end{equation*}
In addition to these maps, we consider morphisms
\begin{equation*}
\phi_{(\ttree,\underline{e})}: \DGSigma^{-\sharp V(\ttree)+m}\AOp(\ttree/\{e_1,\dots,e_m\})\rightarrow\DGSigma^{-\sharp V(\ttree)}\BOp(\ttree).
\end{equation*}
associated to the pairs $(\ttree,\underline{e})$, where $\ttree$ is a tree and $\underline{e} = (e_1,\dots,e_m)$
is an ordered collection of pairwise distinct edges $e_i\in\mathring{E}(\ttree)$,
as in Construction~\ref{construction:bar-complex-homotopy}.
To define the latter maps, we again consider the sequence of composable tree morphisms
\begin{gather*}
\sigma(\ttree,\underline{e}) = \{\ttree\rightarrow\ttree/{e_1}\rightarrow\ttree/{\{e_1,e_2\}}\rightarrow\dots\rightarrow\ttree/\{e_1,\dots,e_m\}\},
\intertext{which we associate to any such pair $(\ttree,\underline{e})$ in Construction~\ref{construction:bar-complex-homotopy},
and we set}
\phi_{(\ttree,\underline{e})} = \phi^{\square}_{\sigma(\ttree,\underline{e})},
\end{gather*}
where we take the top component of the morphism $\phi_{\sigma(\ttree,\underline{e})}$ (such as defined in Lemma~\ref{lemma:homotopy-morphism-top-component}).
In this construction, we also use the same blow-up process as in Construction~\ref{construction:bar-complex-homotopy}
to pass from the tensor product $\DGSigma^{-\sharp V(\ttree)+m}\AOp(\ttree/\{e_1,\dots,e_m\})
= \bigl(\bigotimes_{x\in V(\ttree/\{e_1,\dots,e_m\})}\underline{01}^{\sharp}_x\bigr)\otimes\AOp(\ttree/\{e_1,\dots,e_m\})$
to $\DGSigma^{-\sharp V(\ttree)}\BOp(\ttree) = \bigl(\bigotimes_{x\in V(\ttree)}\underline{01}^{\sharp}_x\bigr)\otimes\AOp(\ttree)$
and to determine a possible sign, which we associate to our map $\phi_{(\ttree,\underline{e})}$.
In what follows, we identify the morphisms $\phi_{\ttree}$, induced by the underlying maps of our homotopy morphism $\phi: \AOp\rightarrow\BOp$,
with the case $m=0$ of these homomorphisms $\phi_{(\ttree,\underline{e})}$.

Finally, we take:
\begin{gather*}
\phi_m = \sum_{(\ttree,(e_1,\dots,e_m)}\phi_{(\ttree,(e_1,\dots,e_m))},\quad\text{for $m\geq 0$},
\quad\text{and}\quad\phi = \sum_{m\geq 0}\phi_m
\intertext{to get a map}
\phi_*: \DGB^c(\AOp)(r)\rightarrow\DGB^c(\BOp)(r),\quad\text{for each arity $r>0$}.
\end{gather*}
(Note that we use the connectedness condition on $\BOp$ to ensure that the above sum reduces to a finite number of terms on each summand $\DGSigma^{-\sharp V(\ttree)}\AOp(\ttree)$.)
We aim to prove that this map is compatible with the structure operations of the cobar construction.
\end{constr}

We check the preservation of differentials first. This claim follows from the following more precise observation.

\begin{lemm}\label{lemma:differential-bar-morphisms}
We have the relation
\begin{equation*}
\delta^{\BOp}\phi_m = \phi_m \delta^{\AOp} + \partial^{\BOp}\phi_{m-1} - \sum_{i=0}^{m-1}\phi_i\partial_{m-i}^{\AOp},
\end{equation*}
for all $m\geq 0$,
where $\delta = \delta^{\AOp},\delta^{\BOp}$ denotes the term of the differential of the cobar construction induced by the internal differential of the objects $\COp = \AOp,\BOp$,
we denote by $\partial^{\AOp} = \sum_{m = 1}^{\infty}\partial_m^{\AOp}$ the twisting map of the cobar construction
of the homotopy Segal shuffle dg cooperad $\AOp$,
while $\partial^{\BOp}$ denotes the twisting differential of the cobar construction
of the strict Segal shuffle dg cooperad $\AOp$.
\end{lemm}

\begin{proof}
We argue as in the proof of Lemma \ref{lemma:differential-bar-construction-homotopy}.
We use the relation of Equation~(\ref{equation:homotopy-morphism-top-component}) of Lemma~\ref{lemma:homotopy-morphism-top-component}
to write
\begin{align*}
\delta^{\BOp}\phi_{(\ttree,\underline{e})} - \phi_{(\ttree,\underline{e})}\delta^{\AOp} &
= \sum_{i=1}^{m-1}\pm\phi^{\square}_{\ttree\rightarrow\dots\rightarrow\widehat{\ttree/\{e_1,\dots,e_i\}}\rightarrow\dots\rightarrow\ttree/\{e_1,\dots,e_m\}} \\
& - \rho^{\BOp}_{\ttree\rightarrow\ttree/e_1}\phi^{\square}_{\ttree/e_1\rightarrow\dots\rightarrow\ttree/\{e_1,\dots,e_m\}}\\
& + \sum_{i=0}^{m-1}\pm\phi^{\square}_{\ttree\rightarrow\dots\rightarrow\ttree/\{e_1,\dots,e_i\}}\rho^{\AOp\square}_{\ttree/\{e_1,\dots,e_i\}\rightarrow\dots\rightarrow\ttree/\{e_1,\dots,e_m\}}.
\end{align*}
Then, by taking the sum of these expressions over the set of pairs $(\ttree,\underline{e})$, we obtain the formula:
\begin{align*}
\delta^{\BOp}\phi_m - \phi_m\delta^{\AOp} &
= \sum_{\substack{(\ttree,\underline{e})\\i=1,\dots,m}}\pm\phi^{\square}_{\ttree\rightarrow\dots\rightarrow\widehat{\ttree/\{e_1,\dots,e_i\}}\rightarrow\dots\rightarrow\ttree/\{e_1,\dots,e_m\}} \\
& + \underbrace{\sum_{(\ttree,\underline{e})}\pm\rho^{\BOp}_{\ttree\rightarrow\ttree/e_1}\phi^{\square}_{\ttree/e_1\rightarrow\dots\rightarrow\ttree/\{e_1,\dots,e_m\}}}_{=\partial^{\BOp}\phi_{m-1}} \\
& + \sum_{i=0}^{m-1}\bigl(\underbrace{\sum_{(\ttree,\underline{e})}\pm\phi^{\square}_{\ttree\rightarrow\dots\rightarrow\ttree/\{e_1,\dots,e_i\}}
\circ\rho^{\AOp\square}_{\ttree/\{e_1,\dots,e_i\}\rightarrow\dots\rightarrow\ttree/\{e_1,\dots,e_m\}}}_{= \phi_i\partial_{m-i}}\bigr).
\end{align*}
In the first sum of this formula, the term that corresponds to the removal of the node $\ttree/\{e_1,\dots,e_{i-1},e_i\}$
and the term that corresponds to the removal of the node $\ttree/\{e_1,\dots,e_{i-1},e_{i+1}\}$
for the pair $(\ttree,(e_1,\dots,e_{i-1},e_{i+1},e_i,e_{i+2},\dots,e_m))$
with $e_i$ and $e_{i+1}$ switched
are again equal up to a sign. We can still check that these signs are opposite, so that these terms cancel out in our sum. The conclusion of the lemma follows.
\end{proof}


We now check that our morphisms preserves the composition products. This claim follows from the following more precise observation.

\begin{lemm}\label{lemma:morphism-cobar-operad}
We have the relation $\phi_m\circ\circ_{i_p} = \sum_{r+s=m}\circ_{i_p}\circ(\phi_r\otimes\phi_s)$, for all $m\geq 0$.
\end{lemm}

\begin{proof}
The proof is similar to that of Lemma~\ref{lemma:bar-construction-product-operad-homotopy}.

For $m = 0$, the relation follows the commutativity of the diagram of Figure~\ref{homotopy-dg-morphisms:Segal-maps} (since $\circ_{i_p}$ is defined as a sum of Segal maps).
We therefore focus on the case $m\geq 1$.
We consider again a summand $\Sigma^{-\sharp V(\stree)}\AOp(\stree)\otimes\Sigma^{-\sharp V(\ttree)}\AOp(\ttree)$
of the tensor product $\DGB^c(\AOp)(\{i_1<\dots<i_k\})\otimes\DGB^c(\AOp)(\{j_1<\dots<j_l\})$,
where $\stree\in\Tree(\{i_1<\dots<i_k\})$, $\ttree\in\Tree(\{j_1<\dots<j_l\})$.
The composition product maps $\circ_{i_p}$ carry this summand into $\Sigma^{-\sharp V(\thetatree)}\AOp(\thetatree)$,
with $\thetatree = \stree\circ_{i_p}\ttree$.
The components of the map $\phi_m$ carry this summand into terms of the form $\Sigma^{-\sharp V(\thetatree')}\AOp(\thetatree')$,
for trees $\thetatree'$ equipped with a set of internal edges $(e_1,\dots,e_m)$
such that $\thetatree'/\{e_1,\dots,e_m\} = \thetatree$.
We still have $\thetatree' = \stree'\circ_{i_p}\ttree'$ and $\stree = \stree'/\{e_{\alpha_1},\dots,e_{\alpha_r}\}$, $\ttree = \ttree'/\{e_{\beta_1},\dots,e_{\beta_s}\}$,
for a partition $\{e_{\alpha_1},\dots,e_{\alpha_r}\}\amalg\{e_{\beta_1},\dots,e_{\beta_s}\} = \{e_1,\dots,e_m\}$
such that $e_{\alpha_1},\dots,e_{\alpha_r}\in\mathring{E}(\stree')$ and $e_{\beta_1},\dots,e_{\beta_s}\in\mathring{E}(\ttree')$.
We then have the following commutative diagram:
\begin{equation*}
\xymatrixcolsep{8pc}\xymatrix{
\AOp(\stree)\otimes\AOp(\ttree)
\ar[r]^-{\phi_{\sigma(\stree',\underline{e}|_{\stree'})}\otimes\phi_{\sigma(\ttree',\underline{e}|_{\ttree'})}}
\ar[ddd]|{i_{\stree\circ_{i_p}\ttree}} &
(\BOp(\stree')\otimes I^r)\otimes(\BOp(\ttree')\otimes I^s)
\ar[d]|{\id\otimes s^{m-\beta_*}\otimes\id\otimes s^{m-\alpha_*}} \\
& (\BOp(\stree')\otimes I^m)\otimes(\BOp(\ttree')\otimes I^m)
\ar[d]^{\simeq} \\
& \BOp(\stree')\otimes\BOp(\ttree')\otimes I^m\otimes I^m
\ar[d]^{i_{\stree'\circ_{i_p}\ttree'}\otimes\mu} \\
\AOp(\thetatree)\ar[r]_{\phi_{\sigma(\thetatree',\underline{e})}} &
\BOp(\thetatree')\otimes I^m }
\end{equation*}
(by the relations of Figure~\ref{homotopy-dg-morphisms:homotopy-Segal-maps} and Figure~\ref{homotopy-morphisms:degeneracies}).
We use the same notation as in the proof of Lemma~\ref{lemma:bar-construction-product-operad-homotopy}
in this diagram
and we still consider the morphisms $\phi_{\sigma(\stree',\underline{e}|_{\stree'})}: \AOp(\stree'/\{e_{\alpha_1},\dots,e_{\alpha_r}\})\rightarrow\AOp(\stree')\otimes I^r$
and $\phi_{\sigma(\ttree',\underline{e}|_{\ttree'})}: \AOp(\ttree'/\{e_{\beta_1},\dots,e_{\beta_s}\})\rightarrow\AOp(\ttree')\otimes I^s$
associated to the sequences of tree morphisms
such that $\sigma(\stree',\underline{e}|_{\stree'}) = \{\stree'\rightarrow\stree'/e_{\alpha_1}\rightarrow\dots\rightarrow\stree'/\{e_{\alpha_1},\dots,e_{\alpha_r}\}\}$
and $\sigma(\ttree',\underline{e}|_{\ttree'}) = \{\ttree'\rightarrow\ttree'/e_{\beta_1}\rightarrow\dots\rightarrow\ttree'/\{e_{\beta_1},\dots,e_{\beta_s}\}\}$.
We also use the notation $s^{m-\alpha_*} = s^{m-\alpha_1} s^{m-\alpha_2}\cdots s^{m-\alpha_r}$, $s^{m-\beta_*} = s^{m-\beta_1} s^{m-\beta_2}\cdots s^{m-\beta_s}$,
and $\mu: I^m\otimes I^m\rightarrow I^m$ denotes the product of the dg algebra $I^m$
as usual.

We see, by elaborating on the arguments of the proof of Lemma~\ref{lemma:bar-construction-product-operad-homotopy},
that the composite of the right-hand side vertical morphisms of this diagram does not meet $\BOp(\thetatree')\otimes\underline{01}^{\sharp}{}^{\otimes m}$
unless we have $\{e_{\alpha_1},\dots,e_{\alpha_r}\} = \{e_1,\dots,e_r\}$ and $\{e_{\beta_1},\dots,e_{\beta_s}) = (e_{r+1},\dots,e_m\}$. (We have in this case
$s^{m-\beta_*}(\underline{01}^{\sharp}{}^{\otimes r}) = \underline{01}^{\sharp}{}^{\otimes r}\otimes 1^{\otimes s}$,
$s^{m-\alpha_*}(\underline{01}^{\sharp}{}^{\otimes s}) = \underline{1}^{\sharp}{}^{\otimes r}\otimes\underline{01}^{\sharp}{}^{\otimes s} + \text{other terms}$,
and $\mu(s^{m-\beta_*}(\underline{01}^{\sharp}{}^{\otimes r})\otimes s^{m-\alpha_*}(\underline{01}^{\sharp}{}^{\otimes s})) = \underline{01}^{\sharp}{}^{\otimes m}$.)
We conclude from this analysis that the composite $\phi_{(\thetatree',\underline{e})}\circ\circ_{i_p}$ vanishes unless the edge collection $\underline{e} = (e_1,\dots,e_m)$
is equipped with an order such that $\{e_{\alpha_1},\dots,e_{\alpha_r}\} = \{e_1,\dots,e_r\}$
and $\{e_{\beta_1},\dots,e_{\beta_s}\} = \{e_{r+1},\dots,e_m\}$.
We get in this case $\phi_{(\thetatree',\underline{e})}\circ\circ_{i_p} = \circ_{i_p}\circ\phi_{(\stree',\underline{e}|_{\stree'})}\otimes\phi_{(\ttree',\underline{e}|_{\ttree'})}$
and summing over the pairs $(\thetatree',(e_1,\dots,e_m))$ with $\thetatree' = \stree'\circ_{i_p}\ttree'$, $e_1,\dots,e_r\in\mathring{E}(\stree')$, $e_{r+1},\dots,e_m\in\mathring{E}(\stree')$,
amounts to summing over the pairs $(\stree',(e_1,\dots,e_r))$ and $(\ttree',(e_{r+1},\dots,e_m))$
such that $\stree'/\{e_1,\dots,e_r\} = \stree$ and $\ttree'/\{e_1,\dots,e_r\} = \ttree$.
We therefore obtain the relation of the lemma $\phi_m\circ\circ_{i_p} = \sum_{r+s=m}\circ_{i_p}\circ(\phi_r\otimes\phi_s)$
when we perform this sum.
\end{proof}

We get the following concluding statement:

\begin{thm-defn}
The collection of morphisms $\phi_*: \DGB^c(\AOp)(r)\rightarrow\DGB^c(\BOp)(r)$, $r>0$, defined in Construction~\ref{construction:bar-complex-homotopy-morphism},
defines a morphism of shuffle dg operads $\phi_*: \DGB^c(\AOp)\rightarrow\DGB^c(\BOp)$,
the morphism induced by the homotopy morphism of connected homotopy Segal shuffle cooperads $\phi: \AOp\rightarrow\BOp$
on the cobar construction.\qed
\end{thm-defn}







\subsection{The equivalence with strict $ E_\infty $-cooperads}\label{subsection:strict-equivalence}

We devote this final subsection to proving the following result.

\begin{thm}\label{theorem:strictification}
Every connected homotopy Segal $E_\infty$-Hopf cooperad (either symmetric or shuffle)
is weakly-equivalent to a connected strict Segal $E_\infty$-Hopf cooperad.
\end{thm}

We prove Theorem~\ref{theorem:strictification} by constructing a functor $\AOp\mapsto\DGK^c(\AOp)$, from the category of connected homotopy Segal $E_\infty$-Hopf cooperads (symmetric or shuffle)
to the category of strict Segal $E_\infty$-Hopf cooperads, and a zigzag of weak-equivalences between $\AOp$ and $\DGK^c(\AOp)$.

For this purpose, we consider a category $\Tree^{\square}$, enriched in dg modules, which encodes the homotopy coproduct operators of homotopy Segal cooperads.
We explain the definition of this category $\Tree^{\square}$ in the next paragraph.
We have a morphism of enriched categories $\Tree^{\square}\rightarrow\Tree$, where, by an abuse of notation, we denote by $\Tree$ the enriched category in $\kk$-modules
whose hom-objects are the $\kk$-modules spanned by the set-theoretic tree morphisms.
We define our functor $\DGK^c(-)$ as a homotopy Kan extension, by dualizing a two-sided bar complex over the enriched category $\Tree^{\square}$.

Note that, in our statement, we again assume that our Segal cooperad $\AOp$ is connected in the sense of~\S\ref{subsection:conilpotence}.
This assumption enables us to simplify our constructions and to avoid technical difficulties
in the verification of our result. We assume that our cooperads satisfy this connectedness condition all along this section.

Recall that a (homotopy or strict) Segal cooperad $\AOp$ is connected if we have $\AOp(\ttree) = 0$
when the tree $\ttree$ is not reduced (has at least one vertex with a single ingoing edge)
and that this condition implies that we can restrict ourselves to the subcategories of reduced trees, denoted by $\widetilde{\Tree}(r)\subset\Tree(r)$, $r>0$,
in the definition of the structure operations that we associate to our objects.
For simplicity, all along this subsection, we keep the notation $\Tree$ for our constructions on tree categories (for instance, we use the notation $\Tree^{\square}$
for our enriched category of trees).
Nevertheless, we restrict ourselves to reduced trees, as permitted by our connectedness assumption on cooperads,
and for this reason, we only define the enriched hom-objects $\Tree^{\square}(\ttree,\stree)$ associated to the full subcategories of reduced trees $\widetilde{\Tree}^{\square}(r)$, $r>0$.

Recall that for reduced trees $\stree,\ttree\in\widetilde{\Tree}(r)$, the set of tree morphisms $\Tree(\ttree,\stree)$
is either empty or reduced to a point (see \cite[Theorem B.0.6]{FresseBook}).
For the enriched version of this category, we therefore have:
\begin{equation*}
\Tree(\ttree,\stree) = \begin{cases} \kk, & \text{if we have a morphism $\ttree\rightarrow\stree$}, \\
0, & \text{otherwise}, \end{cases}
\end{equation*}
for any pair of reduced trees $\stree,\ttree\in\widetilde{\Tree}(r)$, $r>0$.
Note that we may still write $\Tree(\ttree,\stree) = *$ in this setting, because we identify the object $\Tree(\ttree,\stree) = \kk$
with the terminal object of the category of cocommutative coalgebras, and we can actually regard $\Tree$
as a category enriched in cocommutative coalgebras.
This observation, to which we go back later on, motivates our abuse of notation.


\begin{constr}\label{construction:W-homotopy}
We define the enriched category $\Tree^{\square}$ in this paragraph. We take the same set of objects as the category of reduced trees $\widetilde{\Tree}$.

In the definition of the hom-objects, we consider the cubical chain complexes $\DGN_*(\Delta^1)^{\otimes k}$, dual to the cubical cochain algebras $I^k = \DGN^*(\Delta^1)^{\otimes k}$
of the definition of homotopy Segal cooperads.
We use the coface operators $d_i^{\epsilon}: \DGN_*(\Delta^1)^{\otimes k-1}\rightarrow\DGN_*(\Delta^1)^{\otimes k}$
and the codegeneracy operators $s_j: \DGN_*(\Delta^1)^{\otimes k}\rightarrow\DGN_*(\Delta^1)^{\otimes k-1}$
dual to the operators $d^i_{\epsilon}: I^k\rightarrow I^{k-1}$, $i = 1,\dots,k$, $\epsilon = 0,1$,
and $s^j: I^{k-1}\rightarrow I^k$, $j = 0,\dots,k$,
considered in Construction~\ref{constr:cubical-cochain-algebras}.
We then have $d_i^{\epsilon} = \id^{\otimes k-i}\otimes d^{\epsilon}\otimes\id^{\otimes i-1}$, $s_0 = \id^{\otimes k-1}\otimes s^0$,
$s_j = \id^{\otimes k-j-1}\otimes\nabla_*\otimes\id^{\otimes j-1}$, for $j = 1,\dots,k-1$,
and $s_k = s^k\otimes\id^{\otimes k-1}$,
where $d^{\epsilon}: \kk = \DGN_*(\Delta^0)\rightarrow\DGN_*(\Delta^1)$, $\epsilon = 0,1$, and $s^0: \DGN_*(\Delta^1)\rightarrow\DGN_*(\Delta^0) = \kk$
are the cofaces and the codegeneracy of the normalized chain complex of the one-simplex,
while $\nabla_*: \DGN_*(\Delta^1)\otimes\DGN_*(\Delta^1)\rightarrow\DGN_*(\Delta^1)$
denotes the connection of Construction~\ref{constr:cubical-cochain-connection}.

We precisely define the dg module $\Tree^{\square}(\ttree,\stree)$, which represents the dg hom-object associated to a pair of reduced trees $\stree,\ttree\in\widetilde{\Tree}$
such that $\ttree\not=\stree$ in our enriched category, by the following quotient
\begin{equation*}
\Tree^{\square}(\ttree,\stree)
= \bigoplus_{k\geq 0}\left(\bigoplus_{\ttree\rightarrow\ttree_k\rightarrow\dots\rightarrow\ttree_1\rightarrow\stree}
\DGN_*(\Delta^1)^{\otimes k}_{\ttree\rightarrow\ttree_k\rightarrow\dots\rightarrow\ttree_1\rightarrow\stree}
\right)/\equiv,
\end{equation*}
where a copy of the cubical chain complex $\DGN_*(\Delta^1)^{\otimes k}$ is assigned to every sequence
of composable tree morphisms $\ttree\rightarrow\ttree_k\rightarrow\dots\rightarrow\ttree_1\rightarrow\stree$,
and we mod out by the relations
\begin{align*}
d^0_i(\underline{\sigma})_{\ttree\rightarrow\cdots\rightarrow\stree}
& \equiv\underline{\sigma}_{\ttree\rightarrow\cdots\widehat{\ttree_i}\cdots\rightarrow\stree}, \\
s_j(\underline{\sigma})_{\ttree\rightarrow\cdots\rightarrow\stree}
& \equiv\underline{\sigma}_{\ttree\rightarrow\cdots\ttree_j = \ttree_j\cdots\rightarrow\stree},
\end{align*}
where $\underline{\sigma}$ denotes an element of the cubical chain complex (of appropriate dimension).
The composition operations of this enriched category
\begin{equation*}
\circ: \Tree^{\square}(\utree,\stree)\otimes\Tree^{\square}(\ttree,\utree)\rightarrow\Tree^{\square}(\ttree,\stree)
\end{equation*}
are given by
\begin{multline}\tag{$*$}\label{cubical-enriched-category:composition}
\underline{\sigma}_{\utree\rightarrow\utree_k\rightarrow\dots\rightarrow\utree_1\rightarrow\stree}
\circ\underline{\tau}_{\ttree\rightarrow\ttree_l\rightarrow\dots\rightarrow\ttree_1\rightarrow\utree}\\
= (\underline{\tau}\otimes\underline{0}\otimes\underline{\sigma})_{\ttree\rightarrow\ttree_l\rightarrow\dots\rightarrow\ttree_1\rightarrow\utree
\rightarrow\utree_k\rightarrow\dots\rightarrow\utree_1\rightarrow\stree}
\end{multline}
as long as $\utree\not=\stree$ and $\ttree\not=\utree$. We just take in addition $\Tree^{\square}(\stree,\stree) = \kk$ to provide our category with identity homomorphisms.
We easily check that the above formula preserves the relations of our hom-objects
and hence gives a well-defined morphism
of dg modules. We immediately see that these composition operations
are associative too.

We can also define the objects $\Tree^{\square}(\ttree,\stree)$ in terms of a coend.
We then consider an indexing category $\CubeCat$ generated by the $0$-cofaces $d^0_i$ and the codegeneracies $s^j$
attached to our cubical chain complexes
and which reflect the face and degeneracy operations that we apply to the sequences of composable tree morphisms.
The objects of this category are the ordinals $\underline{k+2} = \{k+1>k>\dots>1>0\}$, with $k\geq 0$.
The morphisms are the non decreasing maps $u: \underline{k+2}\rightarrow\underline{l+2}$
such that $u(k+1) = l+1$ and $u(0) = 0$.
The coface $d^0_i$ corresponds to the map $d^0_i: \underline{k+1}\rightarrow\underline{k+2}$ that jumps over $i+1$ in $\underline{k+2}$,
while the codegeneracy $s_j$ corresponds to the map $s_j: \underline{k+2}\rightarrow\underline{k+1}$
such that $s_j(x) = x$ for $x = 0,\dots,j$ and $s_j(x) = x-1$ for $x = j+1,\dots,k+2$.
We easily check that the collection of cubical chain complexes $\DGN_*(\Delta^1)^{\otimes k}$,
equipped with the previously defined operators $d^0_i: \DGN_*(\Delta^1)^{\otimes k-1}\rightarrow\DGN_*(\Delta^1)^{\otimes k}$
and $s_j: \DGN_*(\Delta^1)^{\otimes k}\rightarrow\DGN_*(\Delta^1)^{\otimes k-1}$,
defines a functor $\underline{k+2}\mapsto\DGN_*(\Delta^1)^{\otimes k}$ on this category $\CubeCat$.

In general, we denote by $X_k$ the image of an object $\underline{k+2}\in\CubeCat$
under a (contravariant or covariant) functor $X$
over the category $\CubeCat$.

For a pair of reduced trees $\ttree,\stree\in\widetilde{\Tree}(r)$ with $\ttree\not=\stree$,
we also consider the functor $\NCat(\ttree,\stree): \CubeCat^{op}\rightarrow\Set$
such that
\begin{equation*}
\NCat(\ttree,\stree)_k = \bigl\{\ttree\rightarrow\ttree_k\rightarrow\dots\rightarrow\ttree_1\rightarrow\stree|\ttree_i\in\widetilde{\Tree}(r)\,(\forall i)\bigr\},
\end{equation*}
where we consider the set of all sequences of composable tree morphisms of length $k+1$ with $\ttree_{k+1} = \ttree$, $\ttree_0 = \stree$,
and we equip this set with the obvious action of the category $\CubeCat$ (we adapt the usual definition
of the simplicial nerve of a category).
We then have:
\begin{equation*}
\Tree^{\square}(\ttree,\stree) = \int^{\underline{k+2}\in\CubeCat}\kk[\NCat(\ttree,\stree)_k]\otimes\DGN_*(\Delta^1)^{\otimes k},
\end{equation*}
where $\kk[\NCat(\ttree,\stree)_k]$ is the $\kk$-module generated by the set $\NCat(\ttree,\stree)_k$.
We can also express the composition operation of the category $\Tree^{\square}$ in terms of a termwise composition operation on this coend
which we define by the above formula~(\ref{cubical-enriched-category:composition}).

We have well-defined (contravariant) facet operators
\begin{equation*}
i_{\sigmatree,\stree}: \Tree^{\square}(\ttree,\stree)\rightarrow\Tree^{\square}(\thetatree,\sigmatree),
\end{equation*}
which we associate to all subtree inclusions $\sigmatree\subset\stree$, where $\thetatree = f^{-1}(\sigmatree)$ is the pre-image of the subtree $\sigmatree\subset\stree$
under the (at most unique) morphism $f: \ttree\rightarrow\stree$.
We define these facet operators on our coend termwise,
by the map
\begin{equation*}
i_{\sigmatree,\stree}: \kk[\NCat(\ttree,\stree)_k]\otimes\DGN_*(\Delta^1)^{\otimes k}\rightarrow\kk[\NCat(\thetatree,\sigmatree)_k]\otimes\DGN_*(\Delta^1)^{\otimes k}
\end{equation*}
induced by the set-theoretic facet operator $\NCat(\ttree,\stree)_k\rightarrow\NCat(\thetatree,\sigmatree)_k$
which carries any sequence of composable tree morphisms $\ttree\xrightarrow{f_k}\ttree_k\xrightarrow{f_{k-1}}\cdots\xrightarrow{f_1}\ttree_1\xrightarrow{f_0}\stree$
to the sequence of tree morphisms such that $(f_0\cdots f_k)^{-1}(\sigmatree)\rightarrow(f_0\cdots f_{k-1})^{-1}(\sigmatree)\rightarrow\cdots\rightarrow f_0^{-1}(\sigmatree)\rightarrow\sigmatree$,
where we use $\thetatree = (f_0\cdots f_k)^{-1}(\sigmatree)$.
We can also associate a Segal map
\begin{equation*}
i_{\lambda_{\utree}(\sigmatree_*)}: \Tree^{\square}(\ttree,\stree)\rightarrow\bigotimes_{u\in V(\utree)}\Tree^{\square}(\thetatree_u,\sigmatree_u),
\end{equation*}
to every tree decomposition $\stree = \lambda_{\utree}(\sigmatree_*)$, where we set $\thetatree_u = f^{-1}(\sigmatree_u)$, for all factors $\sigmatree_u\subset\stree$.
We then take the tensor product of the product of the above set-theoretic assignments $\NCat(\ttree,\stree)_k\rightarrow\prod_{u\in V(\utree)}\NCat(\thetatree_u,\sigmatree_u)_k$
with the map $\mu^*: \DGN_*(\Delta^1)^{\otimes k}\rightarrow\bigotimes_{u\in V(\utree)}\DGN_*(\Delta^1)^{\otimes k}$
induced by the coassociative coproduct of the cubical chain complex $\DGN_*(\Delta^1)^{\otimes k}$.
We just need to fix an ordering on the set of vertices of our trees since this coproduct is not associative.
We easily check that the facet operators and the Segal maps satisfy natural associativity relations and are compatible, in some natural sense,
with the composition operation of our enriched category structure.

In the symmetric context, we can also observe that the hom-objects $\Tree^{\square}(\ttree,\stree)$
inherit an action of the symmetric group such that
\begin{equation*}
s_*: \Tree^{\square}(\ttree,\stree)\rightarrow\Tree^{\square}(s\ttree,s\stree),
\end{equation*}
for every pair of reduced trees $\stree,\ttree\in\widetilde{\Tree}(r)$, $r>0$,
and for every permutation $s\in\Sigma_r$,
which are induced by the mappings $\ttree_i\mapsto s\ttree_i$ at the level of the sets $\NCat(\ttree,\stree)_k$.
These operators are compatible with the enriched category structure (so that the mapping $s_*: \ttree\mapsto s\ttree$
defines a functor on the enriched category $\Tree^{\square}$)
and with the facet operators. (But the action of permutation is not compatible with the Segal maps with values in the tensor product, since we need to order the vertices
of our trees when we form these maps.)
\end{constr}

We use that this enriched category in dg modules $\Tree^{\square}$ can be upgraded to a category enriched over the category of $\EOp$-coalgebras.

\begin{prop}\label{proposition:E-algebra-tree-category}
Each object $\Tree^{\square}(\ttree,\stree)$ inherits an $\EOp$-coalgebra structure from the cubical chain complexes $\DGN_*(\Delta^1)^{\otimes k}$.
The composition products define morphisms of $\EOp$-coalgebras $\circ: \Tree^{\square}(\ttree,\utree)\otimes\Tree^{\square}(\utree,\stree)\rightarrow\Tree^{\square}(\ttree,\stree)$
(we switch the conventional order of the factors of the composition to make the diagonal action of the Barratt--Eccles operad compatible with these operations).
The facet operators $i_{\sigmatree,\stree}: \Tree^{\square}(\ttree,\stree)\rightarrow\Tree^{\square}(\thetatree,\sigmatree)$
also define morphisms of $\EOp$-coalgebras,
as well as the operators that give the action of permutations $s_*: \Tree^{\square}(\ttree,\stree)\rightarrow\Tree^{\square}(s\ttree,s\stree)$
in the symmetric setting.

The constructions of the previous paragraph accordingly give a category $\Tree^{\square}$ enriched in $\EOp$-coalgebras
and equipped with facet operators (together with an action of permutations), which are defined within the category of $\EOp$-coalgebras
and are compatible with the composition structure of our objects.
\end{prop}

\begin{proof}
For each $k\geq 2$, we use that $\kk[\NCat(\ttree,\stree)_k]$ inherits a cocommutative coalgebra structure (given by the diagonal of the set $\NCat(\ttree,\stree)_k$)
in order to extend the $\EOp$-coalgebra structure of the cubical chain complex $\DGN_*(\Delta^1)^{\otimes k}$
to the tensor product $\kk[\NCat(\ttree,\stree)_k]\otimes\DGN_*(\Delta^1)^{\otimes k}$.
We readily check that this $\EOp$-coalgebra structure is compatible with the action of the category $\CubeCat$
and therefore passes to our coend. (Recall simply that the forgetful functor from a category of coalgebras to a base category creates coends.)

We easily check that the composition operations of the category $\Tree^{\square}$ are also compatible with the $\EOp$-coalgebra structure termwise,
as well the facet operators. Then we just pass to the coend
to get the conclusions of the proposition.
\end{proof}

We now consider the enriched category in cocommutative coalgebras such that
\begin{equation*}
\Tree(\ttree,\stree) = \begin{cases} \kk, & \text{if we have a morphism $\ttree\rightarrow\stree$}, \\
0, & \text{otherwise}, \end{cases}
\end{equation*}
for any pair of reduced trees $\stree,\ttree\in\widetilde{\Tree}(r)$, $r>0$ (with the same abuse of notation as in the introduction of this subsection).
We immediately see that this enriched category inherits the same structures (facet operators, action of permutations) within the category of cocommutative coalgebras
as the enriched category in $\EOp$-coalgebras $\Tree^{\square}$.
We also have the following observation:

\begin{prop}\label{lemma:tree-square-contractible}
\begin{enumerate}
\item
The hom-objects of the enriched category $\Tree^{\square}$ are endowed with weak-equivalences
\begin{equation*}
\epsilon: \Tree^{\square}(\ttree,\stree)\xrightarrow{\sim}\kk,
\end{equation*}
which are yielded by the augmentation maps such that $\epsilon(1_{\ttree\rightarrow\stree}) = 1$
and $\epsilon(\underline{\sigma}_{\ttree\rightarrow\ttree_k\rightarrow\dots\rightarrow\ttree_0\rightarrow\stree}) = 0$
in cubical dimension $k>0$ when $\ttree\rightarrow\ttree_k\rightarrow\dots\rightarrow\ttree_0\rightarrow\stree$
is non degenerate.
\item
These weak-equivalences also define a morphism of enriched categories
\begin{equation*}
\epsilon: \Tree^{\square}\xrightarrow{\sim}\Tree,
\end{equation*}
which is compatible with the facet operators and with the action of permutations whenever we consider this extra structure.
(We then regard the enriched category in cocommutative coalgebras $\Tree$ as a category enriched in $\EOp$-coalgebras
by restriction of structure through the augmentation map of the Barratt--Eccles operad.)
\end{enumerate}
\end{prop}

\begin{proof}
The morphism $\epsilon: \Tree^{\square}(\ttree,\stree)\xrightarrow{\sim}\kk$ has an obvious section $\eta: \kk\rightarrow\Tree^{\square}(\ttree,\stree)$
given by $\eta(1) = 1_{\ttree\rightarrow\stree}$.
We construct a chain homotopy $h: \Tree^{\square}(\ttree,\stree)\otimes\DGN_*(\Delta^1)\rightarrow\Tree^{\square}(\ttree,\stree)$
between $\eta\circ\epsilon = h(-\otimes\underline{1})$ and $\id = h(-\otimes\underline{0})$
to prove that $\epsilon$ is a weak-equivalence.

We proceed as follows. For $k\geq 0$, we consider the map
\begin{equation*}
h_k: \DGN_*(\Delta^1)^{\otimes k}\otimes\DGN_*(\Delta^1)\rightarrow\DGN_*(\Delta^1)^{\otimes k}
\end{equation*}
defined by the composite
\begin{multline*}
\DGN_*(\Delta^1)^{\otimes k}\otimes\DGN_*(\Delta^1)\xrightarrow{\id\otimes\mu^*}\DGN_*(\Delta^1)^{\otimes k}\otimes\DGN_*(\Delta^1)^{\otimes k} \\
\xrightarrow{\simeq}(\DGN_*(\Delta^1)\otimes\DGN_*(\Delta^1))^{\otimes k}\xrightarrow{(\nabla_*^{\max})^{\otimes k}}\DGN_*(\Delta^1)^{\otimes k},
\end{multline*}
where $\mu^*$ is the ($k$-fold) Alexander--Whitney diagonal and $\nabla_*^{\max} = \DGN_*(\max)\circ\EM$ is the composite of the Eilenberg--MacLane map
with the morphism induced by the map of simplicial sets $\max: \Delta^1\times\Delta^1\rightarrow\Delta^1$
such that $\max: (s,t)\mapsto\max(s,t)$. (Thus, we consider a mirror of the connection $\nabla_* = \nabla_*^{\min}$,
which we use in the definition of the codegeneracies of our cubical complex $\DGN_*(\Delta^1)^{\otimes k}$.)

We claim that these maps preserve the action of the cofaces $d^0_i$ and of the codegeneracies $s_j$ on the cubical complex $\DGN_*(\Delta^1)^{\otimes k}$.
The preservation of the cofaces $d^0_i$ follows from the formulas $\nabla_*^{\max}(\underline{1},\underline{01}) = 0$
and $\nabla_*^{\max}(\underline{1},\underline{\tau}) = \underline{1}$ for $\deg(\underline{\tau}) = 0$.
The preservation of the codegeneracies $s_j$ is immediate in the cases $j = 0$ and $j = k$.
We use the properties of the Alexander--Whitney diagonal and of the Eilenberg--MacLane map to reduce the verification of the preservation of the codegeneracies $s_j$
such that $j = 1,\dots,k-1$ to the case $k=2$.
We then deduce our claim from the distribution relation
\begin{equation*}
\nabla_*^{\max}(\nabla_*^{\min}(\underline{\sigma}_2\otimes\underline{\sigma}_1)\otimes\underline{\tau})
= \sum_{(\underline{\tau})}\nabla_*^{\min}(\nabla_*^{\max}(\underline{\sigma}_2\otimes\underline{\tau}')\otimes\nabla_*^{\max}(\underline{\sigma}_1\otimes\underline{\tau}'')),
\end{equation*}
valid for $\underline{\sigma}_2,\underline{\sigma}_1,\underline{\tau}\in\DGN_*(\Delta^1)$,
and where we write $\mu^*(\underline{\tau}) = \sum_{(\underline{\tau})}\underline{\tau}'\otimes\underline{\tau}''$
for the Alexander--Whitney diagonal of the element $\underline{\tau}\in\DGN_*(\Delta^1)$.
(This relation, which reflects the classical min-max distribution relation, can easily be checked by hand.)
We deduce from these verifications that these maps $h_k$, tensored with the identity of the factor $\kk[\NCat(\ttree,\stree)_k]$,
induce a well-defined map on our coend $h: \Tree^{\square}(\ttree,\stree)\otimes\DGN_*(\Delta^1)\rightarrow\Tree^{\square}(\ttree,\stree)$.

We also have $h_k(\underline{\sigma}\otimes\underline{1}) = \underline{1}^{\otimes k}$ and $h_k(\underline{\sigma}\otimes\underline{0}) = \underline{\sigma}$, for each $k\geq 0$,
and these identities give the relations $\eta\circ\epsilon = h(-\otimes\underline{1})$ and $\id = h(-\otimes\underline{0})$
at the coend level.
This verification completes the proof of the first assertion of the proposition, while an immediate inspection gives the verification of the second assertion.
\end{proof}

This proposition has the following corollary, which we use in our verification that the homotopy Kan construction returns Segal $E_\infty$-Hopf pre-cooperads
that satisfy the Segal condition.

\begin{cor}\label{cor:Segal-condition-tree-square}
The Segal map of Construction~\ref{construction:W-homotopy} defines weak-equivalence of dg modules
\begin{equation*}
i_{\lambda_{\utree}(\sigmatree_*)}: \Tree^{\square}(\ttree,\stree)\rightarrow\bigotimes_{u\in V(\utree)}\Tree^{\square}(\thetatree_u,\sigmatree_u),
\end{equation*}
for every tree decomposition $\stree = \lambda_{\utree}(\sigmatree_*)$, where we again set $\thetatree_u = f^{-1}(\sigmatree_u)$, for all factors $\sigmatree_u\subset\stree$.\qed
\end{cor}

We can reformulate the definition of the structure operators of homotopy Segal $E_\infty$-Hopf shuffle pre-cooperads in terms of the category $\Tree^{\square}$.
We get the following result, which follows from formal verifications.

\begin{prop}\label{prop:dg-category-trees}
Let $\AOp$ be a connected homotopy Segal $E_\infty$-Hopf shuffle pre-cooperad (either symmetric or shuffle).
The objects $\AOp(\ttree)$, $\ttree\in\Tree(r)$, $r>0$, inherit an action of the enriched category $\Tree^{\square}$,
given by operators
\begin{equation*}
\rho: \AOp(\stree)\otimes\Tree^{\square}(\ttree,\stree)\rightarrow\AOp(\ttree),
\end{equation*}
defined in the category of dg modules, and which preserve the action of the facet operators in some natural sense (as well as the action of permutations in the symmetric setting).
The homotopy coproduct operators $\rho_{\ttree\rightarrow\ttree_k\rightarrow\dots\rightarrow\ttree_1\rightarrow\stree}$
are identified with the adjoint morphisms
of the maps
\begin{multline*}
\AOp(\stree)\otimes[\ttree\rightarrow\ttree_k\rightarrow\dots\rightarrow\ttree_1\rightarrow\stree]\otimes\DGN_*(\Delta^1)^{\otimes k}\\
\rightarrow\AOp(\stree)\otimes\underbrace{\int^{\underline{k+2}\in\CubeCat}\kk[\NCat(\ttree,\stree)_k]\otimes\DGN_*(\Delta^1)^{\otimes k}}_{\Tree^{\square}(\ttree,\stree)}
\rightarrow\AOp(\ttree),
\end{multline*}
which we determine from this action.\qed
\end{prop}

In the case where $\AOp$ is a strict Segal $E_\infty$-Hopf shuffle pre-cooperad, we may see that the action of the enriched category $\Tree^{\square}$
defined in this proposition factors through an action of the category enriched in cocommutative coalgebras $\Tree$ (which satisfies the same properties).
Thus we just retrieve the obvious functor structure of the object $\AOp$
in this case.

\medskip
We now tackle the definition of our homotopy Kan construction. We have to dualize the structure operations attached to the category $\Tree^{\square}$.
For this purpose, we use the following observation.

\begin{lemm}\label{lemma:finite-generated-tree-square}
The dg hom-object $\Tree^{\square}(\ttree,\stree)$ forms a free module of finite rank over the ground ring $\kk$, for all trees $\ttree,\stree\in\widetilde{\Tree}(r)$, $r>0$.
\end{lemm}

\begin{proof}
Each element of our coend has a unique representative as a (linear combination) of tensors
of the form
\begin{equation*}
[\ttree\xrightarrow{\not=}\ttree_k\xrightarrow{\not=}\dots\xrightarrow{\not=}\ttree_1\xrightarrow{\not=}\stree]\otimes\underline{\sigma}
\in\kk[\NCat(\ttree,\stree)_k]\otimes\mathring{\DGN}_*(\Delta^1)^{\otimes k},
\end{equation*}
where $\mathring{\DGN}_*(\Delta^1) = \kk\underline{0}\oplus\kk\underline{01}$. Then we just use that the set of sequences of composable morphisms
of the form $\ttree\xrightarrow{\not=}\ttree_k\xrightarrow{\not=}\dots\xrightarrow{\not=}\ttree_1\xrightarrow{\not=}\stree$, $k\geq 0$,
is finite,
because each tree $\ttree_i$ is given by the contraction of a subset of inner edges $e\in\mathring{E}(\ttree)$,
which are in finite number.
\end{proof}

We can now define the two-sided cobar complex which underlies our homotopy Kan construction $\DGK^c(\AOp)$.
We address this construction in the next paragraph.

\begin{constr}\label{constr:two-sided-cobar}
We use the statement of the previous lemma to dualize the structure operations associated to the enriched category $\Tree^{\square}$.

For a pair of trees $\stree,\ttree\in\Tree(r)$, $r>0$, we let $\Tree^{\square}(\ttree,\stree)^{\sharp}$
denote the dual $\EOp$-algebra
of the $\EOp$-coalgebra $\Tree^{\square}(\ttree,\stree)$.
The composition operations of the enriched category $\Tree^{\square}$ induce a coproduct map
\begin{align*}
\gamma^*: \Tree^{\square}(\ttree,\stree)^{\sharp}
& \rightarrow\prod_{\ttree\rightarrow\utree\rightarrow\stree}\Tree^{\square}(\ttree,\utree)^{\sharp}\otimes\Tree^{\square}(\utree,\stree)^{\sharp},
\intertext{which also forms a morphism of $\EOp$-algebras. We also have an augmentation map}
\eta^*: \Tree^{\square}(\ttree,\stree)^{\sharp} & \rightarrow\kk,
\end{align*}
which we take as the identity in the case $\ttree = \stree$, as the zero map in the case $\ttree\not=\stree$.
The facet operators induce $\EOp$-algebra morphisms
\begin{equation*}
i_{\sigmatree,\stree}:  \Tree^{\square}(\thetatree,\sigmatree)^{\sharp}\rightarrow\Tree^{\square}(\ttree,\stree)^{\sharp},
\end{equation*}
which preserve the above coproduct and counit operations. Note that the cartesian product in the definition of the operation $\gamma^*$ reduces to a direct sum,
since every morphism $\ttree\rightarrow\stree$ admits finitely many factorizations $\ttree\rightarrow\utree\rightarrow\stree$.
In the symmetric setting, we also consider $\EOp$-algebra morphisms
\begin{equation*}
s^*:  \Tree^{\square}(s\ttree,s\stree)^{\sharp}\rightarrow\Tree^{\square}(\ttree,\stree)^{\sharp}
\end{equation*}
given by the action of permutations $s\in\Sigma_r$.

For a homotopy Segal $E_\infty$-Hopf cooperad, the operations of Proposition~\ref{prop:dg-category-trees}
dualize to $\EOp$-algebra morphisms
\begin{equation*}
\rho^*: \AOp(\stree)\rightarrow\prod_{\ttree\rightarrow\stree}\AOp(\ttree)\otimes\Tree^{\square}(\ttree,\stree)^{\sharp},
\end{equation*}
which we can also identify with an end of the homotopy coproducts attached to our object by the adjoint definition
of these operations in our Proposition~\ref{prop:dg-category-trees} (we again use a variant of the observations
of the previous lemma to obtain that the tensor product with $\AOp(\ttree)$
distributes over this end).
These morphisms are coassociative and counital with respect to the coproduct and counit operations of the objects $\Tree^{\square}(\ttree,\stree)^{\sharp}$,
commute with the action of the facet operators on our homotopy Segal $E_\infty$-Hopf cooperad
and on the objects $\Tree^{\square}(\ttree,\stree)^{\sharp}$ (and commute with the action of permutations in the symmetric setting).

We then let $\FOp_{\stree}(\ttree)\in\EAlg$ be a collection of $\EOp$-algebras, defined for any fixed tree $\stree\in\widetilde{\Tree}(r)$, for all $\ttree\in\widetilde{\Tree}(r)/\stree$,
and equipped with coproduct operations
\begin{equation*}
\gamma^*: \FOp_{\stree}(\ttree)\rightarrow\prod_{\ttree\rightarrow\utree\rightarrow\stree}\Tree^{\square}(\ttree,\utree)^{\sharp}\otimes\FOp_{\stree}(\utree),
\end{equation*}
defined in the category of $\EOp$-algebras, and which are again coassociative and counital with respect to the coproduct and counit operations
of the objects $\Tree^{\square}(\ttree,\stree)^{\sharp}$.
We assume that this collection defines a functor in the tree $\stree$ when $\stree$ varies
and that we have facet operators
\begin{equation*}
i_{\sigmatree,\stree}:  \FOp_{\sigmatree}(\thetatree)\rightarrow\FOp_{\stree}(\ttree),
\end{equation*}
associated to all subtrees $\sigmatree\subset\stree$, with $\thetatree = f^{-1}(\sigmatree)$, the pre-image of $\sigmatree$ under the morphism $f: \ttree\rightarrow\stree$,
which again satisfy natural functoriality relations and are compatible with the coproduct operations.
In the symmetric setting, we also assume that we have an action of the permutations $s^*: \FOp_{s\stree}(s\ttree)\rightarrow\FOp_{\stree}(\ttree)$,
which is compatible with the structure operations
attached to our collection.
In what follows, we consider the cases $\FOp_{\stree}(\ttree) = \Tree^{\square}(\ttree,\stree)^{\sharp}$ and $\FOp_{\stree}(\ttree) = \Tree(\ttree,\stree)^{\sharp}$,
where in the latter case $\Tree(\ttree,\stree)^{\sharp}$ denotes the commutative algebra of functions $u: \Tree(\ttree,\stree)\rightarrow\kk$
on the morphism sets of the tree category $\Tree(\ttree,\stree)$.

For each $n\in\NN$, we set
\begin{multline*}
\DGK^n(\AOp,\Tree^{\square},\FOp_{\stree})\\
= \prod_{\ttree_n\rightarrow\cdots\rightarrow\ttree_0\rightarrow\stree}
\AOp(\ttree_n)\otimes\Tree^{\square}(\ttree_n,\ttree_{n-1})^{\sharp}\otimes\dots\otimes\Tree^{\square}(\ttree_1,\ttree_0)^{\sharp}\otimes\FOp_{\stree}(\ttree_0),
\end{multline*}
and we equip this object with the coface operators $d^i: \DGK^{n-1}(\AOp,\Tree^{\square},\FOp_{\stree})\rightarrow\DGK^n(\AOp,\Tree^{\square},\FOp_{\stree})$
defined termwise by the maps such that
\begin{align*}
d^i & = \begin{cases}
\id\otimes\id^{\otimes n-1}\otimes\gamma^*, & \text{for $i = 0$}, \\
\id\otimes\id^{\otimes n-i-1}\otimes\gamma^*\otimes\id^{\otimes i-1}\otimes\id, & \text{for $i = 1,\dots,n-1$}, \\
\rho^*\otimes\id^{\otimes n-1}\otimes\id, & \text{for $i = n$},
\end{cases}
\intertext{and with the codegeneracies $s^j: \DGK^{n+1}(\AOp,\Tree^{\square},\FOp_{\stree})\rightarrow\DGK^n(\AOp,\Tree^{\square},\FOp_{\stree})$
defined termwise by the maps}
s^j & = \id\otimes\id^{\otimes n-j}\otimes\eta^*\otimes\id^{\otimes j}\otimes\id\quad\text{for $j = 0,\dots,n$}.
\end{align*}
We easily check that this definition returns a cosimplicial object in the category of $\EOp$-algebras.
We also have facet operators
\begin{equation*}
i_{\sigmatree,\stree}: \DGK^{\bullet}(\AOp,\Tree^{\square},\FOp_{\sigmatree})\rightarrow\DGK^{\bullet}(\AOp,\Tree^{\square},\FOp_{\stree}),
\end{equation*}
compatible with the cosimplicial structure, and defined by the termwise tensor products of facet operators
\begin{multline*}
i_{\thetatree_n,\ttree_n}\otimes i_{\thetatree_{n-1},\ttree_{n-1}}\otimes\dots\otimes i_{\thetatree_0,\ttree_0}\otimes i_{\sigmatree,\stree}:\\
\AOp(\thetatree_n)\otimes\Tree^{\square}(\thetatree_n,\thetatree_{n-1})^{\sharp}\otimes\dots\otimes\Tree^{\square}(\thetatree_1,\thetatree_0)^{\sharp}\otimes\FOp_{\stree}(\thetatree_0)\\
\rightarrow\AOp(\ttree_n)\otimes\Tree^{\square}(\ttree_n,\ttree_{n-1})^{\sharp}\otimes\dots\otimes\Tree^{\square}(\ttree_1,\ttree_0)^{\sharp}\otimes\FOp_{\stree}(\ttree_0),
\end{multline*}
where $\thetatree_i = (f_0\cdots f_i)^{-1}(\sigmatree)$, $i = 0,\dots,n$, denotes the pre-image of the subtree $\sigmatree\subset\stree$
under the composite of the tree morphisms $\ttree_i\xrightarrow{f_i}\cdots\xrightarrow{f_1}\ttree_0\xrightarrow{f_0}\stree$.

In the symmetric setting, we still consider an action of permutations
\begin{equation*}
s^*: \DGK^{\bullet}(\AOp,\Tree^{\square},\FOp_{s\stree})\rightarrow\DGK^{\bullet}(\AOp,\Tree^{\square},\FOp_{\stree}),
\end{equation*}
defined again by an obvious termwise construction, and compatible with the facet operators.
\end{constr}

We record the outcome of the previous construction in the next proposition.

\begin{prop}\label{proposition:cosimplicial-B}
The construction of the previous paragraph returns a collection of cosimplicial $\EOp$-algebras
\begin{equation*}
\DGK^{\bullet}(\AOp,\Tree^{\square},\FOp_{\stree})\in\EAlg,\quad\stree\in\widetilde{\Tree}(r),\quad r>0,
\end{equation*}
equipped with compatible facet operators $i_{\sigmatree,\stree}: \DGK^{\bullet}(\AOp,\Tree^{\square},\FOp_{\stree})\rightarrow\DGK^{\bullet}(\AOp,\Tree^{\square},\FOp_{\stree})$,
which satisfy the usual functoriality relations.
If the homotopy Segal $E_\infty$-Hopf cooperad $\AOp$ is symmetric and $\FOp_{\stree}(-)$ is endowed with a symmetric structure as well,
then we also have an action of permutations on our objects $s^*: \DGK^{\bullet}(\AOp,\Tree^{\square},\FOp_{s\stree})\rightarrow\DGK^{\bullet}(\AOp,\Tree^{\square},\FOp_{\stree})$
compatible with the cosimplicial structure and with the facet operators.
\qed
\end{prop}

In the case of a decomposition $\stree = \lambda_{\utree}(\sigmatree_u,u\in V(\utree))$, we can assemble the facet operators
$i_{\sigma_u,\stree}: \FOp_{\sigmatree_u}(\thetatree_u)\rightarrow\FOp_{\stree}(\ttree)$
associated to a bicollection of $\EOp$-algebras $\FOp_{\stree}(\ttree)\in\EAlg$
as in Construction~\ref{constr:two-sided-cobar}
into a Segal map
\begin{equation*}
i_{\lambda_{\utree}(\sigmatree_*)}: \bigvee_{u\in V(\utree)}\FOp_{\sigmatree_u}(\thetatree_u)\rightarrow\FOp_{\stree}(\ttree),
\end{equation*}
and we can define a Segal map similarly on our cosimplicial object $\DGK^{\bullet}(\AOp,\Tree^{\square},\FOp_{\stree})$.
We say that our bicollection $\FOp_{\stree}(\ttree)\in\EAlg$ satisfies the Segal condition when the above Segal map is a weak-equivalence.
We have the following statement.

\begin{prop}\label{prop:cosimplicial-Segal}
If $\AOp$ is a homotopy Segal $E_\infty$-Hopf cooperad and hence satisfies the Segal condition, and if the bicollection $\FOp_{\stree}(\ttree)\in\EAlg$ satisfies the Segal condition as well,
then the Segal maps that we associate to our cosimplicial object $\DGK^{\bullet}(\AOp,\Tree^{\square},\FOp_{\stree})$
define levelwise weak-equivalences of cosimplicial $\EOp$-algebras
\begin{equation*}
i_{\lambda_{\utree}(\sigmatree_*)}: \bigvee_{u\in V(\utree)}\DGK^{\bullet}(\AOp,\Tree^{\square},\FOp_{\sigmatree_u})\xrightarrow{\sim}\DGK^{\bullet}(\AOp,\Tree^{\square},\FOp_{\stree}),
\end{equation*}
for all tree decomposition $\stree = \lambda_{\utree}(\sigmatree_u,u\in V(\utree))$,
so that $\DGK^{\bullet}(\AOp,\Tree^{\square},\FOp_{\stree})$ also satisfies a form of our Segal condition levelwise.
\end{prop}

\begin{proof}
The Segal maps of the proposition are given, on each term of the cosimplicial object $\DGK^{\bullet}(\AOp,\Tree^{\square},\FOp_{\stree})$, by expressions of the form
\begin{multline*}
\bigvee_{u\in V(\utree)}
\biggl(\AOp(\thetatree_n^u)\otimes\cdots\otimes\Tree^{\square}(\thetatree_i^u,\thetatree_{i-1}^u)^{\sharp}\otimes\cdots\otimes\FOp_{\stree}(\thetatree_0^u)\biggr)\\
\rightarrow\AOp(\ttree_n)\otimes\cdots\otimes\Tree^{\square}(\ttree_i,\ttree_{i-1})^{\sharp}\otimes\cdots\otimes\FOp_{\stree}(\ttree_0),
\end{multline*}
where $\thetatree_i^u$ denote the pre-images of the subtree $\sigmatree\subset\stree$ under the tree morphisms $\ttree_i\rightarrow\dots\rightarrow\ttree_0\rightarrow\stree$
and we take a tensor product of facet operators
on each factor. We compose this map with the Eilenberg--MacLane map to pass from the coproduct $\bigvee_{u\in V(\utree)}$
to a tensor product $\bigotimes_{u\in V(\utree)}$ (as in Proposition~\ref{proposition:forgetful-strict}).
We have an obvious commutative diagram which enables us to identify the obtained Segal map
with a tensor product of the form:
\begin{multline*}
\biggl(\bigotimes_{u\in V(\utree)}\AOp(\thetatree_n^u)\biggr)
\otimes\cdots\otimes\biggl(\bigotimes_{u\in V(\utree)}\Tree^{\square}(\thetatree_i^u,\thetatree_{i-1}^u)^{\sharp}\biggr)\otimes\cdots
\otimes\biggl(\bigotimes_{u\in V(\utree)}\FOp_{\stree}(\thetatree_0^u)\biggr)\\
\rightarrow\AOp(\ttree_n)\otimes\cdots\otimes\Tree^{\square}(\ttree_i,\ttree_{i-1})^{\sharp}\otimes\cdots\otimes\FOp_{\stree}(\ttree_0),
\end{multline*}
where we take a factorwise tensor product of Segal maps associated to the objects $\AOp(-)$, $\Tree^{\square}(-,-)^{\sharp}$ and $\FOp_{\stree}(-)$.
In the case of the objects $\Tree^{\square}(-,-)^{\sharp}$, we retrieve the dual of the Segal maps
considered in Corollary~\ref{cor:Segal-condition-tree-square}.
These Segal maps are weak-equivalences therefore, like the Segal maps associated to the objects $\AOp(-)$ and $\FOp_{\stree}(-)$ by assumption.
The conclusion follows.
\end{proof}

We now focus on the cases $\FOp_{\stree}(-) = \Tree^{\square}(-,\stree)^{\sharp},\Tree(-,\stree)^{\sharp}$. The coproduct operation on $\Tree^{\square}(-,-){\sharp}$,
such as defined in Construction~\ref{constr:two-sided-cobar},
gives a natural transformation $\Tree^{\square}(-,\stree)^{\sharp}\rightarrow\Tree^{\square}(-,\ttree)^{\sharp}\otimes\Tree^{\square}(\ttree,\stree)^{\sharp}$,
which passes to our cosimplicial object, by functoriality of our construction, and yields a morphism of cosimplicial $\EOp$-algebras
\begin{equation*}
\rho^*: \DGK^{\bullet}(\AOp,\Tree^{\square},\Tree^{\square}(-,\stree)^{\sharp})
\rightarrow\DGK^{\bullet}(\AOp,\Tree^{\square},\Tree^{\square}(-,\ttree)^{\sharp})\otimes\Tree^{\square}(\ttree,\stree)^{\sharp},
\end{equation*}
for every pair of objects $\stree,\ttree\in\widetilde{\Tree}(r)$, $r>0$.
In the case $\FOp_{\stree}(-) = \Tree(-,\stree)^{\sharp}$, we similarly get morphisms of cosimplicial $\EOp$-algebras
of the form
\begin{equation*}
\rho^*: \DGK^{\bullet}(\AOp,\Tree^{\square},\Tree(-,\stree)^{\sharp})\rightarrow\DGK^{\bullet}(\AOp,\Tree^{\square},\Tree(-,\ttree)^{\sharp})\otimes\Tree(\ttree,\stree)^{\sharp}.
\end{equation*}
Note that the map $\Tree^{\square}(-,-)\rightarrow\Tree(-,-)$ of Proposition~\ref{lemma:tree-square-contractible}
induces a natural transformation in the converse direction between these cosimplicial objects $\DGK^{\bullet}(\AOp,\Tree^{\square},\FOp_{\stree})$
that we associate to $\FOp_{\stree}(-) = \Tree^{\square}(-,\stree)^{\sharp}$ and to $\FOp_{\stree}(-) = \Tree(-,\stree)^{\sharp}$.
We then have the following result.

\begin{prop}\label{prop:cosimplicial-cooperad-B}
\begin{enumerate}
\item
The above coproduct operations provide the collection of cosimplicial $\EOp$-algebras
\begin{equation*}
\DGK^{\bullet}(\AOp,\Tree^{\square},\Tree^{\square}(-,\stree)^{\sharp})\in\cosimp\EAlg,\quad\stree\in\widetilde{\Tree}(r),\quad r>0,
\end{equation*}
with the coproduct operators of a homotopy Segal $E_\infty$-Hopf pre-cooperad structure.
These coproduct operators are compatible with the facet operators and hence $\DGK^{\bullet}(\AOp,\Tree^{\square},\Tree^{\square}(-,\stree)^{\sharp})$
forms a cosimplicial object in the category of homotopy Segal $E_\infty$-Hopf cooperads (shuffle or symmetric when $\AOp$ is so).
\item
In the case of the cosimplicial object $\DGK^{\bullet}(\AOp,\Tree^{\square},\Tree(-,\stree)^{\sharp})$,
we similarly obtain a strict Segal $E_\infty$-Hopf pre-cooperad structure
compatible with the cosimplicial structure
on $\DGK^{\bullet}(\AOp,\Tree^{\square},\Tree(-,\stree)^{\sharp})$.
Furthermore, the natural transformations $\Tree(-,\stree)^{\sharp}\rightarrow\Tree^{\square}(-,\stree)^{\sharp}$
induce levelwise weak-equivalences of cosimplicial $\EOp$-algebras
\begin{equation*}
\DGK^{\bullet}(\AOp,\Tree^{\square},\Tree(-,\stree)^{\sharp})\xrightarrow{\sim}\DGK^{\bullet}(\AOp,\Tree^{\square},\Tree^{\square}(-,\stree)^{\sharp}),
\end{equation*}
which preserve the homotopy Segal $E_\infty$-Hopf pre-cooperad structures, and hence, define a levelwise weak-equivalence
of cosimplicial objects in the category of homotopy Segal $E_\infty$-Hopf pre-cooperads.
\item
Both objects $\DGK^{\bullet}(\AOp,\Tree^{\square},\Tree^{\square}(-,\stree)^{\sharp})$ and $\DGK^{\bullet}(\AOp,\Tree^{\square},\Tree(-,\stree)^{\sharp})$
also satisfy the Segal condition levelwise, and hence define (homotopy) Segal $E_\infty$-Hopf cooperads
when $\AOp$ does so.
\end{enumerate}
\end{prop}

\begin{proof}
The first assertion is an immediate consequence of the functoriality properties of our construction.
In the second assertion, we use that the natural transformation $\Tree^{\square}(-,\stree)^{\sharp}\rightarrow\Tree(-,\stree)^{\sharp}$
is dual to the augmentation map of Proposition~\ref{lemma:tree-square-contractible},
which is a weak-equivalence by the result of this proposition.
The third assertion follows from the result of Proposition~\ref{prop:cosimplicial-Segal}
since Corollary~\ref{cor:Segal-condition-tree-square} implies that the bicollection $\FOp_{\stree}(\ttree) = \Tree^{\square}(\ttree,\stree)^{\sharp}$
satisfies the Segal condition and this is also obviously the case of the bicollection $\FOp_{\stree}(\ttree) = \Tree(\ttree,\stree)^{\sharp}$.
\end{proof}

We use a totalization functor to transform the cosimplicial (homotopy) Segal $E_\infty$-Hopf cooperads
of the previous proposition into ordinary (homotopy) Segal $E_\infty$-Hopf cooperads
in dg modules.

\begin{constr}\label{definition:corealization}
Let $R^{\bullet}$ be a cosimplicial object of the category of $\EOp$-algebras. We set
\begin{equation*}
\DGN^*(R^{\bullet}) = \int_n R^n\otimes\DGN^*(\Delta^n),
\end{equation*}
and we equip this object with the $\EOp$-algebra structure induced by the diagonal $\EOp$-algebra structure on $R^n\otimes\DGN^*(\Delta^n)$ termwise.
If we forget about $\EOp$-algebra structures, then we can identify this object with the conormalized complex of cosimplicial dg modules
(see for instance~\cite[\S II.5.0.12 and \S II.9.4.6]{FresseBook}),
and as such, this functor carries the levelwise weak-equivalences of cosimplicial objects
to weak-equivalences in the category of dg modules.

For a cosimplicial connected (homotopy) $E_\infty$-Hopf pre-cooperad $\KOp^{\bullet}$, the collection $\DGN^*(\KOp^{\bullet}(\stree))$,
which we obtain by applying this conormalized complex functor termwise,
also inherits the structure of a (homotopy) $E_\infty$-Hopf pre-cooperad (shuffle or symmetric when $\KOp^{\bullet}$ is so)
by functoriality of our conormalized complex construction.
\end{constr}

We use the following lemma.

\begin{lemm} \label{lemma:cosimplicial-quasi-iso}
Let $R^{\bullet}$ and $S^{\bullet}$ be cosimplicial $\EOp$-algebras. The $\EOp$-algebra morphism $\DGN^*(R^{\bullet})\vee\DGN^*(S^{\bullet})\rightarrow\DGN^*(R^{\bullet}\vee S^{\bullet})$
induced by the canonical morphisms $R^{\bullet}\rightarrow R^{\bullet}\vee S^{\bullet}$ and $S^{\bullet}\rightarrow R^{\bullet}\vee S^{\bullet}$
is a weak-equivalence.
\end{lemm}

\begin{proof}
We have a commutative diagram
\begin{equation*}
\xymatrixcolsep{5pc}\xymatrix{ \DGN^*(R^{\bullet})\vee\DGN^*(S^{\bullet})\ar[r] &
\DGN^*(R^{\bullet}\vee S^{\bullet}) \\
\DGN^*(R^{\bullet})\otimes\DGN^*(S^{\bullet})\ar[r]^{\AW}\ar[u]^{\EM} &
\DGN^*(R^{\bullet}\otimes S^{\bullet})\ar[u]_{\DGN^*(\EM)}
},
\end{equation*}
where the vertical maps are given by the natural transformations between the tensor product and the coproduct in the category of $\EOp$-algebras,
such as defined in Construction~\ref{constr:Barratt-Eccles-diagonal},
and the bottom horizontal map $\AW$ is the generalization of the Alexander--Whitney diagonal
for the conormalized cochain complex of cosimplicial dg modules.
In the case of cosimplicial dg algebras, this map $\AW$ is used to represent a product operation.
The vertical maps $\EM$ also identifies tensor products with associative products in the coproduct of $\EOp$-algebras
by the definition of Construction~\ref{constr:Barratt-Eccles-diagonal}. The commutativity of the diagram
readily follows from this interpretation of our maps.

The vertical maps are weak-equivalences by Proposition~\ref{claim:Barratt-Eccles-algebra-coproducts}.
The bottom horizontal map is also a weak-equivalence (by the general theory of the Eilenberg--Zilber equivalence).
Therefore the map of the proposition, which represents the upper horizontal map
of our diagram, is also a weak-equivalence.
\end{proof}

This lemma has the following immediate consequence.

\begin{prop}\label{prop:totalization-Segal-cooperads}
If $\ROp^{\bullet}$ is a cosimplicial (homotopy) Segal $E_\infty$-Hopf pre-cooperad that satisfies the Segal condition levelwise,
then $\DGN^*(\ROp^{\bullet})$ satisfies the Segal condition as well, and hence forms a (homotopy) Segal $E_\infty$-Hopf cooperad
in the category of dg modules.\qed
\end{prop}

Then we have the following statement.

\begin{prop}
Let $\AOp$ be connected homotopy Segal $E_\infty$-Hopf cooperad (shuffle or symmetric).
\begin{enumerate}
\item
The objects $\DGN^*(\DGK^{\bullet}(\AOp,\Tree^{\square},\Tree^{\square}))$ and $\DGN^*(\DGK^{\bullet}(\AOp,\Tree^{\square},\Tree))$,
defined by taking the totalization of the cosimplicial (homotopy) Segal $E_\infty$-Hopf cooperads
of Proposition~\ref{prop:cosimplicial-cooperad-B},
respectively form a homotopy Segal $E_\infty$-Hopf cooperad and a strict homotopy Segal $E_\infty$-Hopf cooperad (which are shuffle or symmetric when $\AOp$ is so).
The levelwise weak-equivalence of Proposition~\ref{prop:cosimplicial-cooperad-B}
induces a weak-equivalence of homotopy Segal $E_\infty$-Hopf cooperads
\begin{equation*}
\DGN^*(\DGK^{\bullet}(\AOp,\Tree^{\square},\Tree(-,\stree)^{\sharp}))\xrightarrow{\sim}\DGN^*(\DGK^{\bullet}(\AOp,\Tree^{\square},\Tree^{\square}(-,\stree)^{\sharp})),
\end{equation*}
when we pass to this totalization.
\item
Furthermore, we have a weak-equivalence of homotopy Segal $E_\infty$-Hopf cooperads (shuffle or symmetric)
\begin{equation*}
\AOp(\stree)\xrightarrow{\sim}\DGN^*(\DGK^{\bullet}(\AOp,\Tree^{\square},\Tree^{\square}(-,\stree)^{\sharp})).
\end{equation*}
This weak-equivalence is natural in $\AOp$.
\end{enumerate}
\end{prop}

\begin{proof}
The first assertion of the proposition immediately follows from the statements of Proposition~\ref{prop:cosimplicial-cooperad-B}
and from the result of Proposition~\ref{prop:totalization-Segal-cooperads}.

Thus, we focus on the proof of the second assertion.
We define our natural transformation first.
We use that the coproduct map $\rho^*: \AOp(\stree)\rightarrow\prod_{\ttree\rightarrow\stree}\AOp(\ttree)\otimes\Tree^{\square}(\ttree,\stree)^{\sharp}$
which we associate to our object in Construction~\ref{constr:two-sided-cobar}
defines a coaugmentation over the cosimplicial object $\DGK^{\bullet}(\AOp,\Tree^{\square},\Tree^{\square}(-,\stree)^{\sharp})$,
or equivalently, a morphism
\begin{equation*}
\eta: \AOp\rightarrow\DGK^{\bullet}(\AOp,\Tree^{\square},\Tree^{\square}(-,\stree)^{\sharp}),
\end{equation*}
where we regard $\AOp$ as a constant cosimplicial object (see~\cite[\S II.5.4]{FresseBook} for an account of these concepts).
We immediately see that this morphism is compatible with the coproduct operators, with the facets (and with the action of permutations whenever defined),
and hence, defines a morphism of cosimplicial homotopy Segal $E_\infty$-Hopf cooperads (shuffle or symmetric)
which yields a morphism of homotopy Segal $E_\infty$-Hopf cooperads in dg modules of the form of the proposition
when we pass to conormalized cochain complexes (we just use that we have $\DGN^*(\AOp) = \AOp$
in the case of the constant cosimplicial object $\AOp$).

The weak-equivalence claim follows from the observation that, in the case $F_{\stree}(-) = \Tree^{\square}(-,\stree)^{\sharp}$,
the cosimplicial object $\DGK^{\bullet}(\AOp,\Tree^{\square},\Tree^{\square}(-,\stree)^{\sharp})$
is endowed with an extra codegeneracy $s^{-1}$,
which is defined by extending the definition of Construction~\ref{constr:two-sided-cobar}
to the case $j=-1$:
\begin{equation*}
s^{-1} = \id\otimes\id^{\otimes n+1}\otimes\eta^*.
\end{equation*}
(We again refer to~\cite[\S II.5.4]{FresseBook} for a proof that the existence of this extra codegeneracy forces the contractibility of the conormalized cochain complex
in the cosimplicial direction, and hence, forces the acyclicity of our map.)
This observation finishes the proof of the proposition.
\end{proof}

\begin{proof}[Proof of Theorem~\ref{theorem:strictification}]
We can now conclude the proof of Theorem~\ref{theorem:strictification}. The results of the previous proposition
returns a zigzag of weak-equivalences of homotopy Segal $E_\infty$-Hopf cooperads (shuffle or symmetric)
\begin{equation*}
\AOp(\stree)\xrightarrow{\sim}\DGN^*(\DGK^{\bullet}(\AOp,\Tree^{\square},\Tree^{\square}(-,\stree)))
\xleftarrow{\sim}\DGN^*(\DGK^{\bullet}(\AOp,\Tree^{\square},\Tree(-,\stree)^{\sharp})),
\end{equation*}
from which the result follows since the Segal $E_\infty$-Hopf cooperad
\begin{equation*}
\DGK^c(\AOp)(\stree) = \DGN^*(\DGK^{\bullet}(\AOp,\Tree^{\square},\Tree(-,\stree)^{\sharp}))
\end{equation*}
is strict by construction.
\end{proof}


\begin{appendix}

\renewcommand{\thesubsubsection}{\thesection.\arabic{subsubsection}}

\section{The Barratt--Eccles operad and $E_\infty$-algebras}\label{sec:Barratt-Eccles-operad}

The purpose of this appendix is to prove the results on the category of algebras over the Barratt--Eccles operad that we use throughout this article.
In preliminary paragraphs, we briefly review the definition of the chain Barratt--Eccles operad
and the definition of the associated category of $E_\infty$-algebras. We mostly follow the conventions of~\cite{BergerFresse} for the definition of this operad
and we refer to this article for more detailed explanations
on this subject. We also briefly explain our conventions and basic definitions on permutations.
We devote the next paragraph to this subject.

In~\S\ref{section:background}, we recall the definition of a cooperad without counit and we also forget about counits in the definition
of the notions of Segal cooperad that we consider all along this paper. But the Barratt--Eccles operad is more naturally defined as a unital operad.
Therefore we go back to the usual definition of an operad with unit in this appendix. Similarly, as the Barratt--Eccles operad naturally forms a symmetric operad,
we consider composition products in the standard form $\circ_i: \EOp(k)\otimes\EOp(l)\rightarrow\EOp(k+l-1)$
in this appendix, and not the general operations $\circ_{i_p}: \EOp(\{i_1<\dots<i_k\})\otimes\EOp(\{j_1<\dots<j_l\})\rightarrow\EOp(\{1<\dots<r\})$,
since we can deduce the latter from the former by the action of a shuffle permutation
on the Barratt--Eccles operad (see~\S\ref{subsection:shuffle-cooperads}).

\begin{recoll}[Conventions on permutations and the associative operad]\label{recoll:permutations}
We denote the symmetric group on $r$ letters by $\Sigma_r$.
We represent a permutation $s\in\Sigma_r$ by giving the sequence of its values
\begin{equation*}
s = (s(1),\dots,s(r)).
\end{equation*}

We use that the collection of the symmetric groups $\Sigma_r$, $r\in\NN$, form an operad in sets.
The symmetric structure of this operad is given by the left translation action.
The operadic composition $u\circ_i v\in\Sigma_{k+l-1}$ of permutations $u\in\Sigma_k$ and $v\in\Sigma_l$
is obtained by inserting the sequence of values of the permutation $v = (v(1),\dots,v(l))$
at the position of the value $i\in\{1,\dots,k\}$
in the permutation $u = (u(1),\dots,u(k))$,
by performing the value shift $v(y)\mapsto v(y) + i-1$
on the terms of the permutation $v$
and the shift $u(x)\mapsto u(x) + l-1$ on the terms of the permutation $u$ such that $u(x)>i$.
Thus, we have
\begin{equation*}
(u(1),\dots,u(k))\circ_i(v(1),\dots,v(l)) = (u(1)',\dots,\underbrace{v(1)',\dots,v(l)'}_{u(t)},\dots,u(k)'),
\end{equation*}
where $t$ is the position of the value $i$ in the sequence $(u(1),\dots,u(k))$, while $v(y)'$ and $u(x)'$ denote the result of our shift operations
so that we have $v(y)' = v(y) + i-1$, for all terms $v(y)$, we have $u(x)' = u(x)$ when $u(x)<i$
and $u(x)' = u(x) + l-1$ when $u(x)>i$.

This operad in sets governs the category of associative monoids. In our constructions, we also use a counterpart of this operad in our base category of modules.
This associative operad $\AsOp$ is defined by taking the modules spanned by the sets of permutations $\AsOp(r) = \kk[\Sigma_r]$, for $r\in\NN$,
with the induced structure operations.
In what follows, we generally identify a permutation $s\in\Sigma_r$ with the associated basis element in $\AsOp(r)$.
We also use the notation $\mu\in\AsOp(2)$ for the element of the associative operad given by the identity permutation on $2$ letters $\mu = \id_2\in\Sigma_2$,
which governs the product operation when we pass to associative algebras.
We trivially have $\AsOp(0) = \kk$ and we can identify the element given by the trivial permutation $* = \id_0\in\Sigma_0$ with an arity zero operation
which represents a unit for this product structure.
\end{recoll}

\begin{recoll}[The Barratt--Eccles operad and $E_\infty$-algebra structures]\label{recoll:Barratt-Eccles-operad}
The chain Barratt--Eccles operad $\EOp$ is defined by the normalized chain complexes of the homogeneous bar construction
of the symmetric groups.
Thus, we have:
\begin{equation*}
\EOp(r) = \DGN_*(W(\Sigma_r)),
\end{equation*}
for each arity $r\in\NN$, where $W(\Sigma_r)$ denotes the simplicial such that
\begin{equation*}
W(\Sigma_r)_n = \underbrace{\Sigma_r\times\dots\times\Sigma_r}_{n+1},
\end{equation*}
for each dimension $n$, together with the face and degeneracy operators such that
\begin{align*}
d_i(w_0,\dots,w_n) & = (w_0,\dots,\widehat{w_i},\dots,w_n), \\
s_j(w_0,\dots,w_n) & = (w_0,\dots,w_j,w_j,\dots,w_n),
\end{align*}
for any $(w_0,\dots,w_n)\in W(\Sigma_r)_n$.

For simplicity, we do not make any distinction between a simplex $(w_0,\dots,w_n)\in W(\Sigma_r)$
and the class of this simplex in the normalized chain complex $\EOp(r) = \DGN_*(W(\Sigma_r))$
in our notation. We just get $(w_0,\dots,w_j,w_j,\dots,w_n)\equiv 0$ for the degenerate simplices when we pass to the normalized chain complex.
The differential of simplices in $\EOp(r)$
is given by the usual formula:
\begin{equation*}
\delta(w_0,\dots,w_n) = \sum_{i=0}^n (-1)^i(w_0,\dots,\widehat{w_i},\dots,w_n).
\end{equation*}

The action of the symmetric group $\Sigma_r$ on $\EOp(r)$ is induced by the left translation action of permutations
on these simplices. We explicitly have:
\begin{equation*}
s\cdot(w_0,\dots,w_n) = (s w_0,\dots,s w_n),
\end{equation*}
for each permutation $s\in\Sigma_r$.
The operadic composition operations $\circ_i: \EOp(k)\otimes\EOp(l)\rightarrow\EOp(k+l-1)$ are given by the composite of a termwise application
of operadic composition operations on permutations with the Eilenberg--MacLane map
when we pass to normalized chains. For $(u_0,\dots,u_m)\in\EOp(k)$ and $(v_0,\dots,v_n)\in\EOp(l)$,
we explicitly have:
\begin{equation*}
(u_0,\dots,u_m)\circ_i(v_0,\dots,v_n) = \sum_{(i_*,j_*)}\pm(u_{i_0}\circ_i v_{j_0},\dots,u_{i_{m+n}}\circ_i v_{j_{m+n}}),
\end{equation*}
where the sum runs over the set of paths $\{(i_t,j_t),t=0,\dots,m+n\}$ which start at $(i_0,j_0) = (0,0)$ and end at $(i_{m+n},j_{m+n}) = (m,n)$
in an $m\times n$ cartesian diagram, the expression $\pm$ denotes a sign which we associate to any such path,
and along our paths, we take the operadic composites $u_{i_t}\circ_i v_{j_t}\in\Sigma_{k+l-1}$
of the permutations $u_{i_t}\in\Sigma_k$ and $v_{j_t}\in\Sigma_l$.
(The sign $\pm$ is determined by the shuffle of horizontal and vertical moves which we use when we form our path.)
For convenience, we may also represent such a composite by a picture
of the following form:
\begin{equation*}
(u_0,\dots,u_m)\circ_i(v_0,\dots,v_n) = \sum\pm\left(\vcenter{\xymatrix@R=2mm@C=3mm{ u_0\circ_i v_0\ar@{-}[d]\ar@{-}[r] &
u_1\circ_i v_0\ar@{-}[d]\ar@{-}[r] &
*{\cdots}\ar@{-}[r] &
u_m\circ_i v_0\ar@{-}[d] \\
u_0\circ_i v_1\ar@{-}[d]\ar@{-}[r] &
u_1\circ_i v_1\ar@{-}[d]\ar@{-}[r] &
*{\cdots}\ar@{-}[r] &
u_m\circ_i v_1\ar@{-}[d] \\
*{\vdots}\ar@{-}[d] & *{\vdots}\ar@{-}[d] & & *{\vdots}\ar@{-}[d] \\
u_0\circ_i v_n\ar@{-}[r] &
u_1\circ_i v_n\ar@{-}[r] &
*{\cdots}\ar@{-}[r] &
u_m\circ_i v_n }}\right),
\end{equation*}
where we take the sum of the simplices that we may form by running over all paths contained in the diagram materialized in our figure. (To be fully explicit, we take the paths
which go from the upper-left corner to the lower-right corner of the diagram
by a shuffle of horizontal moves $\xymatrix@R=2mm@C=3mm{u_x\circ_i v_y\ar@{-}[r] & u_{x+1}\circ_i v_y }$
and of vertical moves $\xymatrix@R=2mm@C=3mm{u_x\circ_i v_y\ar@{-}[r] & u_x\circ_i v_{y+1}}$.)

Recall that the operad of permutations in sets is identified with the set-theoretic associative operad (the operad which governs the category of associative monoids).
From the relation $W(\Sigma_r)_0 = \Sigma_r$ for any $r\in\NN$, we get an operad embedding $\AsOp\subset\EOp$
which identifies the module-theoretic version of the associative operad $\AsOp$
with the degree zero component of the Barratt-Eccles operad $\EOp$.
In what follows, we still use the notation $\mu\in\EOp(2)$ for the degree $0$ operation, represented by the identity permutation $\mu = \id_2\in\Sigma_2$,
which governs the product operation of associative algebra structures
in the Barratt--Eccles operad.
Note that we still have $\EOp(0) = \AsOp(0) = \kk$ (we take the convention to consider operads with a term in arity zero throughout this paragraph)
and the generating element of this arity zero term $*\in\EOp(0)$ also represents a unit operation
when we pass to the category of algebras over the Barratt--Eccles operad.

The Barratt--Eccles operad $\EOp$ is weakly-equivalent to the operad of commutative algebras $\ComOp$, and forms, as such, an instance of an $E_\infty$-operad.
Recall that we have $\ComOp(r) = \kk$, for any $r\in\NN$. The weak-equivalence $\EOp\xrightarrow{\sim}\ComOp$
is given by the standard augmentation $\DGN_*(W(\Sigma_r))\rightarrow\DGN_*(\pt) = \kk$
on the normalized chain complexes $\EOp(r) = \DGN_*(W(\Sigma_r))$, $r\in\NN$,
and sits in a factorization $\AsOp\hookrightarrow\EOp\xrightarrow{\sim}\ComOp$ of the usual morphism $\AsOp\rightarrow\ComOp$
between the associative operad $\AsOp$ and the commutative operad $\ComOp$.

We take the category of algebras over the Barratt--Eccles operad to get our model of the category of $E_\infty$-algebras.
Recall that we have $\EOp(0) = \kk$ so that our $\EOp$-algebras are equipped with a unit, which is represented by the generating element of this arity zero term of our operad $*\in\EOp(0)$.

By the main result of the article~\cite{BergerFresse}, the normalized cochain complex of a simplicial set $\DGN^*(X)$
is endowed with the structure of an algebra over the Barratt--Eccles operad.
This $\EOp$-algebra structure is functorial in $X\in\simp\Set$,
and extends the classical associative algebra structure of normalized cochains.
\end{recoll}

\begin{constr}[The diagonal and the action of the Barratt--Eccles operad on tensor products]\label{constr:Barratt-Eccles-diagonal}
In our constructions, we use that the Alexander--Whitney diagonal on the normalized chain complexes $\EOp(r) = \DGN_*(W(\Sigma_r))$, $r\in\NN$,
induces a morphism of dg operads $\Delta: \EOp\rightarrow\EOp\otimes\EOp$, where $\EOp\otimes\EOp$
is given by the arity-wise tensor product $(\EOp\otimes\EOp)(r) = \EOp(r)\otimes\EOp(r)$,
for any $r\in\NN$.
This map is given by the usual formula:
\begin{equation*}
\Delta(w_0,\dots,w_n) = \sum_{k=0}^n(w_0,\dots,w_k)\otimes(w_k,\dots,w_n),
\end{equation*}
for any $(w_0,\dots,w_n)\in\EOp(r)$.

The existence of this diagonal implies that a tensor product of $\EOp$-algebras $A\otimes B$
inherits an $\EOp$-algebra structure, since we can make an operation $\pi\in\EOp(r)$
act on $A\otimes B$
through its diagonal $\Delta(\pi)\in\EOp(r)\otimes\EOp(r)$. We explicitly take:
\begin{equation*}
\pi(a_1\otimes b_1,\dots,a_r\otimes b_r) = \sum_{(\pi)}\pi'(a_1,\dots,a_r)\otimes\pi''(b_1,\dots,b_r),
\end{equation*}
for all $a_1\otimes b_1,\dots,a_r\otimes b_r\in A\otimes B$, where we use the expression $\Delta(\pi) = \sum_{(\pi)}\pi'\otimes\pi''$
for the expansion of the coproduct of the operation $\pi\in\EOp(r)$
in the Barratt--Eccles operad.


In the paper, we also use that we have a morphism of $\EOp$-algebras
\begin{equation*}
\AW: A\vee B\rightarrow A\otimes B,
\end{equation*}
for any pair of $\EOp$-algebras $A$ and $B$, where we adopt the notation $\vee$ for the coproduct in the category of $\EOp$-algebras.
This morphism is induced by the inclusions $A\otimes *\rightarrow A\otimes B\leftarrow *\otimes B$
on each factor of the coproduct $A\vee B$, where we still use the notation $*$
for the unit of the $\EOp$-algebras $A$ and $B$.
(We will see in the proof of the next proposition that we can identify this map with an instance of an Alexander--Whitney diagonal. We therefore adopt the notation $\AW$ for this morphism.)
We have a morphism of dg modules which goes in the converse direction
\begin{equation*}
\EM: A\otimes B\rightarrow A\vee B,
\end{equation*}
and which is given by the formula $\EM(a\otimes b) = \mu(a,b)$, for each tensor $a\otimes b\in A\otimes B$, where $\mu(a,b)$ denotes the product of the elements $a\in A$ and $b\in B$
in the $\EOp$-algebra $A\vee B$. (We are going to see that we can identify this map with an instance of an Eilenberg--MacLane map.)
Note that neither $\AW$ nor $\EM$ are symmetric, and actually, the tensor product $A\otimes B$
does not define a symmetric bifunctor on the category of $\EOp$-algebras since the Alexander--Whitney diagonal $\Delta: \EOp(r)\rightarrow\EOp(r)\otimes\EOp(r)$
fails to be cocommutative.
\end{constr}

We have the following useful property.

\begin{prop}\label{claim:Barratt-Eccles-algebra-coproducts}
The above morphisms $\AW: A\vee B\rightarrow A\otimes B$ and $\EM: A\otimes B\rightarrow A\vee B$ satisfy $\AW\EM = \id$
and we have a natural chain homotopy $H: A\vee B\rightarrow A\vee B$
such that $\delta H + H\delta = \EM\AW - \id$.
Hence, our morphisms induce homotopy inverse weak-equivalences of dg modules
\begin{equation*}
\AW: A\vee B\xrightarrow{\sim} A\otimes B\quad\text{and}\quad\EM: A\otimes B\xrightarrow{\sim} A\vee B,
\end{equation*}
for all $\EOp$-algebras $A$ and $B$.
\end{prop}

\begin{proof}
We consider the case of free $\EOp$-algebras first $A = \EOp(X)$ and $B = \EOp(Y)$.
We represent the elements of a free algebra such as $A = \EOp(X)$
by formal expressions of the form $a = u(x_1,\dots,x_r)$,
where $u\in\EOp(r)$ and $x_1,\dots,x_r\in X$.
We then have the expressions:
\begin{align*}
A\otimes B = \EOp(X)\otimes\EOp(Y) & = \bigoplus_{p,q}\EOp(p)\otimes_{\Sigma_p} X^{\otimes p}\otimes\EOp(p)\otimes_{\Sigma_q} Y^{\otimes q}, \\
A\vee B = \EOp(X\oplus Y) & = \bigoplus_{p,q}\EOp(p+q)\otimes_{\Sigma_p\times\Sigma_q} X^{\otimes p}\otimes Y^{\otimes q}.
\end{align*}
We use that the operadic composite $\id_2(u,v)$ of permutations $u\in\Sigma_p$ and $v\in\Sigma_q$
is identified with the result of a direct sum operation
such that $u\oplus v = (u(1),\dots,u(p),p+v(1),\dots,p+v(q))$.
Recall that $\mu = \id_2$ represents the associative product when we pass to the Barratt--Eccles operad $\EOp$.
For a tensor $a\otimes b = u(x_1,\dots,x_p)\otimes v(y_1,\dots,y_q)$
such that $u = (u_0,\dots,u_m)\in\EOp(p)$ and $v = (v_0,\dots,v_n)\in\EOp(q)$,
we have $\mu(a,b) = \mu(u,v)(x_1,\dots,x_p,y_1,\dots,y_q)$,
where we take the composite $\mu(u,v)$
in the Barratt--Eccles operad.
By definition of this composite in terms of shuffles of termwise composites $\mu(u_i,v_j) = u_i\oplus v_j$ (we apply the Eilenberg--MacLane map),
we obtain an expression of the following form:
\begin{multline*}
\EM(u(x_1,\dots,x_p)\otimes v(y_1,\dots,y_q))\\
= \sum\pm\underbrace{\left(\vcenter{\xymatrix@R=2mm@C=3mm{ u_0\oplus v_0\ar@{-}[d]\ar@{-}[r] &
u_1\oplus v_0\ar@{-}[d]\ar@{-}[r] &
*{\cdots}\ar@{-}[r] &
u_m\oplus v_0\ar@{-}[d] \\
u_0\oplus v_1\ar@{-}[d]\ar@{-}[r] &
u_1\oplus v_1\ar@{-}[d]\ar@{-}[r] &
*{\cdots}\ar@{-}[r] &
u_m\oplus v_1\ar@{-}[d] \\
*{\vdots}\ar@{-}[d] & *{\vdots}\ar@{-}[d] & & *{\vdots}\ar@{-}[d] \\
u_0\oplus v_n\ar@{-}[r] &
u_1\oplus v_n\ar@{-}[r] &
*{\cdots}\ar@{-}[r] &
u_m\oplus v_n }}\right)}_{\in\EOp(p+q)}
(x_1,\dots,x_p,y_1,\dots,y_q),
\end{multline*}
where the sum runs over all paths that we may form in the diagram of our figure (as in our representation of the composition of the Barratt--Eccles operad in~\S\ref{recoll:Barratt-Eccles-operad}).
We use this representation to identify our morphism $\EM: A\otimes B\rightarrow A\vee B$ with an instance of an Eilenberg--MacLane map.

The morphism $\AW: A\vee B\rightarrow A\otimes B$, on the other hand, carries any free algebra element $c = w(x_1,\dots,x_p,y_1,\dots,y_q)$
such that $w = (w_0,\dots,w_n)\in\EOp(p+q)$
to the image of the elements $x_1\otimes *,\dots,x_p\otimes *,*\otimes y_1,\dots,*\otimes y_q\in\EOp(X)\otimes\EOp(Y)$
under the action of the operation $w$ on $\EOp(X)\otimes\EOp(Y)$,
and hence to the tensor such that $\AW(w(x_1,\dots,x_p,y_1,\dots,y_q)) = \sum_{(w)} w'(x_1,\dots,x_p,*,\dots,*)\otimes w''(*,\dots,*,y_1,\dots,y_q)$,
where we use the notation $\Delta(w) = \sum_{(w)}w'\otimes w''$
for the expansion of the coproduct of the simplex $w$
in the Barratt--Eccles operad. Thus, if we assume $w = (w_0,\dots,w_n)$, then we have $\sum_{(w)}w'\otimes w'' = \sum_{k=0}^n (w_0,\dots,w_k)\otimes(w_k,\dots,w_n)$.
From this analysis, we deduce that our morphism $\AW$ is given by the following Alexander--Whitney type formula:
\begin{multline*}
\AW(w(x_1,\dots,x_p,y_1,\dots,y_q))\\
= \sum_{k=0}^n (w_0|_I,\dots,w_k|_I)(x_1,\dots,x_p)\otimes(w_k|_J,\dots,w_n|_J)(y_1,\dots,y_q),
\end{multline*}
where we set $I = \{1,\dots,p\}$ and $J = \{1,\dots,q\}$ for short and $|_{I}$, $|_{J}$
denote the obvious restriction operations on permutations
which we apply to our simplices termwise. (In this construction, we also use the canonical bijection $\{p+1,\dots,p+q\}\simeq\{1,\dots,q\}$
to identify the permutations of the set $\{p+1,\dots,p+q\}$ with permutations of the set $\{1,\dots,q\}$.)

For short, we may adopt the notation $x_*$ and $y_*$ for the words of variables $x_* = x_1,\dots,x_p$ and $y_* = y_1,\dots,y_q$
that occur in our expression of free algebra elements.
The chain homotopy $H$ can be given by a formula of the following form:
\begin{multline*}
H(w(x_*,y_*)) = \sum\pm(w_0,\dots,w_k)\\
\star\left(\vcenter{\xymatrix@R=2mm@C=3mm{ \scriptstyle{w_k|_I\oplus w_l|_J}\ar@{-}[d]\ar@{-}[r] &
\scriptstyle{w_{k+1}|_I\oplus w_l|_J}\ar@{-}[d]\ar@{-}[r] &
*{\cdots}\ar@{-}[r] &
\scriptstyle{w_l|_I\oplus w_l|_J}\ar@{-}[d] \\
\scriptstyle{w_k|_I\oplus w_{l+1}|_J}\ar@{-}[d]\ar@{-}[r] &
\scriptstyle{w_{k+1}|_I\oplus w_{l+1}|_J}\ar@{-}[d]\ar@{-}[r] &
*{\cdots}\ar@{-}[r] &
\scriptstyle{w_l|_I\oplus w_{l+1}|_J}\ar@{-}[d] \\
*{\vdots}\ar@{-}[d] & *{\vdots}\ar@{-}[d] & & *{\vdots}\ar@{-}[d] \\
\scriptstyle{w_k|_I\oplus w_n|_J}\ar@{-}[r] &
\scriptstyle{w_{k+1}|_I\oplus w_n|_J}\ar@{-}[r] &
*{\cdots}\ar@{-}[r] &
\scriptstyle{w_l|_I\oplus w_n|_J} }}\right)(x_*,y_*),
\end{multline*}
where $\star$ denotes a ``join'' operation (given by the obvious concatenation operation in $W(\Sigma_*)$), and the sum runs over the indices $0\leq k\leq l\leq n$,
together with the set of paths $\{(i_t,j_t),t=k,\dots,n\}$ which start at $(i_k,j_k) = (k,l)$ and end at $(i_n,j_n) = (l,n)$
in the $l-k\times n-l$-diagram represented in our figure.

The relation $\AW\EM = \id$ is an instance of the general inversion relation between the Alexander--Whitney diagonal and the Eilenberg--MacLane map.
In our context, we can also deduce this relation from the observation that the element $a\otimes b\in A\otimes B$
in a tensor product of $\EOp$-algebras $A$ and $B$
represents the product of the tensors $a\otimes *,*\otimes b\in A\otimes B$,
so that we have the identity $a\otimes b = \mu(a\otimes *,*\otimes b)$
in $A\otimes B$.

The proof of the chain homotopy relation $\delta H + H\delta = \AW\EM - \id$
is straightforward: the composite $\AW\EM$ corresponds to the $0$-face of the terms with $k=0$ in the expression of $H$
while the identity map corresponds to the $n$-face of the terms with $k=l=n$, and the other faces cancel out
when we form the anti-commutator $\delta(H) = \delta H + H\delta$.

The morphisms $\AW$ and $\EM$, which are defined for all $\EOp$-algebras $A$ and $B$,
are obviously functorial. We check that our chain homotopy $H$ is also functorial
with respect to the action of the morphisms
of free $\EOp$-algebras $\phi: \EOp(X)\rightarrow\EOp(X)$ and $\psi: \EOp(Y)\rightarrow\EOp(Y)$
on the coproduct $\EOp(X)\vee\EOp(Y) = \EOp(X\oplus Y)$.
Thus, we establish that we have the following relation:
\begin{equation*}
H((\phi\vee\psi)(w(x_*,y_*))) = (\phi\vee\psi)(H(w(x_*,y_*))),
\end{equation*}
for all $c = w(x_*,y_*)\in\EOp(X\oplus Y)$.

We use that $\phi = \phi_f: \EOp(X)\rightarrow\EOp(X)$ and $\psi = \psi_g: \EOp(Y)\rightarrow\EOp(Y)$
are induced by morphisms of dg modules $f: X\rightarrow\EOp(X)$
and $g: Y\rightarrow\EOp(Y)$.
We use the short notation $f(x_i) = \sum s^i(\underline{x}_i')$ and $g(y_j) = \sum t^j(\underline{y}_j')$
for the expansion of these free algebra elements $f(x_i)\in\EOp(X)$ and $g(y_j)\in\EOp(Y)$,
which we associate to the factors of a tensor $c = w(x_1,\dots,x_p,y_1,\dots,y_q)$.
We also write $s^i = (s^i_0,\dots,s^i_{d_i})$ and $t^j = (t^j_0,\dots,t^j_{e_j})$, where $d_i$ and $e_j$ are dimension variables.
We have the formula:
\begin{equation}\tag{$*$}\label{eqn:coproduct_functoriality}
(\phi\vee\psi)(w(x_*,y_*)) = w(f(x_*),g(y_*)) = \sum w(s^*,t^*)(\underline{x}_*',\underline{y}_*'),
\end{equation}
where we form the composite $w(s^*,t^*) = w(s^1,\dots,s^p,t^1,\dots,t^q)$ inside the Barratt--Eccles operad.
(In this formula, we also use the notation $s^*$ and $t^*$ for the words of simplices $s^* = s^1,\dots,s^p$ and $t^* = t^1,\dots,t^q$,
as well as the notation $\underline{x}_*'$ and $\underline{y}_*'$
for the composite words $\underline{x}_*' = \underline{x}_1',\dots,\underline{x}_p'$
and $\underline{y}_*' = \underline{y}_1',\dots,\underline{y}_q'$.)
We then use a multidimensional generalization of the picture of~\S\ref{recoll:Barratt-Eccles-operad}
for the definition of the operadic composition
in the Barratt--Eccles operad. We are going to use that the shuffles of horizontal and vertical moves, which we carry out in this definition of operadic composites,
satisfy natural associativity and commutativity relations when we perform a multidimensional
application of this operation.

We analyze the expression of the composite $(\phi\vee\psi)(H(w(x_*,y_*)))$ first.
We then apply the composition operation $w'\mapsto w'(s^*,t^*)$
to the simplices that occur in the expression of the chain homotopy $H(w(x_*,y_*))$.
We decompose the result of this operation as a join of simplices, using the join decomposition of the simplices that occur in the definition of our chain homotopy.
We identify the first factor of our join with a chain of composite permutations $w_x(s^*_{\alpha_*},t^*_{\beta_*})$,
where $w_x$ runs over the vertices of the simplex $(w_0,\dots,w_k)$,
which represents the first factor of our join in the expression of $H(w(x_*,y_*))$.
We take all shuffles of $w_x$-directional moves in the chain $(w_0,\dots,w_k)$
with $s^i_{\alpha_i}$-directional and $t^j_{\beta_j}$-directional moves
in subchains of the simplices $s^i$ and $t^j$
of the form $s^i{}' = (s^i_0,\dots,s^i_{d_i'})$ and $t^j{}' = (t^j_0,\dots,t^j_{e_j'})$,
where $d_i'\leq d_i$ and $e_j'\leq e_j$.
We identify the second factor of our join with a chain of composite permutations
of the form $w_x|_I(s^*_{\alpha_*})\oplus w_y|_J(t^*_{\beta_*})$,
where $w_x|_I\oplus w_y|_J$
runs over the vertices of the second join factor in the expression of $H(w(x_*,y_*))$, starting at $w_k|_I\oplus w_l|_J$
and ending at $w_l|_I\oplus w_n|_J$.
When we pass from the previous join factor to this second join factor in our computation of $(\phi\vee\psi)(H(w(x_*,y_*)))$,
we carry out a move of the form $\xymatrix@R=2mm@C=3mm{w_k\ar@{-}[r] & w_k|_I\oplus w_l|_J }$,
and hence, our move has to be constant in the $s^i_{\alpha_i}$ directions
and in the $t^j_{\beta_j}$ directions.
This observation implies that we start this simplex at the end point of the chains $s^i{}' = (s^i_0,\dots,s^i_{d_i'})$ in the $s^i_{\alpha_i}$ directions
and at the end point of the chains $t^j{}' = (t^j_0,\dots,t^j_{e_j'})$
in the $t^j_{\beta_j}$ directions.
Thus, to form the chains of composite permutations of our second join factor, we shuffle $w_x|_I$-directional and $w_y|_J$-directional moves
in the $2$-dimensional diagram $(w_k|_I,\dots,w_l|_I)\times(w_l|_J,\dots,w_n|_J)$
with $s^i_{\alpha_i}$-directional moves in the chains $s^i{}'' = (s^i_{d_i'},\dots,s^i_{d_i})$
and $t^j_{\beta_j}$-directional moves in the chains $t^j{}'' = (t^j_{e_j'},\dots,t^j_{e_j})$.

We now analyze the expression of the composite $H((\phi\vee\psi)(w(x_*,y_*)))$,
which we obtain by applying our chain homotopy $H$
to the element $(\phi\vee\psi)(w(x_*,y_*))
= \sum w(s^*,t^*)(\underline{x}_*',\underline{y}_*')$.
The simplices of the composite $w(s^*,t^*)$
consist of chains of composite permutations $w_x(s^*_{\alpha_*},t^*_{\beta_*})$,
which we obtain after shuffling $w_x$-directional moves in the chain $w = (w_0,\dots,w_n)$
with $s^i_{\alpha_i}$-directional moves in the chains $s^i{}'' = (s^i_0,\dots,s^i_{d_i})$
and $t^j_{\beta_j}$-directional moves in the chains $t^j{}'' = (t^j_0,\dots,t^j_{e_j})$.

To form our chain homotopy, we cut this chain at two positions $w_k(s^*_{d_*'},t^*_{e_*'}) = w_k(s^1_{d_1'},\dots,s^p_{d_p'},t^1_{e_1'},\dots,t^q_{e_q'})$
and $w_l(s^*_{d_*''},t^*_{e_*''}) = w_l(s^1_{d_1''},\dots,s^p_{d_p''},t^1_{e_1''},\dots,t^q_{e_q''})$ with $0\leq k\leq l\leq n$
and $0\leq d_i'\leq d_i''\leq d_i$, $0\leq e_j'\leq e_j''\leq e_j$.
We take the subchain of permutations running from $w_0(s^*_0,t^*_0)$
up to $w_k(s^*_{d_*'},t^*_{e_*'})$
to get the first join factor of our chain homotopy. We exactly retrieve the same chains as in our decomposition of~$(\phi\vee\psi)(H(w(\underline{x},\underline{y})))$.

We then form the restrictions $w_x(s^*_{\alpha_*},t^*_{\beta_*})|_{I'} = w_x|_I(s^*_{\alpha_*})$
and $w_y(s^*_{\alpha_*},t^*_{\beta_*})|_{J'} = w_y|_J(t^*_{\beta_*})$
where $I'$ denotes the terms of the composite permutations $w_x(s^*_{\alpha_*},t^*_{\beta_*})$
that correspond to the positions of the variables $\underline{x}_*'$,
whereas $J'$ denotes the terms of the composite permutations $w_y(s^*_{\alpha_*},t^*_{\beta_*})$
that correspond to the positions of the variables $\underline{y}_*'$.
We take a chain of direct sums $w_x(s^*_{\alpha_*},t^*_{\beta_*})|_{I'}\oplus w_y(s^*_{\alpha_*},t^*_{\beta_*})|_{J'}$
such that $w_x(s^*_{\alpha_*},t^*_{\beta_*})$ runs from $w_k(s^*_{d_*'},t^*_{e_*'})$ up to $w_l(s^*_{d_*''},t^*_{e_*''})$,
while $w_y(s^*_{\alpha_*},t^*_{\beta_*})$ runs from $w_l(s^*_{d_*''},t^*_{e_*''})$ up to $w_l(s^*_{d_*},t^*_{e_*})$.
(We then take a shuffle of $w_x(s^*_{\alpha_*},t^*_{\beta_*})$-directional moves and of $w_y(s^*_{\alpha_*},t^*_{\beta_*})$-direction moves.)
We easily see that this operation produces a degeneracy in the case where we have $e_j'<e_j''$ for some $j$,
because in this case the subchain of permutations $w_x(s^*_{\alpha_*},t^*_{\beta_*})$
contains a $t^j_{\beta_j}$-directional move, which produces a degeneracy when we pass
to the restriction $w_x(s^*_{\alpha_*},t^*_{\beta_*})|_{I'} = w_x|_I(s^*_{\alpha_*})$.
We similarly see that our operation produces a degeneracy in the case where we have $d_i''<d_i$ for some $i$.
We therefore have to assume $d_i'' = d_i$ for all $i$ and $e_j' = e_j$ for all $j$ in order to avoid possible degeneracies,
and in these cases, we exactly retrieve the same chains as in our expression of the second join factor in our decomposition of~$(\phi\vee\psi)(H(w(x_*,y_*)))$.

We conclude from this analysis that the expansions of $H((\phi\vee\psi)(w(x_*,y_*)))$ and $(\phi\vee\psi)(H(w(x_*,y_*)))$ consist of the same joins of simplices,
and therefore these composites agree, as expected.

To finish the proof of our proposition, we use that every object of the category of $\EOp$-algebras has a presentation in terms of a natural reflexive coequalizer of free $\EOp$-algebras.
If we have a pair of objects $A$ and $B$, then we can form a commutative diagram:
\begin{equation*}
\xymatrix{ \EOp(X_1)\otimes\EOp(Y_1)\ar@<-2pt>[r]_{\EM}\ar@<+2pt>[d]\ar@<-2pt>[d] &
\EOp(X_1)\vee\EOp(Y_1)\ar@<-2pt>[l]_{\AW}\ar@<+2pt>[d]\ar@<-2pt>[d]\ar@(ur,dr)[]!UR;[]!DR^{H} \\
\EOp(X_0)\otimes\EOp(Y_0)\ar@<-2pt>[r]_{\EM}\ar@/_1em/[u]\ar@{.>}[d] &
\EOp(X_0)\vee\EOp(Y_0)\ar@<-2pt>[l]_{\AW}\ar@/_1em/[u]\ar@{.>}[d]\ar@(ur,dr)[]!UR;[]!DR^{H} \\
A\otimes B\ar@<-2pt>[r]_{\EM} & A\vee B\ar@<-2pt>[l]_{\AW}\ar@{.>}@(ur,dr)[]!UR;[]!DR^{H} }
\end{equation*}
where we take the presentations of $A$ and $B$ by reflexive coequalizers in the vertical direction and our deformation retract diagram
of Alexander--Whitney and Eilenberg--MacLane maps
in the horizontal direction. We use this diagram to check that our chain homotopy $H$ passes to the quotient
and induces a chain homotopy such that $\delta H + H\delta = \EM\AW - \id$
on $A\vee B$,
as indicated in our figure.
\end{proof}

\begin{constr}[The action of the Barratt--Eccles operad on the interval, on cubical cochains and the definition of connections]\label{constr:cubical-cochain-connection}
We already recalled that the normalized cochain complex of a simplicial set $\DGN^*(X)$ inherits the structure of an algebra over the Barratt--Eccles operad.
We refer to~\cite{BergerFresse} for the precise definition.
We consider the cochain algebra of the $1$-simplex $X = \Delta^1$ in our definition of homotopy Segal $E_\infty$-cooperads
together with the cubical cochain algebras
\begin{equation*}
I^m = \underbrace{\DGN^*(\Delta^1)\otimes\dots\otimes\DGN^*(\Delta^1)}_m
\end{equation*}
which we provide with an $\EOp$-algebra structure, using the action of the Barratt--Eccles operad on each cochain complex factor $\DGN^*(\Delta^1)$,
and the diagonal operation of Construction~\ref{constr:Barratt-Eccles-diagonal}.

We study structures attached to the cochain algebra $I = \DGN^*(\Delta^1)$ in this paragraph.
We also consider the normalized chain complex $\DGN_*(\Delta^1)$, dual to $\DGN^*(\Delta^1)$.
We have
\begin{equation*}
\DGN_*(\Delta^1) = \kk\underline{0}\oplus\kk\underline{1}\oplus\kk\underline{01},
\end{equation*}
where $\underline{01}$ denotes the class of the fundamental simplex of $\Delta^1$ in the normalized chain complex,
whereas $\underline{0}$ and $\underline{1}$ denote the class of the vertices of $\underline{01}$.
We can identify $\kk\underline{0}\subset\DGN_*(\Delta^1)$ with the image of the map $\DGN_*(d^1): \DGN_*(\Delta^0)\rightarrow\DGN_*(\Delta^1)$
induced by the $1$-coface $d^1: \Delta^0\rightarrow\Delta^1$,
while $\kk\underline{1}\subset\DGN_*(\Delta^1)$ is identified with the image of the map $d^0: \DGN_*(\Delta^0)\rightarrow\DGN_*(\Delta^1)$
induced by the $0$-coface $d^0: \Delta^0\rightarrow\Delta^1$. We have $\delta(\underline{01}) = \underline{1} - \underline{0}$.
For the cochain algebra, we have
\begin{equation*}
\DGN^*(\Delta^1) = \kk\underline{0}^{\sharp}\oplus\kk\underline{1}^{\sharp}\oplus\kk\underline{01}^{\sharp},
\end{equation*}
where we take the basis $(\underline{0}^{\sharp},\underline{1}^{\sharp},\underline{01}^{\sharp})$
dual to $(\underline{0},\underline{1},\underline{01})$.

We now explain the definition of a connection $\nabla^*: \DGN^*(\Delta^1)\otimes\DGN^*(\Delta^1)\rightarrow\DGN^*(\Delta^1)$,
which we use in the construction of degeneracy operators in our definition of homotopy Segal $E_\infty$-cooperads.
We consider the simplicial map $\min: \Delta^1\times\Delta^1\rightarrow\Delta^1$ defined by the mapping $(s,t)\mapsto\min(s,t)$
on topological realizations, or equivalently, by the following representation:
\begin{equation*}
\vcenter{\hbox{\includegraphics{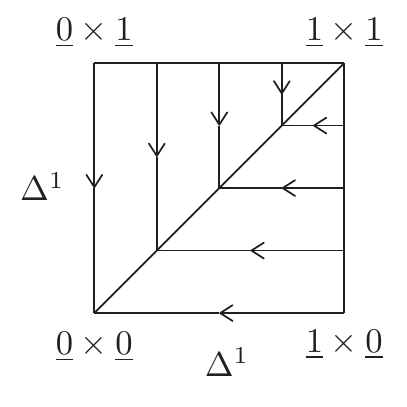}}},
\end{equation*}
where we take the projection onto the diagonal simplex along the lines depicted in the figure.
We take the composite $\nabla_* = \DGN_*(\min)\circ\EM$ of the induced map on normalized chain complexes $\DGN_*(\min): \DGN_*(\Delta^1\times\Delta^1)\rightarrow\DGN_*(\Delta^1)$
with the Eilenberg--MacLane map $\EM: \DGN_*(\Delta^1)\otimes\DGN_*(\Delta^1)\rightarrow\DGN_*(\Delta^1\times\Delta^1)$.
We get the following formulas:
\begin{align*}
& \nabla_*(\underline{0}\otimes\underline{0}) = \nabla_*(\underline{1}\otimes\underline{0}) = \nabla_*(\underline{0}\otimes\underline{1}) = \underline{0}, \\
& \nabla_*(\underline{1}\otimes\underline{1}) = \underline{1}, \\
& \nabla_*(\underline{1}\otimes\underline{01}) = \nabla_*(\underline{01}\otimes\underline{1}) = \underline{01}, \\
& \nabla_*(\underline{0}\otimes\underline{01}) = \nabla_*(\underline{01}\otimes\underline{0}) = 0, \\
& \nabla_*(\underline{01}\otimes\underline{01}) = 0.
\end{align*}
We define our connection on normalized cochains $\nabla^*$ as the dual map of this morphism $\nabla_*$.
We accordingly take:
\begin{equation*}
\nabla^* = \EM\circ\DGN^*(\min): \DGN^*(\Delta^1)\rightarrow\DGN^*(\Delta^1)\otimes\DGN^*(\Delta^1),
\end{equation*}
and we can determine this morphism by the following formulas on our basis elements:
\begin{align*}
& \nabla^*(\underline{0}^{\sharp}) = \underline{0}^{\sharp}\otimes\underline{0}^{\sharp} + \underline{1}^{\sharp}\otimes\underline{0}^{\sharp}
+ \underline{0}^{\sharp}\otimes\underline{1}^{\sharp}, \\
& \nabla^*(\underline{1}^{\sharp}) = \underline{1}^{\sharp}\otimes\underline{1}^{\sharp}, \\
& \nabla^*(\underline{01}^{\sharp}) = \underline{01}^{\sharp}\otimes\underline{1}^{\sharp} + \underline{1}^{\sharp}\otimes\underline{01}^{\sharp}.
\end{align*}
We crucially need the observation of the next proposition in our constructions.
\end{constr}

\begin{prop}\label{prop:cubical-cochain-connection}
The map $\nabla^* = \EM\circ\DGN^*(\min): \DGN^*(\Delta^1)\rightarrow\DGN^*(\Delta^1)\otimes\DGN^*(\Delta^1)$ defined in the above paragraph is a morphism of $\EOp$-algebras,
where we use that the Barratt-Eccles operad $\EOp$ acts on the tensor product $\DGN^*(\Delta^1)\otimes\DGN^*(\Delta^1)$
through the operadic diagonal $\Delta: \EOp\rightarrow\EOp\otimes\EOp$
and its action on each factor $\DGN^*(\Delta^1)$.
\end{prop}

\begin{proof}
We go back to the definition of the $\EOp$-algebra structure of normalized cochain complexes of simplicial sets
in terms of a dual $\EOp$-coalgebra structure on normalized chain complexes.
We prove that the morphism $\nabla_*: \DGN_*(\Delta^1)\otimes\DGN_*(\Delta^1)\rightarrow\DGN_*(\Delta^1)$,
dual to the morphism of our claim, is a morphism of $\EOp$-coalgebras.
We may note that the Eilenberg--MacLane map $\EM: \DGN_*(X)\otimes\DGN_*(Y)\rightarrow\DGN_*(X\times Y)$
does not preserve $\EOp$-coalgebra structures
in general. Nevertheless, such a statement holds when one factor is a one-point set, $X = *$ or $Y = *$,
because in this case, we have $\DGN_*(X)\simeq\DGN_*(X)\otimes\kk\simeq\DGN_*(X\times *)$
or $\DGN_*(Y)\simeq \kk\otimes\DGN_*(Y)\simeq\DGN_*(*\times Y)$, and the Eilenberg--MacLane map
reduces to the identity morphism on the functor of normalized chains.
We readily deduce from this observation that our morphism $\nabla_*: \DGN_*(\Delta^1)\otimes\DGN_*(\Delta^1)\rightarrow\DGN_*(\Delta^1)$
preserves $\EOp$-coalgebra structure on the subcomplex generated by the tensors $\underline{\sigma}\otimes\underline{\tau}\in\DGN_*(\Delta^1)\otimes\DGN_*(\Delta^1)$
such that $\underline{\sigma}\in\{\underline{0},\underline{1}\}$
or $\underline{\tau}\in\{\underline{0},\underline{1}\}$
since such tensors lie in the image of the coface maps $d^i\otimes\id: \DGN^*(\Delta^0)\otimes\DGN_*(\Delta^1)\rightarrow\DGN_*(\Delta^1)\otimes\DGN_*(\Delta^1)$
and $\id\otimes d^i: \DGN^*(\Delta^1)\otimes\DGN_*(\Delta^1)\rightarrow\DGN_*(\Delta^0)\otimes\DGN_*(\Delta^1)$,
with $i = 0,1$.

We use the notation $\pi_*: \DGN_*(\Delta^1)\rightarrow\DGN_*(\Delta^1)^{\otimes r}$ for the operation
that we associate to an element of the Barratt--Eccles operad $\pi\in\EOp(r)$
in the definition of the $\EOp$-coalgebra
structure on $\DGN_*(\Delta^1)$.
In general, for a tensor $\underline{\sigma}\otimes\underline{\tau}\in\DGN_*(\Delta^1)\otimes\DGN_*(\Delta^1)$, we have the formula:
\begin{equation*}
\pi_*(\underline{\sigma}\otimes\underline{\tau}) = \sum_{(\pi)}\sh(\pi'_*(\underline{\sigma})\otimes\pi''_*(\underline{\tau})),
\end{equation*}
where $\Delta(\pi) = \sum_{(\pi)}\pi'\otimes\pi''$ denotes the coproduct of the operation $\pi$,
while $\sh: \DGN_*(\Delta^1)^{\otimes r}\otimes\DGN_*(\Delta^1)^{\otimes r}\rightarrow(\DGN^*(\Delta^1)\otimes\DGN^*(\Delta^1))^{\otimes r}$
is the tensor permutation such that $\sh(a_1\otimes\dots\otimes a_r\otimes b_1\otimes\dots\otimes b_r) = (a_1\otimes b_1)\otimes\dots\otimes(a_r\otimes b_r)$.
Thus the statement of our claim is equivalent to the following relation:
\begin{equation}\tag{$*$}\label{eqn:nabla_product}
\pi_*(\nabla_*(\underline{\sigma}\otimes\underline{\tau})) = \sum_{(\pi)}\nabla_*^{\otimes r}\sh(\pi'_*(\underline{\sigma})\otimes\pi''_*(\underline{\tau})),
\end{equation}
for $\pi\in\EOp(r)$, and for any $\underline{\sigma}\otimes\underline{\tau}\in\DGN_*(\Delta^1)\otimes\DGN_*(\Delta^1)$.

We can use the argument of the previous paragraph to establish the validity of this relation in the case where $\underline{\sigma}$ or $\underline{\tau}$
is a vertex $\underline{0},\underline{1}\in\DGN_*(\Delta^1)$.
We therefore focus on the case $\underline{\sigma}\otimes\underline{\tau} = \underline{01}\otimes\underline{01}$.
We then have $\nabla_*(\underline{01}\otimes\underline{01}) = 0$,
so that the above equation (\ref{eqn:nabla_product})
reduces to the following vanishing relation:
\begin{equation}\tag{$*'$}\label{eqn:vanishing_nabla_product}
\sum_{(\pi)}\nabla_*^{\otimes r}\sh(\pi'_*(\underline{01})\otimes\pi''_*(\underline{01})) = 0.
\end{equation}
We devote the rest of this proof to the verification of this relation.

\textit{The definition of the action of the Barratt--Eccles operad on chains}.
To carry out our verification, we have to go back to the explicit expression of the operation $\varpi_*: \DGN_*(\Delta^1)\rightarrow\DGN_*(\Delta^1)^{\otimes r}$
associated to an element of the Barratt--Eccles operad $\varpi\in\EOp(r)$
in terms of interval cuts associated to a table reduction
of the simplices of the permutations $(s_0,\dots,s_l)$
that occur in the expansion of $\varpi$.
We briefly recall this construction in the general case of a $q$-dimensional simplex $\Delta^q$.
We refer to~\cite{BergerFresse} for details.

The table reduction is a sum of surjective maps $s: \{1,\dots,r+l\}\rightarrow\{1,\dots,r\}$,
which we determine by sequences of values $s = (s(1),\dots,s(r+l))$,
which we arrange on a table, as in the following picture:
\begin{equation*}
s = \left|\begin{array}{l} s(1),\dots,s(e_0-1),s(e_0) \\
s(e_0+1),\ldots,s(e_1-1),s(e_1) \\
\vdots\\
s(e_{l-2}+1),\dots,s(e_{l-1}-1),s(e_{l-1}) \\
s(e_{l-1}+1),\dots,s(r+l)
\end{array}\right..
\end{equation*}
The caesuras $s(e_0),\dots,s(e_{l-1})$, which terminate the rows of the table, are the terms $y = s(x)$ of the sequence $s = (s(1),\dots,s(r+l))$
that do not form the last occurrence of a value $y\in\{1,\dots,r\}$.
Thus, the complement of the caesuras, which consists of the inner terms of the rows $s(e_{i-1}+1),\dots,s(e_i-1)$, $i = 0,\dots,l-1$,
and of the terms of the last rows $s(e_{l-1}+1),\dots,s(r+l)$, consists of the terms of the sequence $s = (s(1),\dots,s(r+l))$
which are not repeated after their position.

The table reduction of a simplex of permutations $\varpi = (s_0,\dots,s_l)$
is a sum of table arrangements of surjections
of the following form:
\begin{equation*}
\TR(s_0,\dots,s_l) = \sum\left|\begin{array}{l} s_0(1),\dots,s_0(r_0-1),s_0(r_0) \\
s_1(1)',\dots,s_1(r_1-1)',s_1(r_1)' \\
\vdots \\
s_{l-1}(1)',\dots,s_{l-1}(r_{l-1}-1)',s_{l-1}(r_{l-1})' \\
s_l(1)',\dots,s_l(r_l)' \end{array}\right.,
\end{equation*}
and which we obtain by browsing the terms of our permutations $s_i$, $i = 0,\dots,l$.
For $i = 0$, we retain all terms of our permutation $s_0(1),\dots,s_0(x),\dots$ up to the choice of a caesura $s_0(r_0)$, where we decide to stop this enumeration.
For $i > 0$, in the enumeration of the terms of the permutation $s_i$ we only retain the values $s_i(1)',\dots,s_i(x)',\dots$
that do not occur before the caesura on the previous rows of our table.
For $i < l$, we again stop this enumeration at the choice of a caesura $s_i(r_i)$.
For $i = l$, we run this process up to the last term of the permutation $s_l$. We sum over all possible choices of caesuras.

For $0\leq\upsilon_0\leq\dots\leq\upsilon_p\leq q$, we generally denote by $\underline{\upsilon_0\dots\upsilon_p}\in\Delta^q$
the $p$-simplex defined by taking the image of the fundamental simplex of the $q$-simplex $\Delta^q$
under the simplicial operator $u^*: \Delta^q_q\rightarrow\Delta^q_p$
associated to the map $u: \{0<\dots<p\}\rightarrow\{0<\dots<q\}$
such that $u(x) = \upsilon_x$, $x = 0,\dots,p$.
The notation $\underline{0\cdots q}$, for instance, represents the fundamental simplex of $\Delta^q$.

Each surjection $s = (s(1),\dots,s(l+r))$
in the table reduction
of an element of the Barratt--Eccles operad $\varpi = (s_0,\dots,s_l)$
is used to assign a sum of tensors
\begin{equation*}
s_*(\underline{0\cdots q}) = \sum_{\alpha}\pm\underline{\sigma}_{(1)}^{\alpha}\otimes\dots\otimes\underline{\sigma}_{(r)}^{\alpha}\in\DGN_*(\Delta^q)^{\otimes r}
\end{equation*}
to the fundamental simplex $\underline{0\cdots q}\in\Delta^q_q$.
We proceed as follows. We fix a sequence of indices $0 = \rho_0\leq\dots\leq\rho_x\leq\dots\leq\rho_{r+l} = q$,
which we associate to an interval decomposition of the indexing sequence
of the fundamental simplex:
\begin{equation*}
\underline{0\cdots q} = \underline{\rho_0\cdots\rho_1}|\underline{\rho_1\cdots\rho_2}|\cdots\,\cdots|\underline{\rho_{r+l-1}\cdots\rho_{r+l}}.
\end{equation*}
For $x = 1,\dots,r+l$, we label the interval $\underline{\rho_{x-1}\cdots\rho_x}$
with the value of the term $s(x)$ of our surjection $s = (s(1),\dots,s(r+l))$
in $\{1,\dots,r\}$,
as in the following picture:
\begin{equation*}
\underline{\rho_0\overset{s(1)}{\cdots}\rho_1}|\underline{\rho_1\overset{s(2)}{\cdots}\rho_2}|\cdots
\,\cdots|\underline{\rho_{x-1}\overset{s(x)}{\cdots}\rho_x}|\cdots
\,\cdots|\underline{\rho_{r+l-1}\overset{s(r+l)}{\cdots}\rho_{r+l}}.
\end{equation*}
Then, for $i\in\{1,\dots,r\}$, we concatenate the intervals $\underline{\rho_{x-1}\cdots\rho_x}$ labelled by the value $s(x) = i$
in order to form the factor $\sigma_{(i)}^{\alpha}$ of our tensor $s_*(\underline{0\cdots q})\in\DGN_*(\Delta^q)^{\otimes r}$.
We sum over all possible choices of indices $0 = \rho_0\leq\dots\leq\rho_x\leq\dots\leq\rho_{r+l} = q$.

The image of the simplex $\underline{0\cdots q}\in\DGN_*(\Delta^q)$ under the operation $\varpi: \DGN_*(\Delta^q)\rightarrow\DGN_*(\Delta^q)^{\otimes r}$
associated to an element of the Barratt-Eccles operad $\varpi\in\EOp(r)$
is given by the sum of the tensors $s_*(\underline{0\cdots q})\in\DGN_*(\Delta^q)^{\otimes r}$,
where $s$ runs over surjections that occur in the table reduction of $\varpi$.

In general, we can determine the action of the operation $\varpi_*: \DGN_*(X)\rightarrow\DGN_*(X)^{\otimes r}$
on an element $\sigma\in\DGN_*(X)$ in the normalized chains of a simplicial set
by using that $\sigma$ is represented by a simplex $\sigma\in X_q$ such that $\sigma = \hat{\sigma}(\underline{0\cdots q})$
for some simplicial map $\hat{\sigma}: \Delta^q\rightarrow X$.
Indeed, by functoriality of the operation $\varpi_*: \DGN_*(X)\rightarrow\DGN_*(X)^{\otimes r}$,
we have the identity $\varpi_*(\sigma) = \DGN_*(\hat{\sigma})^{\otimes r}(\varpi_*(\underline{0\cdots q}))$
in $\DGN_*(X)^{\otimes r}$.
But we do not use really this correspondence in the follow-up, because we are going to focus on the computation of the tensors $\varpi_*(\sigma)\in\DGN_*(X)^{\otimes r}$
for the fundamental simplex of the $1$-simplex $\underline{01}\in\Delta^1$.
We study the terms that may remain in the expansion of $\varpi_*(\underline{01})\in\DGN_*(\Delta^1)^{\otimes r}$ after the reduction of the factors $\sigma_{(i)}^{\alpha}\in\DGN_*(\Delta^1)$
associated to degenerate simplices in $\Delta^1$.

\textit{The reduction of degenerate simplices for the action of the Barratt--Eccles operad on the $1$-simplex $\underline{01}\in\DGN_*(\Delta^1)$}.
In general, in order to get non-degenerate simplices $\sigma_{(i)}^{\alpha}$ in the above definition of the tensors $s_*(\underline{0\cdots q})\in\DGN_*(\Delta^q)^{\otimes r}$,
we have to assume that we have a strict inequality $\rho_x<\rho_{y-1}+1$ for each pair $x<y$
such that $s(x)$ and $s(y)$ represent consecutive occurrences of a value $s(x) = s(y) = i$
in our surjection $s(1),\dots,s(r+l)$.

In the case $q = 1$, we must have $0 = \rho_0 = \dots = \rho_{t-1}<\rho_{t} = \dots = \rho_p = 1$, for some index choice $t$.
Then we associate an interval $\underline{01}$ to the term $s(t)$ of our surjection $s$,
an interval of length one $\underline{0}$ to the terms $s(x)$ such that $x<t$,
and an interval of length one $\underline{1}$ to the terms $s(x)$ such that $x>t$.
In this context, the only possibility to get a tensor product of non degenerate simplices is to assume that $s(t)$
lies the last row of our table.
Indeed, we can observe that no caesura $s(e_i)$ should be associated to an interval $\underline{1}$ or $\underline{01}$,
because in the case where such an interval $\underline{1}$ or $\underline{01}$
is labelled by the value of a caesura $s(e_i)$,
the next occurrence of this value should be associated to the interval $\underline{1}$,
producing a degeneracy $\underline{\cdots 11\cdots}$
when we perform the concatenation operation.
Now, if we assume that all caesuras $s(e_i)$ are associated to the interval $\underline{0}$,
then the term $s(t)$ associated to the interval $\underline{01}$
necessarily occurs after the last caesura,
and hence, necessarily lies on the last row of our table.

In the case of the table reduction of a simplex of permutations $\varpi = (s_0,\dots,s_l)$,
we get that the interval $\underline{01}$ is associated to a term $s_l(t') = s_l(t)'$
which we retain in the sequence of values of the permutation $s_l$
on the last row of our table reduction.
The interval of length one $\underline{1}$ can only be labelled by the value of terms $s_l(x)'$ that occur after $s_l(t)'$
on the last row of our table.
The values of the terms $s_l(x)$ such that $x<t'$ in the permutation $s_l(x)$
can not occur at such positions.
Therefore, the occurrences of these values in our table reduction can only be decorated by intervals of length one $\underline{0}$,
and produce either a vertex $\sigma_{(i)}^{\alpha} = \underline{0}$ or a degenerate element
when we perform our concatenation operation.

This analysis implies that the tensors $\sigma_{(1)}^{\alpha}\otimes\cdots\otimes\sigma_{(r)}^{\alpha}$,
which occur in the expansion of a coproduct $\varpi_*(\underline{01})\in\DGN^*(\Delta^1)^{\otimes r}$, $\varpi = (s_0,\dots,s_l)$,
satisfy the following repartition constraint (when no degeneracy occurs):
\begin{equation*}
\sigma_{(i)}^{\alpha} = \begin{cases} \text{$\underline{0}$}, & \text{for $i = s_l(1),\dots,s_l(t'-1)$}, \\
\text{$\underline{01}$}, & \text{for $i = s_l(t')$},
\end{cases}
\end{equation*}
where $s_l(t') = s_l(t)'$ is the term of the permutation $s_l(1),\dots,s_l(r)$ that we associate to the interval $\underline{01}$ in our interval decomposition process.
Note also that (in non degeneracy cases) we necessarily have
\begin{equation*}
\sigma_{(i)}^{\alpha} = \underline{01},\quad\begin{aligned}[t] & \text{for the values of the caesuras $i = s_*(e_*)'$}\\
& \text{in our table reduction of the simplex $\varpi = (s_0,\dots,s_l)$},
\end{aligned}
\end{equation*}
since the values of the caesuras are repeated in our table (and hence lead to simplices of positive dimension when we perform our interval concatenation).

\textit{The verification of the vanishing relation}.
We go back to the operations $\pi_*': \DGN_*(\Delta^1)\rightarrow\DGN_*(\Delta^1)^{\otimes r}$ and $\pi_*'': \DGN_*(\Delta^1)\rightarrow\DGN_*(\Delta^1)^{\otimes r}$
associated to the factors of a coproduct $\Delta(\pi) = \sum_{(\pi)}\pi'\otimes\pi''$
in the expression of our equation~(\ref{eqn:vanishing_nabla_product}).

Recall that, for an element $\pi = (w_0,\dots,w_n)\in\EOp(r)$, we have by definition $\Delta(\pi) = \sum_k(w_0,\dots,w_k)\otimes(w_k,\dots,w_n)$.
We fix a term of this coproduct $\pi'\otimes\pi'' = (w_0,\dots,w_k)\otimes(w_k,\dots,w_n)$
and interval decompositions that fulfill the conditions of the previous paragraph
for some table reductions of the simplices $\pi' = (w_0,\dots,w_k)$
and $\pi'' = (w_k,\dots,w_n)$.
We consider the associated tensors $\sigma_{(1)}'\otimes\dots\otimes\sigma_{(r)}'$ and $\sigma_{(1)}''\otimes\dots\otimes\sigma_{(r)}''$
in the expansion of $\pi'_*(\underline{01})$
and $\pi''_*(\underline{01})$. We consider the case $k<n$ first.

Let $i = w_k(t')$ be the term of the permutation $w_k$ to which we associate the interval $\underline{01}$ in the table reduction of $\pi' = (w_0,\dots,w_k)$,
so that we have $\sigma_{(i)}' = \underline{01}$ (in a non degeneracy case).
If this term $i = w_k(t')$ occurs on the first row of our table reduction of the simplex $\pi'' = (w_k,\dots,w_n)$,
then we have either $\sigma_{(i)}'' = \underline{0}$ (when $w_k(t')$ occurs before the caesura)
or $\sigma_{(i)}'' = \underline{01}$ (when the caesura is at $w_k(t')$).
In both cases, we have $\nabla_*(\sigma_{(i)}'\otimes\sigma_{(i)}'') = 0$ since $\nabla_*(\underline{01}\otimes\underline{01}) = \nabla_*(\underline{0}\otimes\underline{01}) = 0$
by definition of our map $\nabla_*$,
and therefore such a choice results in a zero term
in the expression $\nabla^*\sh(\pi'_*(\underline{01})\otimes\pi''_*(\underline{01}))$.
If, on the contrary, we take the caesura $j = w_k(r_k'')$ before the value $i = w_k(t')$
occurs in our table reduction of the simplex $\pi'' = (w_k,\dots,w_n)$,
then this means that the value $j = w_k(r_k'')$ occurs before $i = w_k(t')$
in the permutation $w_k$,
and in this case, we have by our previous analysis $\sigma_{(j)}' = \underline{0}$
when we form the coproduct $\pi'_*(\underline{01})$.
We then have $\nabla_*(\sigma_{(j)}'\otimes\sigma_{(j)}'') = 0$ (since $\nabla_*(\underline{0}\otimes\underline{01}) = 0$),
We still conclude that our choice results in a zero term in the expression $\nabla^*\sh(\pi'_*(\underline{01})\otimes\pi''_*(\underline{01}))$.

In the case $k = n\Rightarrow\deg(\pi'') = 0$, we just consider the value $j = w_n(t'')$ to which we assign an interval $\underline{01}$ in our construction
of the tensor $\pi''_*(\underline{01})$, and we argue similarly: if $t'<t''$, then we have $\sigma_{(i)}'' = \underline{0}$,
so that $\nabla_*(\sigma_{(i)}'\otimes\sigma_{(i)}'') = \nabla_*(\underline{01}\otimes\underline{0}) = 0$;
if $t'=t''$, then we have $\sigma_{(i)}' = \sigma_{(i)}'' = \underline{01}$
and $\nabla_*(\sigma_{(i)}'\otimes\sigma_{(i)}'') = \nabla_*(\underline{01}\otimes\underline{01}) = 0$;
if $t'>t''$, then we have $\sigma_{(j)}' = \underline{0}$
and $\nabla_*(\sigma_{(j)}'\otimes\sigma_{(j)}'') = \nabla_*(\underline{0}\otimes\underline{01}) = 0$.

In all cases, we conclude that our choices result in a zero term in $\nabla_*^{\otimes r}\sh(\pi'_*(\underline{01})\otimes\pi''_*(\underline{01}))$.
Hence, we do obtain the vanishing relation $\sum_{(\pi)}\nabla_*^{\otimes r}\sh(\pi'_*(\underline{01})\otimes\pi''_*(\underline{01})) = 0$,
and this result finishes the proof of our proposition.
\end{proof}

\end{appendix}

\bibliographystyle{plain}
\bibliography{E-infinity-cooperads}

\end{document}